\documentclass[11pt]{amsart}
\usepackage[top=1.5in, bottom=1.5in, left=1in, right=1in]{geometry}
\usepackage{todonotes}

\usepackage{tikz}
\usetikzlibrary{decorations.pathmorphing} % for wiggly lines

\newcommand{\gray}{white!50!gray}
\newcommand{\black}{black}

\newcounter{pick}
\newcounter{valley}
\newcounter{valleyleft}
\newcounter{height}
\newcounter{previousheight}
\newcounter{totalheight}
\newcounter{parentheight}
\newcounter{width}
\newcounter{widthzero}
\newcounter{relativedistance}
\newcounter{toremove}
\newcounter{firstremove}

\newcommand{\nuPath}[7]{
\begin{tikzpicture}[
    scale=#3,
    pathedge/.style={line width=1.3pt}
]

\setcounter{pick}{0}
\setcounter{valley}{0}
\setcounter{height}{0}
\setcounter{previousheight}{0}
\setcounter{totalheight}{-1}

\foreach \step in {#1}{
    \addtocounter{totalheight}{1}
}

\draw[\gray,dashed] (0,0) -- (0,\thetotalheight);

\foreach \step in {#1}{

    \addtocounter{valley}{\step}

    \draw[\gray,dashed] (0,\theheight) -- (\thevalley,\theheight);

    \ifthenelse{\thepick=\thevalley}{}{
        \pgfmathparse{\thepick+1}
        \let\respick\pgfmathresult
        \foreach \x in {\respick,...,\thevalley}{
            \draw[\gray,dashed] (\x,\theheight) -- (\x,\thetotalheight);
        }
    }

    \addtocounter{pick}{\step}
    \addtocounter{height}{1}
}

\setcounter{pick}{0}
\setcounter{valley}{0}
\setcounter{height}{0}

\foreach \step in {#2}{

    \addtocounter{valley}{\step}

    \draw[pathedge]
        (\thepick,\thepreviousheight)
        --
        (\thepick,\theheight)
        --
        (\thevalley,\theheight);

    \addtocounter{pick}{\step}
    \setcounter{previousheight}{\theheight}
    \addtocounter{height}{1}
}

\foreach \a/\b/\c/\d in {#4}{
    \ifthenelse{\equal{\c}{}}{}{
        \draw[\c,line width=.5pt] (\a,\b) circle(.25);
    }
    \fill[\d,line width=.5pt] (\a,\b) circle(.15);
}

\foreach \a/\b/\c/\d/\e in {#5}{
    \draw[\d] (\a,\b) node[scale=.9,anchor=\e]{\c};
}

#7

\end{tikzpicture}
}

\newcommand{\nuTree}[7]{
\begin{tikzpicture}[
    scale=#3,
    fillnode/.style={fill=black,draw=black,circle,scale=0.8},
    circlenode/.style={draw=black,circle,scale=1},
    treeedge/.style={line width=0.9pt,black}
]

\setcounter{pick}{0}
\setcounter{valley}{0}
\setcounter{height}{0}
\setcounter{previousheight}{0}
\setcounter{totalheight}{-1}

#7
        
\foreach \step in {#1}{
    \addtocounter{totalheight}{1}
}

\draw[\gray,dashed] (0,0) -- (0,\thetotalheight);

\foreach \step in {#1}{

    \addtocounter{valley}{\step}

    \draw[\gray,dashed] (0,\theheight) -- (\thevalley,\theheight);

    \ifthenelse{\thepick=\thevalley}{}{
        \pgfmathparse{\thepick+1}
        \let\respick\pgfmathresult
        \foreach \x in {\respick,...,\thevalley}{
            \draw[\gray,dashed] (\x,\theheight) -- (\x,\thetotalheight);
        }
    }

    \addtocounter{pick}{\step}
    \addtocounter{height}{1}
}

    \setcounter{height}{0}
    \setcounter{width}{0}

    \edef\mynu{{#1}}
    \edef\mymu{{#2}}

        \pgfmathsetmacro{\stepnu}{\mynu[0]}
        \pgfmathsetmacro{\stepmu}{\mymu[0]}
        \addtocounter{width}{\stepnu}

        \foreach \x in {0,...,\stepmu}{
            \node[fillnode,scale=#3] at (\thewidth-\x,\theheight) {};
        }
        
        \draw[treeedge] (\thewidth-\stepmu,\theheight) -- (\thewidth,\theheight);

        \ifthenelse{0=\thetotalheight}{}{
        \setcounter{parentheight}{\theheight}
        \setcounter{relativedistance}{0}
        \foreach \k in {0,...,\thetotalheight}{
            \addtocounter{parentheight}{1}
            
            \pgfmathsetmacro{\substepnuabove}{\mynu[\k+1]}
            \pgfmathsetmacro{\substepmuabove}{\mymu[\k+1]}
            \addtocounter{relativedistance}{\substepnuabove-\substepmuabove}
            
            \ifthenelse{\therelativedistance>0}{}{\breakforeach} 
        }            
        \draw[treeedge] (\thewidth-\stepmu,\theheight) -- (\thewidth-\stepmu,\theparentheight);
        }

        \addtocounter{height}{1}

    \foreach \j in {1,...,\thetotalheight}{

        \pgfmathsetmacro{\stepnu}{\mynu[\j]}
        \pgfmathsetmacro{\stepmu}{\mymu[\j]}
        \addtocounter{width}{\stepnu}

        \foreach \x in {0,...,\stepmu}{
            \setcounter{relativedistance}{0}
            \setcounter{toremove}{\x}
            \foreach \k in {\j,...,1}{
                \pgfmathsetmacro{\substepnu}{\mynu[\k]}
                \pgfmathsetmacro{\substepmu}{\mymu[\k]}
                \pgfmathsetmacro{\substepmubelow}{\mymu[\k-1]}
                
                \ifthenelse{\k=\j}
                    {\addtocounter{relativedistance}{\x-\substepnu}}
                    {\addtocounter{relativedistance}{\substepmu-\substepnu}}

                \ifthenelse{\therelativedistance<0}
                    {\breakforeach}
                    {\addtocounter{toremove}{\substepmubelow}}
            }

            \ifthenelse{\x=0}
            {\setcounter{firstremove}{\thetoremove}}
            {}
            
            \node[fillnode,scale=#3] at (\thewidth-\thetoremove,\theheight) {};
        }

        \draw[treeedge] (\thewidth-\thetoremove,\theheight) -- (\thewidth-\thefirstremove,\theheight);
        
        \ifthenelse{\j=\thetotalheight}{}{
        \setcounter{parentheight}{\theheight}
        \setcounter{relativedistance}{0}
        \foreach \k in {\j,...,\thetotalheight}{
            \addtocounter{parentheight}{1}

            \pgfmathsetmacro{\substepnuabove}{\mynu[\k+1]}
            \pgfmathsetmacro{\substepmuabove}{\mymu[\k+1]}
            \addtocounter{relativedistance}{\substepnuabove-\substepmuabove}

            \ifthenelse{\therelativedistance>0}{}{\breakforeach} 
        }            
        \draw[treeedge] (\thewidth-\thetoremove,\theheight) -- (\thewidth-\thetoremove,\theparentheight);
        }
        
        \addtocounter{height}{1}
    }

\foreach \a/\b/\c/\d in {#4}{
    \ifthenelse{\equal{\c}{}}{}{
        \draw[\c,line width=.5pt] (\a,\b) circle(.25);
    }
    \fill[\d,line width=.5pt] (\a,\b) circle(.15);
}

\foreach \a/\b/\c/\d/\e in {#5}{
    \draw[\d] (\a,\b) node[scale=.9,anchor=\e]{\c};
}

\end{tikzpicture}
}

\newcommand{\altnuTree}[8]{
\begin{tikzpicture}[
    scale=#4,
    fillnode/.style={fill=black,draw=black,circle,scale=0.8},
    circlenode/.style={draw=black,circle,scale=1},
    treeedge/.style={line width=0.9pt,black}
]

\setcounter{valleyleft}{0}
\setcounter{height}{0}
\setcounter{previousheight}{0}
\setcounter{totalheight}{-1}

#8
        
\foreach \step in {#1}{
    \addtocounter{totalheight}{1}
}

\edef\mynu{{#1}}
\edef\mydelta{{#2}}
\edef\mymu{{#3}}

    \foreach \j in {1,...,\thetotalheight}{    
        \pgfmathsetmacro{\stepnu}{\mynu[\j]}
        \pgfmathsetmacro{\stepdelta}{\mydelta[\j-1]}
        \addtocounter{valleyleft}{\stepnu-\stepdelta}
    }
    \setcounter{valley}{\thevalleyleft}
    \pgfmathsetmacro{\stepnuzero}{\mynu[0]}
    \addtocounter{valley}{\stepnuzero}
    \setcounter{widthzero}{\thevalley}

    \draw[\gray, dashed] (\thevalleyleft,\theheight) -- (\thevalley, \theheight);
    \foreach \i in  {\thevalleyleft,...,\thevalley}{
        \draw[\gray, dashed] (\i,\theheight) -- (\i,\theheight+1);
    }
    \addtocounter{height}{1}

    \foreach \j in {1,...,\thetotalheight}{
        \pgfmathsetmacro{\stepnu}{\mynu[\j]}
        \pgfmathsetmacro{\stepdelta}{\mydelta[\j-1]}
        \addtocounter{valleyleft}{\stepdelta-\stepnu}
        \addtocounter{valley}{\stepdelta}

        \ifthenelse{\j=\thetotalheight}{
            \draw[\gray, dashed] (\thevalleyleft,\theheight) -- (\thevalley, \theheight);
            }
            {
                \draw[\gray, dashed] (\thevalleyleft,\theheight) -- (\thevalley, \theheight);
                \foreach \i in  {\thevalleyleft,...,\thevalley}{
                    \draw[\gray, dashed] (\i,\theheight) -- (\i,\theheight+1);
                }
            }
            
        \addtocounter{height}{1}
    }

    \setcounter{height}{0}
    \setcounter{width}{0}

        \pgfmathsetmacro{\stepnu}{\mynu[0]}
        \pgfmathsetmacro{\stepmu}{\mymu[0]}
        \addtocounter{width}{\thewidthzero} %modified

        \foreach \x in {0,...,\stepmu}{
            \node[fillnode,scale=#4] at (\thewidth-\x,\theheight) {};
        }
        
        \draw[treeedge] (\thewidth-\stepmu,\theheight) -- (\thewidth,\theheight);

        \ifthenelse{0=\thetotalheight}{}{
        \setcounter{parentheight}{\theheight}
        \setcounter{relativedistance}{0}
        \foreach \k in {0,...,\thetotalheight}{
            \addtocounter{parentheight}{1}
            
            \pgfmathsetmacro{\substepnuabove}{\mydelta[\k]} %%modified nu
            \pgfmathsetmacro{\substepmuabove}{\mymu[\k+1]}
            \addtocounter{relativedistance}{\substepnuabove-\substepmuabove}
            
            \ifthenelse{\therelativedistance>0}{}{\breakforeach} 
        }            
        \draw[treeedge] (\thewidth-\stepmu,\theheight) -- (\thewidth-\stepmu,\theparentheight);
        }

        \addtocounter{height}{1}

    \foreach \j in {1,...,\thetotalheight}{

        \pgfmathsetmacro{\stepnu}{\mynu[\j]}
        \pgfmathsetmacro{\stepmu}{\mymu[\j]}
        \pgfmathsetmacro{\stepdelta}{\mydelta[\j-1]}
        \addtocounter{width}{\stepdelta} %% modified

        \foreach \x in {0,...,\stepmu}{
            \setcounter{relativedistance}{0}
            \setcounter{toremove}{\x}
            \foreach \k in {\j,...,1}{
                \pgfmathsetmacro{\substepnu}{\mydelta[\k-1]} %%modified nu
                \pgfmathsetmacro{\substepmu}{\mymu[\k]}
                \pgfmathsetmacro{\substepmubelow}{\mymu[\k-1]}
                
                \ifthenelse{\k=\j}
                    {\addtocounter{relativedistance}{\x-\substepnu}}
                    {\addtocounter{relativedistance}{\substepmu-\substepnu}}

                \ifthenelse{\therelativedistance<0}
                    {\breakforeach}
                    {\addtocounter{toremove}{\substepmubelow}}
            }

            \ifthenelse{\x=0}
            {\setcounter{firstremove}{\thetoremove}}
            {}
            
            \node[fillnode,scale=#4] at (\thewidth-\thetoremove,\theheight) {};
        }

        \draw[treeedge] (\thewidth-\thetoremove,\theheight) -- (\thewidth-\thefirstremove,\theheight);
        
        \ifthenelse{\j=\thetotalheight}{}{
        \setcounter{parentheight}{\theheight}
        \setcounter{relativedistance}{0}
        \foreach \k in {\j,...,\thetotalheight}{
            \addtocounter{parentheight}{1}

            \pgfmathsetmacro{\substepnuabove}{\mydelta[\k]} %%modified nu
            \pgfmathsetmacro{\substepmuabove}{\mymu[\k+1]}
            \addtocounter{relativedistance}{\substepnuabove-\substepmuabove}

            \ifthenelse{\therelativedistance>0}{}{\breakforeach} 
        }            
        \draw[treeedge] (\thewidth-\thetoremove,\theheight) -- (\thewidth-\thetoremove,\theparentheight);
        }
        
        \addtocounter{height}{1}
    }

\foreach \a/\b/\c/\d in {#5}{
    \ifthenelse{\equal{\c}{}}{}{
        \draw[\c,line width=.5pt] (\a,\b) circle(.25);
    }
    \fill[\d,line width=.5pt] (\a,\b) circle(.15);
}

\foreach \a/\b/\c/\d/\e in {#6}{
    \draw[\d] (\a,\b) node[scale=.9,anchor=\e]{\c};
}

\end{tikzpicture}
}

\newcommand{\altnuShape}[8]{
\begin{tikzpicture}[
    scale=#4
]

\setcounter{valleyleft}{0}
\setcounter{height}{0}
\setcounter{previousheight}{0}
\setcounter{totalheight}{-1}

#8
        
\foreach \step in {#1}{
    \addtocounter{totalheight}{1}
}

\edef\mynu{{#1}}
\edef\mydelta{{#2}}
\edef\mymu{{#3}}

    \foreach \j in {1,...,\thetotalheight}{   
        \pgfmathsetmacro{\stepnu}{\mynu[\j]}
        \pgfmathsetmacro{\stepdelta}{\mydelta[\j-1]}
        \addtocounter{valleyleft}{\stepnu-\stepdelta}
    }
    \setcounter{valley}{\thevalleyleft}
    \pgfmathsetmacro{\stepnuzero}{\mynu[0]}
    \addtocounter{valley}{\stepnuzero}
    \setcounter{widthzero}{\thevalley}

    \draw[\black, dashed] (\thevalleyleft,\theheight) -- (\thevalley, \theheight);
    \foreach \i in  {\thevalleyleft,...,\thevalley}{
        \draw[\black, dashed] (\i,\theheight) -- (\i,\theheight+1);
    }
    \addtocounter{height}{1}

    \foreach \j in {1,...,\thetotalheight}{
        \pgfmathsetmacro{\stepnu}{\mynu[\j]}
        \pgfmathsetmacro{\stepdelta}{\mydelta[\j-1]}
        \addtocounter{valleyleft}{\stepdelta-\stepnu}
        \addtocounter{valley}{\stepdelta}

        \ifthenelse{\j=\thetotalheight}{
            \draw[\black, dashed] (\thevalleyleft,\theheight) -- (\thevalley, \theheight);
            }
            {
                \draw[\black, dashed] (\thevalleyleft,\theheight) -- (\thevalley, \theheight);
                \foreach \i in  {\thevalleyleft,...,\thevalley}{
                    \draw[\black, dashed] (\i,\theheight) -- (\i,\theheight+1);
                }
            }
            
        \addtocounter{height}{1}
    }

\foreach \a/\b/\c/\d in {#5}{
    \ifthenelse{\equal{\c}{}}{}{
        \draw[\c,line width=.5pt] (\a,\b) circle(.25);
    }
    \fill[\d,line width=.5pt] (\a,\b) circle(.15);
}

\foreach \a/\b/\c/\d/\e in {#6}{
    \draw[\d] (\a,\b) node[scale=.9,anchor=\e]{\c};
}

\end{tikzpicture}
}

\newcommand{\altnuShapegray}[8]{
\begin{tikzpicture}[
    scale=#4
]

\setcounter{valleyleft}{0}
\setcounter{height}{0}
\setcounter{previousheight}{0}
\setcounter{totalheight}{-1}

#8
        
\foreach \step in {#1}{
    \addtocounter{totalheight}{1}
}

\edef\mynu{{#1}}
\edef\mydelta{{#2}}
\edef\mymu{{#3}}

    \foreach \j in {1,...,\thetotalheight}{   
        \pgfmathsetmacro{\stepnu}{\mynu[\j]}
        \pgfmathsetmacro{\stepdelta}{\mydelta[\j-1]}
        \addtocounter{valleyleft}{\stepnu-\stepdelta}
    }
    \setcounter{valley}{\thevalleyleft}
    \pgfmathsetmacro{\stepnuzero}{\mynu[0]}
    \addtocounter{valley}{\stepnuzero}
    \setcounter{widthzero}{\thevalley}

    \draw[\gray, dashed] (\thevalleyleft,\theheight) -- (\thevalley, \theheight);
    \foreach \i in  {\thevalleyleft,...,\thevalley}{
        \draw[\gray, dashed] (\i,\theheight) -- (\i,\theheight+1);
    }
    \addtocounter{height}{1}

    \foreach \j in {1,...,\thetotalheight}{
        \pgfmathsetmacro{\stepnu}{\mynu[\j]}
        \pgfmathsetmacro{\stepdelta}{\mydelta[\j-1]}
        \addtocounter{valleyleft}{\stepdelta-\stepnu}
        \addtocounter{valley}{\stepdelta}

        \ifthenelse{\j=\thetotalheight}{
            \draw[\gray, dashed] (\thevalleyleft,\theheight) -- (\thevalley, \theheight);
            }
            {
                \draw[\gray, dashed] (\thevalleyleft,\theheight) -- (\thevalley, \theheight);
                \foreach \i in  {\thevalleyleft,...,\thevalley}{
                    \draw[\gray, dashed] (\i,\theheight) -- (\i,\theheight+1);
                }
            }
            
        \addtocounter{height}{1}
    }

\foreach \a/\b/\c/\d in {#5}{
    \ifthenelse{\equal{\c}{}}{}{
        \draw[\c,line width=.5pt] (\a,\b) circle(.25);
    }
    \fill[\d,line width=.5pt] (\a,\b) circle(.15);
}

\foreach \a/\b/\c/\d/\e in {#6}{
    \draw[\d] (\a,\b) node[scale=.9,anchor=\e]{\c};
}

\end{tikzpicture}
}

\newcommand{\altnuShapegraybonus}[9]{
\begin{tikzpicture}[
    scale=#4
]

\setcounter{valleyleft}{0}
\setcounter{height}{0}
\setcounter{previousheight}{0}
\setcounter{totalheight}{-1}

#8
        
\foreach \step in {#1}{
    \addtocounter{totalheight}{1}
}

\edef\mynu{{#1}}
\edef\mydelta{{#2}}
\edef\mymu{{#3}}

    \foreach \j in {1,...,\thetotalheight}{   
        \pgfmathsetmacro{\stepnu}{\mynu[\j]}
        \pgfmathsetmacro{\stepdelta}{\mydelta[\j-1]}
        \addtocounter{valleyleft}{\stepnu-\stepdelta}
    }
    \setcounter{valley}{\thevalleyleft}
    \pgfmathsetmacro{\stepnuzero}{\mynu[0]}
    \addtocounter{valley}{\stepnuzero}
    \setcounter{widthzero}{\thevalley}

    \draw[\gray, dashed] (\thevalleyleft,\theheight) -- (\thevalley, \theheight);
    \foreach \i in  {\thevalleyleft,...,\thevalley}{
        \draw[\gray, dashed] (\i,\theheight) -- (\i,\theheight+1);
    }
    \addtocounter{height}{1}

    \foreach \j in {1,...,\thetotalheight}{
        \pgfmathsetmacro{\stepnu}{\mynu[\j]}
        \pgfmathsetmacro{\stepdelta}{\mydelta[\j-1]}
        \addtocounter{valleyleft}{\stepdelta-\stepnu}
        \addtocounter{valley}{\stepdelta}

        \ifthenelse{\j=\thetotalheight}{
            \draw[\gray, dashed] (\thevalleyleft,\theheight) -- (\thevalley, \theheight);
            }
            {
                \draw[\gray, dashed] (\thevalleyleft,\theheight) -- (\thevalley, \theheight);
                \foreach \i in  {\thevalleyleft,...,\thevalley}{
                    \draw[\gray, dashed] (\i,\theheight) -- (\i,\theheight+1);
                }
            }
            
        \addtocounter{height}{1}
    }

\foreach \a/\b/\c/\d in {#5}{
    \ifthenelse{\equal{\c}{}}{}{
        \draw[\c,line width=.5pt] (\a,\b) circle(.25);
    }
    \fill[\d,line width=.5pt] (\a,\b) circle(.15);
}

\foreach \a/\b/\c/\d/\e in {#6}{
    \draw[\d] (\a,\b) node[scale=.9,anchor=\e]{\c};
}

#9

\end{tikzpicture}
}

\usepackage{amssymb}
\usepackage{amsthm}
\usepackage{amsmath}
\usepackage{mathrsfs}
\usepackage{amsbsy}
\usepackage{bm}
\usepackage{hyperref}
\usepackage{array}
\usepackage{enumerate}
\usepackage{bbm}
\usepackage{comment}
\usepackage{mathtools}
\usepackage{tabu}
\usepackage{makecell} 
\usepackage{colortbl}
\usepackage{xcolor, cancel}
\usepackage{tikz}
\usepackage{tikz-cd}
\usetikzlibrary{positioning,arrows.meta,patterns.meta}

\definecolor{blue}{rgb}{0, 0.445, 0.695}
\definecolor{bluishgreen}{rgb}{0, 0.626, 0.456}
\definecolor{red}{rgb}{0.896, 0.395, 0}
\definecolor{purple}{rgb}{0.783, 0.464, 0.640}
\definecolor{orange}{rgb}{0.999, 0.706, 0.0}
\definecolor{gold}{rgb}{0.81, 0.71, 0.23}
\definecolor{olive}{RGB}{116,141,19}
\definecolor{green}{RGB}{108,208,48}
\definecolor{teal}{RGB}{47,77,62}
\definecolor{turquoise}{RGB}{86,235,211}
\definecolor{lightblue}{RGB}{150,178,153}
\definecolor{skyblue}{rgb}{0.359, 0.752, 0.973}
\definecolor{blue2}{RGB}{25,50,191}
\definecolor{indigo}{RGB}{142,128,251}
\definecolor{indigo2}{RGB}{114,32,246}
\definecolor{lightpurple}{RGB}{243,197,250}
\definecolor{purple2}{RGB}{105,66,131}
\definecolor{magenta}{RGB}{206,43,188}
\definecolor{brown}{rgb}{0.69, 0.4, 0.0}
\definecolor{LightBlue}{rgb}{0,0.8,1} 
\definecolor{NewBlue}{rgb}{0,0.3,0.8}
\definecolor{Turquoise}{rgb}{0,0.7,0.4}

\newcommand{\Tam}[1]{\operatorname{Tam}({#1})}
\newcommand{\altTam}[2]{\operatorname{Tam}_{#1}(#2)}
\newcommand{\Stan}[1]{\operatorname{Stan}({#1})}

\newcommand{\up}{^{\operatorname{up}}}
\newcommand{\down}{^{\operatorname{down}}}
\newcommand{\ver}{^{\operatorname{vert}}}

\DeclareMathOperator{\cat}{Cat}
\DeclareMathOperator{\row}{Row}
\DeclareMathOperator{\flush}{flush}
\DeclareMathOperator{\ddeg}{ddeg} 
\DeclareMathOperator{\udeg}{udeg}

\newcommand{\dfn}[1]{\textcolor{Turquoise}{\emph{#1}}}

\hypersetup{colorlinks=true, citecolor=LightBlue, linkcolor=NewBlue, urlcolor=NewBlue} 

\usepackage{hhline}
\allowdisplaybreaks
\usepackage[noadjust]{cite}

\usepackage{caption}
\usepackage[noabbrev,capitalise,nameinlink]{cleveref}
\crefname{conjecture}{Conjecture}{Conjectures}

\newtheorem{theorem}{Theorem}[section]
\newtheorem{proposition}[theorem]{Proposition}
\newtheorem{corollary}[theorem]{Corollary}
\newtheorem{conjecture}[theorem]{Conjecture}
\newtheorem{question}[theorem]{Question}

\newtheorem{lemma}[theorem]{Lemma}

\theoremstyle{definition}
\newtheorem{definition}[theorem]{Definition}
\newtheorem{remark}[theorem]{Remark}
\newtheorem{example}[theorem]{Example}

\begin{document}

\title[Invariance of rowmotion for variants of the Tamari lattice]{Invariance of rowmotion \\ for variants of the Tamari lattice} 

\author[Adenbaum]{Ben Adenbaum}
\address{}
\email{benadenbaummath@gmail.com} 

\author[Barnard]{Emily Barnard}
\address{Department of Mathematical Sciences, DePaul University, Chicago, IL 60614, USA}
\email{e.barnard@depaul.edu} 

\author[Ceballos]{Cesar Ceballos}
\address{Institute of Geometry, TU Graz, Graz, Austria}
\email{cesar.ceballos@tugraz.at} 

\author[Chenevi\`{e}re]{Cl\'{e}ment Chenevi\`{e}re}
\address{LIGM, Univ Gustave Eiffel, CNRS, ESIEE Paris, F-77454 Marne-la-Vall\'{e}e, France}
\email{clement.cheneviere@univ-eiffel.fr} 

\author[Defant]{Colin Defant}
\address{Department of Mathematics, Harvard University, Cambridge, MA 02138, USA}
\email{colindefant@gmail.com} 

\author[Hopkins]{Sam Hopkins}
\address{Department of Mathematics, Howard University, Washington, DC 20059, USA}
\email{samuelfhopkins@gmail.com} 

\author[M\"{u}ller]{Matthias M\"{u}ller}
\address{Institute of Geometry, TU Graz, Graz, Austria}
\email{matthias.mueller@tugraz.at} 

\author[Rubey]{Martin Rubey}
\address{Institut f\"{u}r Diskrete Mathematik und Geometrie, TU Wien, Wien, Austria}
\email{martin.rubey@tuwien.ac.at} 

\author[Striker]{Jessica Striker}
\address{Department of Mathematics, North Dakota State University, Fargo, ND 58105, USA}
\email{jessica.striker@ndsu.edu} 

\begin{abstract}
We show that the rowmotion operator from dynamical algebraic combinatorics behaves the same on all alt~$\nu$-Tamari lattices for a fixed lattice path~$\nu$. We use this invariance of rowmotion to establish cyclic sieving and homomesy results for rational Tamari lattices. We also conjecture that this invariance extends to the more general cross Tamari lattices. 
\end{abstract} 


\maketitle

\section{Introduction} \label{sec:intro}

\dfn{Rowmotion} is an invertible operator that was originally defined on the set $J(P)$ of order ideals of a finite poset $P$~\cite{brouwer1974period, cameron1995orbits, panyushev2009orbits, striker2012rowmotion}. By viewing $J(P)$ as a finite distributive lattice, the definition of rowmotion has more recently been extended to act on broader classes of finite lattices, such as semidistributive lattices, trim lattices, and beyond~\cite{barnard2019canonical, thomas2019rowmotion, thomas2019independence, defant2023semidistrim, defant2025echelonmotion}.

Rowmotion is one of the operators that has been most intensively studied from the perspective of dynamical algebraic combinatorics~\cite{roby2016dac, striker2017dac}. Two major themes in this field concern the orbit structures of operators acting on combinatorial sets (e.g., the \dfn{cyclic sieving phenomenon}~\cite{reiner2004cyclic, sagan2011cyclic}) and the behavior of statistics along orbits of these operators (e.g., \dfn{homomesy}~\cite{propp2015homomesy}). Much of the previous research on rowmotion has focused on special lattices for which rowmotion behaves particularly well~\cite{panyushev2009orbits,armstrong2013uniform, rush2013orbits}, meaning lattices for which rowmotion has a small, predictable order, exhibits cyclic sieving with respect to natural $q$-analogues, and exhibits homomesy for natural statistics. One common homomesic statistic is the down-degree statistic.

In this paper, rather than studying the behavior of rowmotion on one particular lattice, we compare the behavior of rowmotion among different lattices. More precisely, we show that rowmotion behaves \emph{the same} for several different lattices, where ``the same'' means that rowmotion has the same orbit structure and the same down-degree orbit averages for these lattices.
  
There has been some prior work showing that rowmotion behaves the same for different lattices. For example, in~\cite[Proposition~4.10]{hopkins2022minuscule} it was shown that if the posets $P$ and $Q$ have isomorphic comparability graphs, then the distributive lattices $J(P)$ and $J(Q)$ have the same rowmotion behavior. Similarly, in~\cite{dao2022rowmotion} it was shown that if $P$ is a rectangle poset and $Q$ is a trapezoid poset of the corresponding dimensions, then $J(P)$ and $J(Q)$ have the same rowmotion behavior.

Examples in the existing literature of rowmotion behaving the same for different lattices, where at least one of the lattices is not distributive, are rarer. But one such example is implicitly known in the context of Coxeter--Catalan combinatorics. Fix a positive integer $n$, and consider the lattice of Dyck paths of semilength $n$ ordered by nesting, which is a distributive lattice that is sometimes called the \dfn{Stanley lattice}. In~\cite{armstrong2013uniform}, it was shown that rowmotion acting on this lattice is in equivariant bijection with the Kreweras complement acting on the noncrossing set partitions of the set $[n]\coloneq\{1,\ldots,n\}$. On the other hand, consider the \dfn{Tamari lattice} of triangulations of an $(n+2)$-gon, which is not distributive, but is semidistributive. In~\cite{barnard2019canonical}, it was explained that rowmotion acting on the Tamari lattice is also in equivariant bijection with the Kreweras complement of noncrossing partitions of~$[n]$. Consequently, rowmotion behaves the same for the Stanley and Tamari lattices.

In this paper we dramatically generalize this example of the Stanley lattice and Tamari lattice having the same rowmotion behavior. Specifically, we study rowmotion on the \dfn{alt~$\nu$-Tamari lattices}~\cite{ceballos2024altnu}. We use $\row\colon L \to L$ to denote the rowmotion operator on a lattice~$L$, and $\ddeg\colon L \to \mathbb{N}$ to denote the down-degree statistic on $L$, i.e., $\ddeg(x)$ is the number of elements of~$L$ that $x$ covers. Our main result is the following. 

\begin{theorem} \label{thm:intro_main}
Fix a lattice path $\nu$, and let $L_1$ and $L_2$ be two alt~$\nu$-Tamari lattices. Then rowmotion behaves the same for $L_1$ and $L_2$. That is, there is a bijection $\Phi\colon L_1 \to L_2$ such that:
\begin{itemize}
\item $\Phi( \row (x) ) = \row (\Phi (x))$ for all $x \in L_1$;
\item $\sum_{y \in \{\row^j(x)\colon j\geq 0\}} \ddeg(y) = \sum_{y \in \{\row^j(x)\colon j\geq 0\}} \ddeg(\Phi(y))$ for all $x \in L_1$.
\end{itemize} 
\end{theorem}

In the case where $\nu$ is a staircase lattice path, \cref{thm:intro_main} recovers the invariance of rowmotion for the Stanley and Tamari lattices, and in fact interpolates between these. For example, for the staircase lattice path $\nu=ENENEN$, we depict all the rowmotion orbits for the four alt~$\nu$-Tamari lattices in \cref{fig:alt_tamaris_14}, with the classical Tamari lattice on the top and the Stanley lattice on bottom. Note that each lattice is drawn three times to make the three rowmotion orbits clear. We can check that rowmotion behaves the same for all four lattices, confirming \cref{thm:intro_main} in this case.

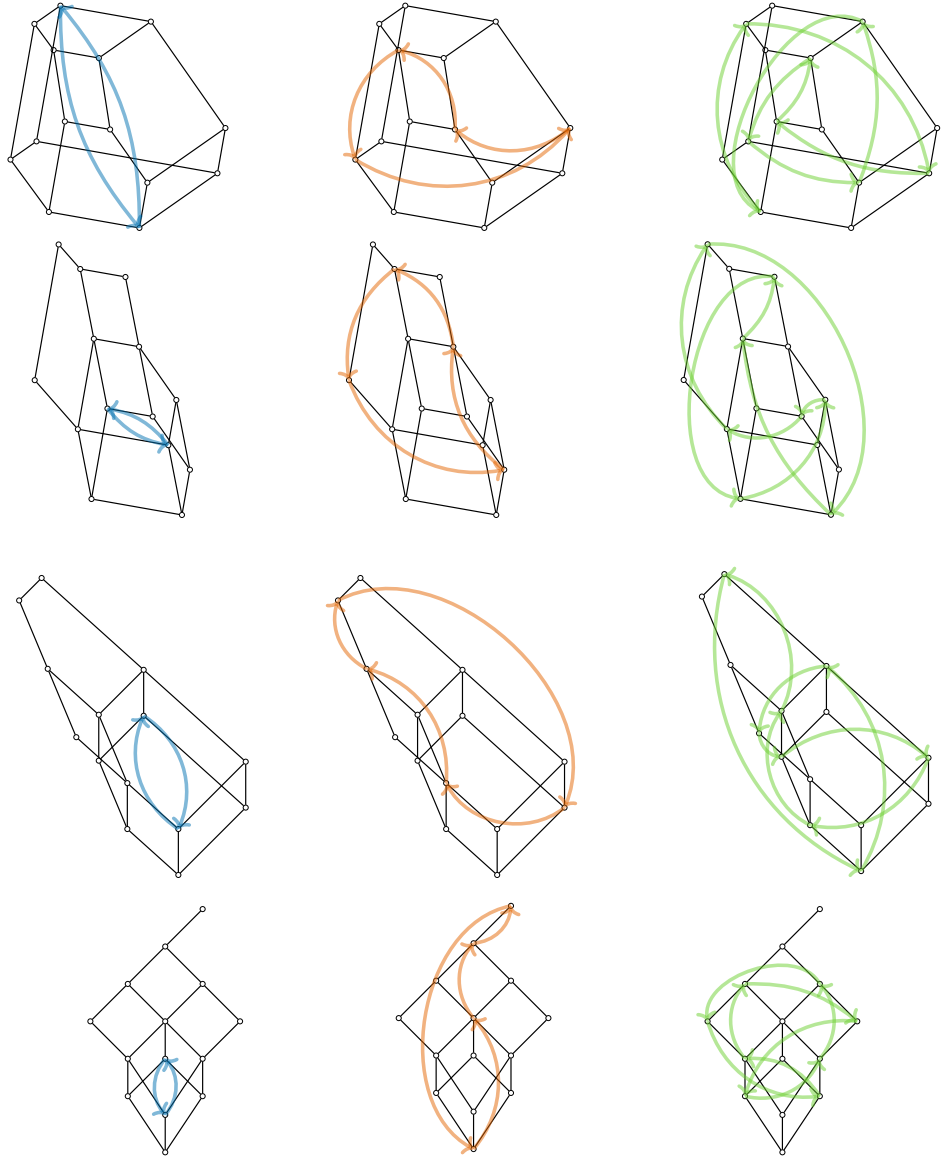
\begin{figure}[ht]
\begin{center}
\scalebox{0.45}{\begin{tikzpicture}%
	[z={(0cm, 4.5cm)},
	y={(-4.5cm, 0cm)},
	x={(2.2cm, 2.2cm)},
	scale=0.3,rotate=-10]

%% Coordinates of the vertices:
\coordinate (c0003) at (3,4,3);
\coordinate (c0012) at (2,4,3);
\coordinate (c0021) at (1,3,3);
\coordinate (c0030) at (0,2,2);
\coordinate (c0102) at (3,2,3);
\coordinate (c0111) at (1,2,3);
\coordinate (c0120) at (0,1,2);
\coordinate (c0201) at (3,0,1);
\coordinate (c0210) at (0,0,1);
\coordinate (c1002) at (3,4,0);
\coordinate (c1011) at (2,4,0);
\coordinate (c1020) at (0,2,0);
\coordinate (c1101) at (3,0,0);
\coordinate (c1110) at (0,0,0);

%% Drawing the vertices (added inner sep=1.5pt to make them smaller)
\node[draw,circle,fill=white,inner sep=1.5pt] (n0003) at (c0003) {};
\node[draw,circle,fill=white,inner sep=1.5pt] (n0012) at (c0012) {};
\node[draw,circle,fill=white,inner sep=1.5pt] (n0021) at (c0021) {};
\node[draw,circle,fill=white,inner sep=1.5pt] (n0030) at (c0030) {};
\node[draw,circle,fill=white,inner sep=1.5pt] (n0102) at (c0102) {};
\node[draw,circle,fill=white,inner sep=1.5pt] (n0111) at (c0111) {};
\node[draw,circle,fill=white,inner sep=1.5pt] (n0120) at (c0120) {};
\node[draw,circle,fill=white,inner sep=1.5pt] (n0201) at (c0201) {};
\node[draw,circle,fill=white,inner sep=1.5pt] (n0210) at (c0210) {};
\node[draw,circle,fill=white,inner sep=1.5pt] (n1002) at (c1002) {};
\node[draw,circle,fill=white,inner sep=1.5pt] (n1011) at (c1011) {};
\node[draw,circle,fill=white,inner sep=1.5pt] (n1020) at (c1020) {};
\node[draw,circle,fill=white,inner sep=1.5pt] (n1101) at (c1101) {};
\node[draw,circle,fill=white,inner sep=1.5pt] (n1110) at (c1110) {};

%% Drawing edges in the back
%% List of edges in the back
\def\listEdgesBack{1002/0003, 1011/1002, 1101/1002}
\foreach \x/\y in \listEdgesBack{
    \draw [color=black, line width=1] (n\x) -- (n\y);
}

%% Drawing edges in the front
%% List of edges in the front
\def\listEdgesFront{
0012/0003, 0021/0012, 0111/0021, 0111/0102, 0102/0003,
1110/1020, 1020/1011, 1110/1101,
1110/0210, 0210/0120, 0120/0030, 0030/0021, 0120/0111,
1020/0030,
1011/0012,
1101/0201, 0201/0102,
0210/0201}
\foreach \x/\y in \listEdgesFront{
    \draw [color=black, line width=1] (n\x) -- (n\y);
}

\draw[->, line width=3pt, blue, opacity=0.5] (c1110) to[bend right=20]  (c0003);
\draw[->, line width=3pt, blue, opacity=0.5] (c0003) to[bend right=20]  (c1110);

\end{tikzpicture} \hspace{3cm}
\begin{tikzpicture}%
	[z={(0cm, 4.5cm)},
	y={(-4.5cm, 0cm)},
	x={(2.2cm, 2.2cm)},
	scale=0.3,rotate=-10]

%% Coordinates of the vertices:
\coordinate (c0003) at (3,4,3);
\coordinate (c0012) at (2,4,3);
\coordinate (c0021) at (1,3,3);
\coordinate (c0030) at (0,2,2);
\coordinate (c0102) at (3,2,3);
\coordinate (c0111) at (1,2,3);
\coordinate (c0120) at (0,1,2);
\coordinate (c0201) at (3,0,1);
\coordinate (c0210) at (0,0,1);
\coordinate (c1002) at (3,4,0);
\coordinate (c1011) at (2,4,0);
\coordinate (c1020) at (0,2,0);
\coordinate (c1101) at (3,0,0);
\coordinate (c1110) at (0,0,0);

%% Drawing the vertices (added inner sep=1.5pt to make them smaller)
\node[draw,circle,fill=white,inner sep=1.5pt] (n0003) at (c0003) {};
\node[draw,circle,fill=white,inner sep=1.5pt] (n0012) at (c0012) {};
\node[draw,circle,fill=white,inner sep=1.5pt] (n0021) at (c0021) {};
\node[draw,circle,fill=white,inner sep=1.5pt] (n0030) at (c0030) {};
\node[draw,circle,fill=white,inner sep=1.5pt] (n0102) at (c0102) {};
\node[draw,circle,fill=white,inner sep=1.5pt] (n0111) at (c0111) {};
\node[draw,circle,fill=white,inner sep=1.5pt] (n0120) at (c0120) {};
\node[draw,circle,fill=white,inner sep=1.5pt] (n0201) at (c0201) {};
\node[draw,circle,fill=white,inner sep=1.5pt] (n0210) at (c0210) {};
\node[draw,circle,fill=white,inner sep=1.5pt] (n1002) at (c1002) {};
\node[draw,circle,fill=white,inner sep=1.5pt] (n1011) at (c1011) {};
\node[draw,circle,fill=white,inner sep=1.5pt] (n1020) at (c1020) {};
\node[draw,circle,fill=white,inner sep=1.5pt] (n1101) at (c1101) {};
\node[draw,circle,fill=white,inner sep=1.5pt] (n1110) at (c1110) {};

%% Drawing edges in the back
%% List of edges in the back
\def\listEdgesBack{1002/0003, 1011/1002, 1101/1002}
\foreach \x/\y in \listEdgesBack{
    \draw [color=black, line width=1] (n\x) -- (n\y);
}

%% Drawing edges in the front
%% List of edges in the front
\def\listEdgesFront{
0012/0003, 0021/0012, 0111/0021, 0111/0102, 0102/0003,
1110/1020, 1020/1011, 1110/1101,
1110/0210, 0210/0120, 0120/0030, 0030/0021, 0120/0111,
1020/0030,
1011/0012,
1101/0201, 0201/0102,
0210/0201}
\foreach \x/\y in \listEdgesFront{
    \draw [color=black, line width=1] (n\x) -- (n\y);
}

\draw[->, line width=3pt, red, opacity=0.5] (c0120) to[bend right=40]  (c0021);
\draw[->, line width=3pt, red, opacity=0.5] (c0021) to[bend right=40]  (c1011);
\draw[->, line width=3pt, red, opacity=0.5] (c1011) to[bend right=40]  (c0201);
\draw[->, line width=3pt, red, opacity=0.5] (c0201) to[bend left=40]  (c0120);

\end{tikzpicture} \hspace{3cm}
\begin{tikzpicture}%
	[z={(0cm, 4.5cm)},
	y={(-4.5cm, 0cm)},
	x={(2.2cm, 2.2cm)},
	scale=0.3,rotate=-10]

%% Coordinates of the vertices:
\coordinate (c0003) at (3,4,3);
\coordinate (c0012) at (2,4,3);
\coordinate (c0021) at (1,3,3);
\coordinate (c0030) at (0,2,2);
\coordinate (c0102) at (3,2,3);
\coordinate (c0111) at (1,2,3);
\coordinate (c0120) at (0,1,2);
\coordinate (c0201) at (3,0,1);
\coordinate (c0210) at (0,0,1);
\coordinate (c1002) at (3,4,0);
\coordinate (c1011) at (2,4,0);
\coordinate (c1020) at (0,2,0);
\coordinate (c1101) at (3,0,0);
\coordinate (c1110) at (0,0,0);

%% Drawing the vertices (added inner sep=1.5pt to make them smaller)
\node[draw,circle,fill=white,inner sep=1.5pt] (n0003) at (c0003) {};
\node[draw,circle,fill=white,inner sep=1.5pt] (n0012) at (c0012) {};
\node[draw,circle,fill=white,inner sep=1.5pt] (n0021) at (c0021) {};
\node[draw,circle,fill=white,inner sep=1.5pt] (n0030) at (c0030) {};
\node[draw,circle,fill=white,inner sep=1.5pt] (n0102) at (c0102) {};
\node[draw,circle,fill=white,inner sep=1.5pt] (n0111) at (c0111) {};
\node[draw,circle,fill=white,inner sep=1.5pt] (n0120) at (c0120) {};
\node[draw,circle,fill=white,inner sep=1.5pt] (n0201) at (c0201) {};
\node[draw,circle,fill=white,inner sep=1.5pt] (n0210) at (c0210) {};
\node[draw,circle,fill=white,inner sep=1.5pt] (n1002) at (c1002) {};
\node[draw,circle,fill=white,inner sep=1.5pt] (n1011) at (c1011) {};
\node[draw,circle,fill=white,inner sep=1.5pt] (n1020) at (c1020) {};
\node[draw,circle,fill=white,inner sep=1.5pt] (n1101) at (c1101) {};
\node[draw,circle,fill=white,inner sep=1.5pt] (n1110) at (c1110) {};

%% Drawing edges in the back
%% List of edges in the back
\def\listEdgesBack{1002/0003, 1011/1002, 1101/1002}
\foreach \x/\y in \listEdgesBack{
    \draw [color=black, line width=1] (n\x) -- (n\y);
}

%% Drawing edges in the front
%% List of edges in the front
\def\listEdgesFront{
0012/0003, 0021/0012, 0111/0021, 0111/0102, 0102/0003,
1110/1020, 1020/1011, 1110/1101,
1110/0210, 0210/0120, 0120/0030, 0030/0021, 0120/0111,
1020/0030,
1011/0012,
1101/0201, 0201/0102,
0210/0201}
\foreach \x/\y in \listEdgesFront{
    \draw [color=black, line width=1] (n\x) -- (n\y);
}

\draw[->, line width=3pt, green, opacity=0.5] (c1101) to[bend left=20]  (c0030);
\draw[->, line width=3pt, green, opacity=0.5] (c0030) to[bend right=20]  (c0111);
\draw[->, line width=3pt, green, opacity=0.5] (c0111) to[bend right=20]  (c1002);
\draw[->, line width=3pt, green, opacity=0.5] (c1002) to[bend right=20]  (c0210);
\draw[->, line width=3pt, green, opacity=0.5] (c0210) to[bend right=20]  (c0102);
\draw[->, line width=3pt, green, opacity=0.5] (c0102) to[bend right=90]  (c1020);
\draw[->, line width=3pt, green, opacity=0.5] (c1020) to[bend left=40]  (c0012);
\draw[->, line width=3pt, green, opacity=0.5] (c0012) to[bend left=40]  (c1101);

\end{tikzpicture}}
\vspace{3mm}
\scalebox{0.45}{\begin{tikzpicture}%
	[z={(0cm, 4.5cm)},
	y={(-4.5cm, 0cm)},
	x={(2cm, 2cm)},
	scale=0.3,rotate=-10]

%% Coordinates of the vertices:
\coordinate (c0003) at (3,6,6);
\coordinate (c0012) at (2,5,6);
\coordinate (c0021) at (1,4,5);
\coordinate (c0030) at (0,3,4);
\coordinate (c0102) at (2,4,6);
\coordinate (c0111) at (1,3,5);
\coordinate (c0120) at (0,2,4);
\coordinate (c0201) at (1,2,4);
\coordinate (c0210) at (0,1,3);
\coordinate (c1002) at (3,6,3);
\coordinate (c1011) at (1,4,3);
\coordinate (c1020) at (0,3,2);
\coordinate (c1101) at (1,2,3);
\coordinate (c1110) at (0,1,2);

%% Drawing the vertices
\node[draw,circle,fill=white,inner sep=1.5pt] (n0003) at (c0003) {};
\node[draw,circle,fill=white,inner sep=1.5pt] (n0012) at (c0012) {};
\node[draw,circle,fill=white,inner sep=1.5pt] (n0021) at (c0021) {};
\node[draw,circle,fill=white,inner sep=1.5pt] (n0030) at (c0030) {};

\node[draw,circle,fill=white,inner sep=1.5pt] (n0102) at (c0102) {};
\node[draw,circle,fill=white,inner sep=1.5pt] (n0111) at (c0111) {};
\node[draw,circle,fill=white,inner sep=1.5pt] (n0120) at (c0120) {};

\node[draw,circle,fill=white,inner sep=1.5pt] (n0201) at (c0201) {};
\node[draw,circle,fill=white,inner sep=1.5pt] (n0210) at (c0210) {};

\node[draw,circle,fill=white,inner sep=1.5pt] (n1002) at (c1002) {};
\node[draw,circle,fill=white,inner sep=1.5pt] (n1011) at (c1011) {};
\node[draw,circle,fill=white,inner sep=1.5pt] (n1020) at (c1020) {};

\node[draw,circle,fill=white,inner sep=1.5pt] (n1101) at (c1101) {};
\node[draw,circle,fill=white,inner sep=1.5pt] (n1110) at (c1110) {};

%% Drawing edges in the back
\def\listEdgesBack{1110/1101,1101/1011,1101/0201}
\foreach \x/\y in \listEdgesBack{
    \draw [color=black, line width=1] (n\x) -- (n\y);
}

%% Drawing edges in the front
\def\listEdgesFront{
1110/1020,1110/0210,1020/1011,1020/0030,1011/1002,1011/0021,1002/0003,0210/0201,0210/0120,0201/0111,0120/0111,0120/0030,0111/0102,0111/0021,0102/0012,0030/0021,0021/0012,0012/0003}
\foreach \x/\y in \listEdgesFront{
    \draw [color=black, line width=1] (n\x) -- (n\y);
}

\draw[->, line width=3pt, blue, opacity=0.5] (c0030) to[bend right=20]  (c1101);
\draw[->, line width=3pt, blue, opacity=0.5] (c1101) to[bend right=20]  (c0030);

\end{tikzpicture} \hspace{4cm}
\begin{tikzpicture}%
	[z={(0cm, 4.5cm)},
	y={(-4.5cm, 0cm)},
	x={(2cm, 2cm)},
	scale=0.3,rotate=-10]

%% Coordinates of the vertices:
\coordinate (c0003) at (3,6,6);
\coordinate (c0012) at (2,5,6);
\coordinate (c0021) at (1,4,5);
\coordinate (c0030) at (0,3,4);
\coordinate (c0102) at (2,4,6);
\coordinate (c0111) at (1,3,5);
\coordinate (c0120) at (0,2,4);
\coordinate (c0201) at (1,2,4);
\coordinate (c0210) at (0,1,3);
\coordinate (c1002) at (3,6,3);
\coordinate (c1011) at (1,4,3);
\coordinate (c1020) at (0,3,2);
\coordinate (c1101) at (1,2,3);
\coordinate (c1110) at (0,1,2);

%% Drawing the vertices
\node[draw,circle,fill=white,inner sep=1.5pt] (n0003) at (c0003) {};
\node[draw,circle,fill=white,inner sep=1.5pt] (n0012) at (c0012) {};
\node[draw,circle,fill=white,inner sep=1.5pt] (n0021) at (c0021) {};
\node[draw,circle,fill=white,inner sep=1.5pt] (n0030) at (c0030) {};

\node[draw,circle,fill=white,inner sep=1.5pt] (n0102) at (c0102) {};
\node[draw,circle,fill=white,inner sep=1.5pt] (n0111) at (c0111) {};
\node[draw,circle,fill=white,inner sep=1.5pt] (n0120) at (c0120) {};

\node[draw,circle,fill=white,inner sep=1.5pt] (n0201) at (c0201) {};
\node[draw,circle,fill=white,inner sep=1.5pt] (n0210) at (c0210) {};

\node[draw,circle,fill=white,inner sep=1.5pt] (n1002) at (c1002) {};
\node[draw,circle,fill=white,inner sep=1.5pt] (n1011) at (c1011) {};
\node[draw,circle,fill=white,inner sep=1.5pt] (n1020) at (c1020) {};

\node[draw,circle,fill=white,inner sep=1.5pt] (n1101) at (c1101) {};
\node[draw,circle,fill=white,inner sep=1.5pt] (n1110) at (c1110) {};

%% Drawing edges in the back
\def\listEdgesBack{1110/1101,1101/1011,1101/0201}
\foreach \x/\y in \listEdgesBack{
    \draw [color=black, line width=1] (n\x) -- (n\y);
}

%% Drawing edges in the front
\def\listEdgesFront{
1110/1020,1110/0210,1020/1011,1020/0030,1011/1002,1011/0021,1002/0003,0210/0201,0210/0120,0201/0111,0120/0111,0120/0030,0111/0102,0111/0021,0102/0012,0030/0021,0021/0012,0012/0003}
\foreach \x/\y in \listEdgesFront{
    \draw [color=black, line width=1] (n\x) -- (n\y);
}

\draw[->, line width=3pt, red, opacity=0.5] (c0210) to[bend left=30]  (c0111);
\draw[->, line width=3pt, red, opacity=0.5] (c0111) to[bend right=30]  (c0012);
\draw[->, line width=3pt, red, opacity=0.5] (c0012) to[bend right=30]  (c1002);
\draw[->, line width=3pt, red, opacity=0.5] (c1002) to[bend right=40]  (c0210);

\end{tikzpicture} \hspace{4cm}
\begin{tikzpicture}%
	[z={(0cm, 4.5cm)},
	y={(-4.5cm, 0cm)},
	x={(2cm, 2cm)},
	scale=0.3,rotate=-10]

%% Coordinates of the vertices:
\coordinate (c0003) at (3,6,6);
\coordinate (c0012) at (2,5,6);
\coordinate (c0021) at (1,4,5);
\coordinate (c0030) at (0,3,4);
\coordinate (c0102) at (2,4,6);
\coordinate (c0111) at (1,3,5);
\coordinate (c0120) at (0,2,4);
\coordinate (c0201) at (1,2,4);
\coordinate (c0210) at (0,1,3);
\coordinate (c1002) at (3,6,3);
\coordinate (c1011) at (1,4,3);
\coordinate (c1020) at (0,3,2);
\coordinate (c1101) at (1,2,3);
\coordinate (c1110) at (0,1,2);

%% Drawing the vertices
\node[draw,circle,fill=white,inner sep=1.5pt] (n0003) at (c0003) {};
\node[draw,circle,fill=white,inner sep=1.5pt] (n0012) at (c0012) {};
\node[draw,circle,fill=white,inner sep=1.5pt] (n0021) at (c0021) {};
\node[draw,circle,fill=white,inner sep=1.5pt] (n0030) at (c0030) {};

\node[draw,circle,fill=white,inner sep=1.5pt] (n0102) at (c0102) {};
\node[draw,circle,fill=white,inner sep=1.5pt] (n0111) at (c0111) {};
\node[draw,circle,fill=white,inner sep=1.5pt] (n0120) at (c0120) {};

\node[draw,circle,fill=white,inner sep=1.5pt] (n0201) at (c0201) {};
\node[draw,circle,fill=white,inner sep=1.5pt] (n0210) at (c0210) {};

\node[draw,circle,fill=white,inner sep=1.5pt] (n1002) at (c1002) {};
\node[draw,circle,fill=white,inner sep=1.5pt] (n1011) at (c1011) {};
\node[draw,circle,fill=white,inner sep=1.5pt] (n1020) at (c1020) {};

\node[draw,circle,fill=white,inner sep=1.5pt] (n1101) at (c1101) {};
\node[draw,circle,fill=white,inner sep=1.5pt] (n1110) at (c1110) {};

%% Drawing edges in the back
\def\listEdgesBack{1110/1101,1101/1011,1101/0201}
\foreach \x/\y in \listEdgesBack{
    \draw [color=black, line width=1] (n\x) -- (n\y);
}

%% Drawing edges in the front
\def\listEdgesFront{
1110/1020,1110/0210,1020/1011,1020/0030,1011/1002,1011/0021,1002/0003,0210/0201,0210/0120,0201/0111,0120/0111,0120/0030,0111/0102,0111/0021,0102/0012,0030/0021,0021/0012,0012/0003}
\foreach \x/\y in \listEdgesFront{
    \draw [color=black, line width=1] (n\x) -- (n\y);
}

\draw[->, line width=3pt, green, opacity=0.5] (c1110) to[bend left=20]  (c0021);
\draw[->, line width=3pt, green, opacity=0.5] (c0021) to[bend right=20]  (c0102);
\draw[->, line width=3pt, green, opacity=0.5] (c0102) to[bend right=90]  (c1020);
\draw[->, line width=3pt, green, opacity=0.5] (c1020) to[bend right=40]  (c0201);
\draw[->, line width=3pt, green, opacity=0.5] (c0201) to[bend right=40]  (c0120);
\draw[->, line width=3pt, green, opacity=0.5] (c0120) to[bend left=40]  (c1011);
\draw[->, line width=3pt, green, opacity=0.5] (c1011) to[bend left=40]  (c0003);
\draw[->, line width=3pt, green, opacity=0.5] (c0003) to[bend left=70]  (c1110);

\end{tikzpicture}}
\vspace{3mm}
\scalebox{0.45}{\begin{tikzpicture}%
	[z={(0cm, 4.5cm)},
	y={(-5cm, 0cm)},
	x={(2.2cm, 2.2cm)},
	scale=0.3]

%% Coordinates of the vertices:
\coordinate (c0003) at (3,4,6);
\coordinate (c0012) at (2,4,6);
\coordinate (c0021) at (1,3,5);
\coordinate (c0030) at (0,2,4);
\coordinate (c0102) at (3,2,4);
\coordinate (c0111) at (1,2,4);
\coordinate (c0120) at (0,1,3);
\coordinate (c0201) at (3,0,2);
\coordinate (c0210) at (0,0,2);
\coordinate (c1002) at (3,2,3);
\coordinate (c1011) at (1,2,3);
\coordinate (c1020) at (0,1,2);
\coordinate (c1101) at (3,0,1);
\coordinate (c1110) at (0,0,1);

%% Drawing the vertices
\node[draw,circle,fill=white,inner sep=1.5pt] (n0003) at (c0003) {};
\node[draw,circle,fill=white,inner sep=1.5pt] (n0012) at (c0012) {};
\node[draw,circle,fill=white,inner sep=1.5pt] (n0021) at (c0021) {};
\node[draw,circle,fill=white,inner sep=1.5pt] (n0030) at (c0030) {};

\node[draw,circle,fill=white,inner sep=1.5pt] (n0102) at (c0102) {};
\node[draw,circle,fill=white,inner sep=1.5pt] (n0111) at (c0111) {};
\node[draw,circle,fill=white,inner sep=1.5pt] (n0120) at (c0120) {};

\node[draw,circle,fill=white,inner sep=1.5pt] (n0201) at (c0201) {};
\node[draw,circle,fill=white,inner sep=1.5pt] (n0210) at (c0210) {};

\node[draw,circle,fill=white,inner sep=1.5pt] (n1002) at (c1002) {};
\node[draw,circle,fill=white,inner sep=1.5pt] (n1011) at (c1011) {};
\node[draw,circle,fill=white,inner sep=1.5pt] (n1020) at (c1020) {};

\node[draw,circle,fill=white,inner sep=1.5pt] (n1101) at (c1101) {};
\node[draw,circle,fill=white,inner sep=1.5pt] (n1110) at (c1110) {};

%% Drawing edges in the back
%% List of edges in the back
\def\listEdgesBack{1101/1002,1011/1002,1011/0111,1002/0102}
\foreach \x/\y in \listEdgesBack{
    \draw [color=black, line width=1] (n\x) -- (n\y);
}

%% Drawing edges in the front
%% List of edges in the front
\def\listEdgesFront{
1110/1101,1110/1020,1110/0210,1101/0201,1020/1011,1020/0120,0210/0201,0210/0120,0201/0102,0120/0111,0120/0030,0111/0102,0111/0021,0102/0003,0030/0021,0021/0012,0012/0003}
\foreach \x/\y in \listEdgesFront{
    \draw [color=black, line width=1] (n\x) -- (n\y);
}

\draw[->, line width=3pt, blue, opacity=0.5] (c0210) to[bend left=40]  (c1002);
\draw[->, line width=3pt, blue, opacity=0.5] (c1002) to[bend left=40]  (c0210);

\end{tikzpicture} \hspace{2cm}
\begin{tikzpicture}%
	[z={(0cm, 4.5cm)},
	y={(-5cm, 0cm)},
	x={(2.2cm, 2.2cm)},
	scale=0.3]

%% Coordinates of the vertices:
\coordinate (c0003) at (3,4,6);
\coordinate (c0012) at (2,4,6);
\coordinate (c0021) at (1,3,5);
\coordinate (c0030) at (0,2,4);
\coordinate (c0102) at (3,2,4);
\coordinate (c0111) at (1,2,4);
\coordinate (c0120) at (0,1,3);
\coordinate (c0201) at (3,0,2);
\coordinate (c0210) at (0,0,2);
\coordinate (c1002) at (3,2,3);
\coordinate (c1011) at (1,2,3);
\coordinate (c1020) at (0,1,2);
\coordinate (c1101) at (3,0,1);
\coordinate (c1110) at (0,0,1);

%% Drawing the vertices
\node[draw,circle,fill=white,inner sep=1.5pt] (n0003) at (c0003) {};
\node[draw,circle,fill=white,inner sep=1.5pt] (n0012) at (c0012) {};
\node[draw,circle,fill=white,inner sep=1.5pt] (n0021) at (c0021) {};
\node[draw,circle,fill=white,inner sep=1.5pt] (n0030) at (c0030) {};

\node[draw,circle,fill=white,inner sep=1.5pt] (n0102) at (c0102) {};
\node[draw,circle,fill=white,inner sep=1.5pt] (n0111) at (c0111) {};
\node[draw,circle,fill=white,inner sep=1.5pt] (n0120) at (c0120) {};

\node[draw,circle,fill=white,inner sep=1.5pt] (n0201) at (c0201) {};
\node[draw,circle,fill=white,inner sep=1.5pt] (n0210) at (c0210) {};

\node[draw,circle,fill=white,inner sep=1.5pt] (n1002) at (c1002) {};
\node[draw,circle,fill=white,inner sep=1.5pt] (n1011) at (c1011) {};
\node[draw,circle,fill=white,inner sep=1.5pt] (n1020) at (c1020) {};

\node[draw,circle,fill=white,inner sep=1.5pt] (n1101) at (c1101) {};
\node[draw,circle,fill=white,inner sep=1.5pt] (n1110) at (c1110) {};

%% Drawing edges in the back
%% List of edges in the back
\def\listEdgesBack{1101/1002,1011/1002,1011/0111,1002/0102}
\foreach \x/\y in \listEdgesBack{
    \draw [color=black, line width=1] (n\x) -- (n\y);
}

%% Drawing edges in the front
%% List of edges in the front
\def\listEdgesFront{
1110/1101,1110/1020,1110/0210,1101/0201,1020/1011,1020/0120,0210/0201,0210/0120,0201/0102,0120/0111,0120/0030,0111/0102,0111/0021,0102/0003,0030/0021,0021/0012,0012/0003}
\foreach \x/\y in \listEdgesFront{
    \draw [color=black, line width=1] (n\x) -- (n\y);
}

\draw[->, line width=3pt, red, opacity=0.5] (c0021) to[bend left=40]  (c0012);
\draw[->, line width=3pt, red, opacity=0.5] (c0012) to[bend left=70]  (c1101);
\draw[->, line width=3pt, red, opacity=0.5] (c1101) to[bend left=50]  (c0120);
\draw[->, line width=3pt, red, opacity=0.5] (c0120) to[bend right=40]  (c0021);

\end{tikzpicture} \hspace{2cm}
\begin{tikzpicture}%
	[z={(0cm, 4.5cm)},
	y={(-5cm, 0cm)},
	x={(2.2cm, 2.2cm)},
	scale=0.3]

%% Coordinates of the vertices:
\coordinate (c0003) at (3,4,6);
\coordinate (c0012) at (2,4,6);
\coordinate (c0021) at (1,3,5);
\coordinate (c0030) at (0,2,4);
\coordinate (c0102) at (3,2,4);
\coordinate (c0111) at (1,2,4);
\coordinate (c0120) at (0,1,3);
\coordinate (c0201) at (3,0,2);
\coordinate (c0210) at (0,0,2);
\coordinate (c1002) at (3,2,3);
\coordinate (c1011) at (1,2,3);
\coordinate (c1020) at (0,1,2);
\coordinate (c1101) at (3,0,1);
\coordinate (c1110) at (0,0,1);

%% Drawing the vertices
\node[draw,circle,fill=white,inner sep=1.5pt] (n0003) at (c0003) {};
\node[draw,circle,fill=white,inner sep=1.5pt] (n0012) at (c0012) {};
\node[draw,circle,fill=white,inner sep=1.5pt] (n0021) at (c0021) {};
\node[draw,circle,fill=white,inner sep=1.5pt] (n0030) at (c0030) {};

\node[draw,circle,fill=white,inner sep=1.5pt] (n0102) at (c0102) {};
\node[draw,circle,fill=white,inner sep=1.5pt] (n0111) at (c0111) {};
\node[draw,circle,fill=white,inner sep=1.5pt] (n0120) at (c0120) {};

\node[draw,circle,fill=white,inner sep=1.5pt] (n0201) at (c0201) {};
\node[draw,circle,fill=white,inner sep=1.5pt] (n0210) at (c0210) {};

\node[draw,circle,fill=white,inner sep=1.5pt] (n1002) at (c1002) {};
\node[draw,circle,fill=white,inner sep=1.5pt] (n1011) at (c1011) {};
\node[draw,circle,fill=white,inner sep=1.5pt] (n1020) at (c1020) {};

\node[draw,circle,fill=white,inner sep=1.5pt] (n1101) at (c1101) {};
\node[draw,circle,fill=white,inner sep=1.5pt] (n1110) at (c1110) {};

%% Drawing edges in the back
%% List of edges in the back
\def\listEdgesBack{1101/1002,1011/1002,1011/0111,1002/0102}
\foreach \x/\y in \listEdgesBack{
    \draw [color=black, line width=1] (n\x) -- (n\y);
}

%% Drawing edges in the front
%% List of edges in the front
\def\listEdgesFront{
1110/1101,1110/1020,1110/0210,1101/0201,1020/1011,1020/0120,0210/0201,0210/0120,0201/0102,0120/0111,0120/0030,0111/0102,0111/0021,0102/0003,0030/0021,0021/0012,0012/0003}
\foreach \x/\y in \listEdgesFront{
    \draw [color=black, line width=1] (n\x) -- (n\y);
}

\draw[->, line width=3pt, green, opacity=0.5] (c1110) to[bend right=40]  (c0102);
\draw[->, line width=3pt, green, opacity=0.5] (c0102) to[bend right=40]  (c0030);
\draw[->, line width=3pt, green, opacity=0.5] (c0030) to[bend right=40]  (c1011);
\draw[->, line width=3pt, green, opacity=0.5] (c1011) to[bend left=40]  (c0201);
\draw[->, line width=3pt, green, opacity=0.5] (c0201) to[bend left=40]  (c1020);
\draw[->, line width=3pt, green, opacity=0.5] (c1020) to[bend left=50]  (c0111);
\draw[->, line width=3pt, green, opacity=0.5] (c0111) to[bend right=50]  (c0003);
\draw[->, line width=3pt, green, opacity=0.5] (c0003) to[bend right=40]  (c1110);

\end{tikzpicture}}
\vspace{3mm}
\scalebox{0.45}{\begin{tikzpicture}%
[
scale=0.55]

%% Coordinates of the vertices:

\coordinate (c1110) at (0,0);
\coordinate (c1020) at (0,2);
\coordinate (c0210) at (2,3);
\coordinate (c0120) at (2,5);
\coordinate (c1101) at (-2,3);
\coordinate (c1011) at (-2,5);
\coordinate (c0111) at (0,7);
\coordinate (c0201) at (0,5);
\coordinate (c0030) at (4,7);
\coordinate (c0021) at (2,9);
\coordinate (c0012) at (0,11);
\coordinate (c0003) at (2,13);
\coordinate (c0102) at (-2,9);
\coordinate (c1002) at (-4,7);

%% Drawing the vertices
\node[draw,circle,fill=white,inner sep=1.5pt] (n0003) at (c0003) {};
\node[draw,circle,fill=white,inner sep=1.5pt] (n0012) at (c0012) {};
\node[draw,circle,fill=white,inner sep=1.5pt] (n0021) at (c0021) {};
\node[draw,circle,fill=white,inner sep=1.5pt] (n0030) at (c0030) {};

\node[draw,circle,fill=white,inner sep=1.5pt] (n0102) at (c0102) {};
\node[draw,circle,fill=white,inner sep=1.5pt] (n0111) at (c0111) {};
\node[draw,circle,fill=white,inner sep=1.5pt] (n0120) at (c0120) {};

\node[draw,circle,fill=white,inner sep=1.5pt] (n0201) at (c0201) {};
\node[draw,circle,fill=white,inner sep=1.5pt] (n0210) at (c0210) {};

\node[draw,circle,fill=white,inner sep=1.5pt] (n1002) at (c1002) {};
\node[draw,circle,fill=white,inner sep=1.5pt] (n1011) at (c1011) {};
\node[draw,circle,fill=white,inner sep=1.5pt] (n1020) at (c1020) {};

\node[draw,circle,fill=white,inner sep=1.5pt] (n1101) at (c1101) {};
\node[draw,circle,fill=white,inner sep=1.5pt] (n1110) at (c1110) {};

%% Drawing edges in the back
%% List of edges in the back
\def\listEdgesBack{1101/1011,1101/0201,1011/1002,1011/0111,1002/0102,1020/1011}
\foreach \x/\y in \listEdgesBack{
    \draw [color=black, line width=1] (n\x) -- (n\y);
}

%% Drawing edges in the front
%% List of edges in the front
\def\listEdgesFront{
1110/1101,1110/1020,1110/0210,1020/0120,0210/0201,0210/0120,0201/0111,0120/0111,0120/0030,0111/0102,0111/0021,0102/0012,0030/0021,0021/0012,0012/0003}
\foreach \x/\y in \listEdgesFront{
    \draw [color=black, line width=1] (n\x) -- (n\y);
}

\draw[->, line width=3pt, blue, opacity=0.5] (c1020) to[bend right=40]  (c0201);
\draw[->, line width=3pt, blue, opacity=0.5] (c0201) to[bend right=40]  (c1020);

\end{tikzpicture} \hspace{4cm}
\begin{tikzpicture}%
[
scale=0.55]

%% Coordinates of the vertices:
%\coordinate (c0003) at (3,6,9);
%\coordinate (c0012) at (2,5,8);
%\coordinate (c0021) at (1,4,7);
%\coordinate (c0030) at (0,3,6);
%\coordinate (c0102) at (2,4,7);
%\coordinate (c0111) at (1,3,6);
%\coordinate (c0120) at (0,2,5);
%\coordinate (c0201) at (1,2,5);
%\coordinate (c0210) at (0,1,4);
%\coordinate (c1002) at (2,4,6);
%\coordinate (c1011) at (1,3,5);
%\coordinate (c1020) at (0,2,4);
%\coordinate (c1101) at (1,2,4);
%\coordinate (c1110) at (0,1,3);

\coordinate (c1110) at (0,0);
\coordinate (c1020) at (0,2);
\coordinate (c0210) at (2,3);
\coordinate (c0120) at (2,5);
\coordinate (c1101) at (-2,3);
\coordinate (c1011) at (-2,5);
\coordinate (c0111) at (0,7);
\coordinate (c0201) at (0,5);
\coordinate (c0030) at (4,7);
\coordinate (c0021) at (2,9);
\coordinate (c0012) at (0,11);
\coordinate (c0003) at (2,13);
\coordinate (c0102) at (-2,9);
\coordinate (c1002) at (-4,7);

%% Drawing the vertices
\node[draw,circle,fill=white,inner sep=1.5pt] (n0003) at (c0003) {};
\node[draw,circle,fill=white,inner sep=1.5pt] (n0012) at (c0012) {};
\node[draw,circle,fill=white,inner sep=1.5pt] (n0021) at (c0021) {};
\node[draw,circle,fill=white,inner sep=1.5pt] (n0030) at (c0030) {};

\node[draw,circle,fill=white,inner sep=1.5pt] (n0102) at (c0102) {};
\node[draw,circle,fill=white,inner sep=1.5pt] (n0111) at (c0111) {};
\node[draw,circle,fill=white,inner sep=1.5pt] (n0120) at (c0120) {};

\node[draw,circle,fill=white,inner sep=1.5pt] (n0201) at (c0201) {};
\node[draw,circle,fill=white,inner sep=1.5pt] (n0210) at (c0210) {};

\node[draw,circle,fill=white,inner sep=1.5pt] (n1002) at (c1002) {};
\node[draw,circle,fill=white,inner sep=1.5pt] (n1011) at (c1011) {};
\node[draw,circle,fill=white,inner sep=1.5pt] (n1020) at (c1020) {};

\node[draw,circle,fill=white,inner sep=1.5pt] (n1101) at (c1101) {};
\node[draw,circle,fill=white,inner sep=1.5pt] (n1110) at (c1110) {};

%% Drawing edges in the back
%% List of edges in the back
\def\listEdgesBack{1101/1011,1101/0201,1011/1002,1011/0111,1002/0102,1020/1011}
\foreach \x/\y in \listEdgesBack{
    \draw [color=black, line width=1] (n\x) -- (n\y);
}

%% Drawing edges in the front
%% List of edges in the front
\def\listEdgesFront{
1110/1101,1110/1020,1110/0210,1020/0120,0210/0201,0210/0120,0201/0111,0120/0111,0120/0030,0111/0102,0111/0021,0102/0012,0030/0021,0021/0012,0012/0003}
\foreach \x/\y in \listEdgesFront{
    \draw [color=black, line width=1] (n\x) -- (n\y);
}

\draw[->, line width=3pt, red, opacity=0.5] (c1110) to[bend right=40]  (c0111);
\draw[->, line width=3pt, red, opacity=0.5] (c0111) to[bend left=40]  (c0012);
\draw[->, line width=3pt, red, opacity=0.5] (c0012) to[bend right=40]  (c0003);
\draw[->, line width=3pt, red, opacity=0.5] (c0003) to[bend right=70]  (c1110);

\end{tikzpicture} \hspace{4cm}
\begin{tikzpicture}%
[
scale=0.55]

%% Coordinates of the vertices:
%\coordinate (c0003) at (3,6,9);
%\coordinate (c0012) at (2,5,8);
%\coordinate (c0021) at (1,4,7);
%\coordinate (c0030) at (0,3,6);
%\coordinate (c0102) at (2,4,7);
%\coordinate (c0111) at (1,3,6);
%\coordinate (c0120) at (0,2,5);
%\coordinate (c0201) at (1,2,5);
%\coordinate (c0210) at (0,1,4);
%\coordinate (c1002) at (2,4,6);
%\coordinate (c1011) at (1,3,5);
%\coordinate (c1020) at (0,2,4);
%\coordinate (c1101) at (1,2,4);
%\coordinate (c1110) at (0,1,3);

\coordinate (c1110) at (0,0);
\coordinate (c1020) at (0,2);
\coordinate (c0210) at (2,3);
\coordinate (c0120) at (2,5);
\coordinate (c1101) at (-2,3);
\coordinate (c1011) at (-2,5);
\coordinate (c0111) at (0,7);
\coordinate (c0201) at (0,5);
\coordinate (c0030) at (4,7);
\coordinate (c0021) at (2,9);
\coordinate (c0012) at (0,11);
\coordinate (c0003) at (2,13);
\coordinate (c0102) at (-2,9);
\coordinate (c1002) at (-4,7);

%% Drawing the vertices
\node[draw,circle,fill=white,inner sep=1.5pt] (n0003) at (c0003) {};
\node[draw,circle,fill=white,inner sep=1.5pt] (n0012) at (c0012) {};
\node[draw,circle,fill=white,inner sep=1.5pt] (n0021) at (c0021) {};
\node[draw,circle,fill=white,inner sep=1.5pt] (n0030) at (c0030) {};

\node[draw,circle,fill=white,inner sep=1.5pt] (n0102) at (c0102) {};
\node[draw,circle,fill=white,inner sep=1.5pt] (n0111) at (c0111) {};
\node[draw,circle,fill=white,inner sep=1.5pt] (n0120) at (c0120) {};

\node[draw,circle,fill=white,inner sep=1.5pt] (n0201) at (c0201) {};
\node[draw,circle,fill=white,inner sep=1.5pt] (n0210) at (c0210) {};

\node[draw,circle,fill=white,inner sep=1.5pt] (n1002) at (c1002) {};
\node[draw,circle,fill=white,inner sep=1.5pt] (n1011) at (c1011) {};
\node[draw,circle,fill=white,inner sep=1.5pt] (n1020) at (c1020) {};

\node[draw,circle,fill=white,inner sep=1.5pt] (n1101) at (c1101) {};
\node[draw,circle,fill=white,inner sep=1.5pt] (n1110) at (c1110) {};

%% Drawing edges in the back
%% List of edges in the back
\def\listEdgesBack{1101/1011,1101/0201,1011/1002,1011/0111,1002/0102,1020/1011}
\foreach \x/\y in \listEdgesBack{
    \draw [color=black, line width=1] (n\x) -- (n\y);
}

%% Drawing edges in the front
%% List of edges in the front
\def\listEdgesFront{
1110/1101,1110/1020,1110/0210,1020/0120,0210/0201,0210/0120,0201/0111,0120/0111,0120/0030,0111/0102,0111/0021,0102/0012,0030/0021,0021/0012,0012/0003}
\foreach \x/\y in \listEdgesFront{
    \draw [color=black, line width=1] (n\x) -- (n\y);
}

\draw[->, line width=3pt, green, opacity=0.5] (c1101) to[bend right=40]  (c0120);
\draw[->, line width=3pt, green, opacity=0.5] (c0120) to[bend right=40]  (c0021);
\draw[->, line width=3pt, green, opacity=0.5] (c0021) to[bend right=80]  (c1002);
\draw[->, line width=3pt, green, opacity=0.5] (c1002) to[bend right=40]  (c0210);
\draw[->, line width=3pt, green, opacity=0.5] (c0210) to[bend right=30]  (c1011);
\draw[->, line width=3pt, green, opacity=0.5] (c1011) to[bend left=40]  (c0102);
\draw[->, line width=3pt, green, opacity=0.5] (c0102) to[bend left=20]  (c0030);
\draw[->, line width=3pt, green, opacity=0.5] (c0030) to[bend right=40]  (c1101);

\end{tikzpicture} }
\end{center}
\caption{Rowmotion orbits on all alt~$\nu$-Tamari lattices for $\nu=ENENEN$.}
\label{fig:alt_tamaris_14}
\end{figure}

In \cref{conj:cross_tamari}, we conjecture  more generally that if $L_1$ and $L_2$ are \dfn{cross Tamari lattices}~\cite{vonbell2025framing} associated to two moon polyominoes related via permutation of rows and columns, then rowmotion behaves the same for~$L_1$ and $L_2$. This conjecture is a generalization of \cref{thm:intro_main}, because alt~$\nu$-Tamari lattices are exactly the cross Tamari lattices associated to stack polyominoes, a subset of the moon polyominoes -- in this case, only permutations of columns applied to partition shapes are allowed.

\Cref{thm:intro_main} states that two alt~$\nu$-Tamari lattices for the same lattice path $\nu$ have the same rowmotion orbit structure and the same rowmotion orbit averages for the down-degree statistic. Hence, using this theorem we are able to establish some cyclic sieving and homomesy results for rowmotion of the \dfn{rational Tamari lattices}, thereby settling conjectures from~\cite{defant2024tamari, thomas2019rowmotion}. We do this by transferring to the distributive lattice case, which is easier to analyze.

The paper is structured as follows. In \cref{sec:background}, we review background material on lattices and the rowmotion operation. We also briefly review the history of rowmotion. In \cref{sec:alt}, we define alt~$\nu$-Tamari lattices, discuss their basic properties, and describe the action of rowmotion on them. In \cref{sec:phi}, we prove \cref{thm:intro_main}. We first construct the bijection $\Phi$ between two alt~$\nu$-Tamari lattices, prove that it intertwines rowmotion, and explain how it interacts with the down-degree statistic. In \cref{sec:rational}, we use \cref{thm:intro_main} to establish cyclic sieving and homomesy results for the rational Tamari lattices. Finally, in \cref{sec:future}, we discuss future directions, including the conjectural generalization to cross Tamari lattices.

\begin{remark}
While we were in the final stages of writing this paper, we learned about the independent, simultaneous work of~\cite{eu2026rowmotion}, which also proves \cref{thm:intro_main} in some special cases, i.e., for the families of lattice paths $\nu$ corresponding to hook shapes and two row shapes.
\end{remark}

\subsection*{Acknowledgments}

This paper grew out of the 2025 BIRS workshop on lattice theory. We are especially grateful to Osamu Iyama and Nathan Williams for co-organizing the workshop (along with E. Barnard, C.~Ceballos and C.~Defant), and we thank the other participants too for stimulating conversations. SageMath~\cite{Sage} was an important computational aid for this research. We also thank Yi-Lin Lee for making us aware of~\cite{eu2026rowmotion}. C.~Ceballos and M.~M\"{u}ller were partially supported by the Austrian Science Fund Grants 10.55776/P33278 and 10.55776/I5788. C.~Chenevi\`{e}re was partially supported by the ANR-FWF International Cooperation Project PAGCAP, funded by the ANR Project ANR-21-CE48-0020 and the FWF Project I 5788, and hosted at the LISN, Paris-Saclay University. C.~Defant was supported by a Benjamin Peirce Fellowship at Harvard University. S.~Hopkins was partially supported by Simons Foundation Gift MPS-TSM-00007193. J.~Striker was partially supported by NSF Grants DMS-2247089 and DMS-2554137 and Simons Foundation Gift MP-TSM-00002802.

\subsection*{Statement on AI usage}

We did not use AI tools in our research.

\section{Lattices and rowmotion} \label{sec:background}

\subsection{Definitions} 

We follow standard notation and terminology for posets, as laid out for instance in~\cite[Chapter~3]{stanley2012ec1}. {\bf All posets in this paper are assumed to be finite.}

Recall that a \dfn{lattice} is a poset~$L$ such that any two elements~$x,y\in L$ have a \dfn{meet} (i.e., greatest lower bound), denoted $x\wedge y$, and a \dfn{join} (i.e., least upper bound), denoted $x\vee y$. An element $x$ of a lattice $L$ is \dfn{join-irreducible} (resp.~\dfn{meet-irreducible}) if $x\neq\hat 0$ (resp. $x\neq\hat 1$) and $x$
cannot be written as a nontrivial join (resp.~meet) of elements of $L$. Because of the finiteness assumption on~$L$, this is the same as saying that $x$ covers (resp.~is covered by) exactly one other element. We use~$\mathcal J_L$ and~$\mathcal M_L$ to denote the set of join-irreducible elements and the set of meet-irreducible elements of~$L$, respectively.

A lattice $L$ is \dfn{semidistributive} if for all $x,y,z\in L$, we have the following implications: 
\begin{align*}
x\vee y &= x\vee z \Longrightarrow x\vee(y\wedge z) = x\vee y \\ 
x\wedge y &= x\wedge z\Longrightarrow x\wedge(y\vee z) = x\wedge y.
\end{align*}
Suppose $L$ is a semidistributive lattice, and let $x\lessdot y$ be a cover relation in $L$. It follows from the definition of semidistributivity that $\{z\in L:z\vee x=y\}$ has a minimum element $j_{x,y}$ and that this element is join-irreducible. We view $j_{x,y}$ as a label assigned to the edge $x\lessdot y$ in the Hasse diagram of $L$. Thus, we obtain a \emph{canonical labeling} of the edges of $L$ with the elements of $\mathcal J_L$. For~$x\in L$, we let \[\mathcal D_L(x)=\{j_{w,x}:w\lessdot x\}\quad\text{and}\quad\mathcal U_L(x)=\{j_{x,w}:x\lessdot w\}\] be the sets of labels of edges going down from $x$ and going up from $x$, respectively. It is known that the maps $\mathcal D_L$ and $\mathcal U_L$ are injective on $L$. Moreover, the sets $\{\mathcal D_L(x):x\in L\}$ and $\{\mathcal U_L(x):x\in L\}$ are equal, and they form an abstract simplicial complex known as the \dfn{canonical join complex} of $L$. This implies that there is a unique bijection $\row\colon L\to L$ satisfying 
\[\mathcal U_L(x)=\mathcal D_L(\row(x))\] 
for all $x\in L$. This bijection is \dfn{rowmotion}.\footnote{Some articles define rowmotion to be the inverse of the map we have defined.} Note in particular that rowmotion restricts to a bijection, denoted $\kappa^{-1}$ in Barnard's paper~\cite{barnard2019canonical}, between the meet-irreducible elements~$\mathcal{M}_L$ and join-irreducible elements~$\mathcal{J}_L$. See, e.g.,~\cite{barnard2019canonical} for more details.

Recall that a lattice $L$ is \dfn{distributive} if the join operation distributes over the meet operation, and vice-versa. Of course, if $L$ is distributive, then it is semidistributive. Because of the finiteness assumption, a distributive lattice $L$ is isomorphic to the lattice $J(P)$ of \dfn{order ideals} (i.e., downwards-closed subsets) of a poset~$P$. In fact, this $P$ is the poset of join-irreducible elements of $L$; see, e.g.,~\cite[Theorem~3.4.1]{stanley2012ec1}. In this case, rowmotion has an alternative description as an operator on order ideals. Namely, for any subset~$X\subseteq P$, we can define $\Lambda(X) \coloneqq \{y\in P:y\leq x\text{ for some }x\in X\}$ to be the order ideal generated by $X$. Then for any order ideal $I \in J(P)$, we have $\row(I)=\Lambda(\min(P\setminus I))$, where $\min$ denotes the set of minimal elements. This is the classical definition of rowmotion; see, e.g.,~\cite{striker2012rowmotion}.

The \dfn{length} of a poset $P$, denoted $\ell(P)$, is the maximum length of a chain in $P$, where the length of the chain $x_0 < x_1 < \cdots < x_k$ is $k$. It is straightforward to show that if $L$ is a lattice, then $\ell(L)\leq\min\{|\mathcal J_L|,|\mathcal M_L|\}$. We say that $L$ is \dfn{extremal} if~$\ell(L)=|\mathcal J_L|=|\mathcal M_L|$. Note that all distributive lattices are both semidistributive and extremal. In fact, the lattices we work with in this paper will all be semidistributive and extremal, so the following proposition will be useful.

\begin{proposition} \label{prop:intervals}
Let $L$ be a semidistributive and extremal lattice. Then any interval in $L$ is also a semidistributive and extremal lattice.
\end{proposition}

\begin{proof}
All sublattices of a semidistributive lattice are semidistributive by the definition of semidistributivity in terms of meets and joins. In particular, any interval of a semidistributive lattice is also semidistributive. By \cite[Theorem~1.4]{thomas2019rowmotion}, a finite semidistributive lattice is extremal if and only if it is \dfn{trim}; moreover, \cite[Theorem~3.6]{thomas2019rowmotion} asserts that any interval of a finite trim lattice is again trim. 
\end{proof}

\begin{remark}
Note crucially however that if $L'$ is an interval of $L$, then $\row \colon L' \to L'$ is \emph{not} just the restriction of $\row \colon L \to L$; indeed, rowmotion on $L'$ is basically unrelated to rowmotion on $L$.
\end{remark}

When the lattice $L$ is both semidistributive and extremal, there is an alternative method for computing rowmotion due to Thomas and Williams~\cite{thomas2019rowmotion}, which we now explain.\footnote{Thomas and Williams in fact studied rowmotion on broader families of lattices (like \dfn{trim lattices}~\cite{thomas2019rowmotion}) and posets (like \dfn{independence posets}~\cite{thomas2019independence}), but we will not need that level of generality.} Let $L$ be a semidistributive lattice. For each join-irreducible element $j\in \mathcal J_L$ and each element $x\in L$, let 
\[\tau_j(x)=\begin{cases}
    y & \text{if } j_{x,y}=j\in\mathcal U_L(x)\text{ or }j_{y,x}=j\in\mathcal D_L(x); \\
    x & \text{otherwise}.
\end{cases}\]
In other words, to obtain $\tau_j(x)$ from $x$, we look for an edge in the Hasse diagram of $L$ incident to~$x$ and labeled by $j$. If such an edge exists, then it is unique and we walk along it; if no such edge exists, we stay at $x$. This defines for each $j \in \mathcal{J}_L$ an involution $\tau_j\colon L\to L$ called a \dfn{toggle}. 

\begin{theorem}[{Thomas and Williams~\cite{thomas2019rowmotion}}]
Let $L$ be a semidistributive and extremal lattice, and let $z_{N+1}\lessdot z_{N}\lessdot \cdots z_2 \lessdot z_1$ be a maximum-length chain of $L$, where $N=\ell(L)=|\mathcal J_L|=|\mathcal M_L|$. Then
\[\row=\tau_{j_N}\circ\cdots\circ\tau_{j_1},\]
where $j_i \coloneqq j_{z_{i+1},z_{i}}$ for $i=1,\ldots,N$.
\end{theorem}

Because this formula computes rowmotion in a step-by-step process as a composition of toggles, Thomas and Williams say that it computes rowmotion \emph{in slow motion}.

\begin{example}
Consider a distributive lattice $L$, which we identify with the lattice $J(P)$ of order ideals of a poset $P$. For each $p\in P$, let $\Lambda(p)=\{q\in P:q\leq p\}$ be the order ideal generated by~$p$. The map $p\mapsto \Lambda(p)$ is a bijection from $P$ to $\mathcal J_L$. If $u\lessdot v$ is a cover relation in $L$, then there is a unique $r\in P$ such that $v=u\sqcup\{r\}$, and the join-irreducible element $j_{u,v}$ is $\Lambda(r)$. Thus, the toggles~$\tau_p \coloneqq \tau_{\Lambda(p)}\colon L \to L$ for $p \in P$ are given by 
\[\tau_{p}(I)=\begin{cases}
    I\triangle\{p\} & \text{if } I\triangle \{p\} \in J(P); \\
    I & \text{otherwise}, 
\end{cases}\]
where $\triangle$ denotes symmetric difference. And the slow motion description of rowmotion on $L$ is~$\row = \tau_{p_N} \circ \cdots \circ \tau_{p_1}$, where $p_N, \ldots, p_1$ is any linear extension of $P$ \cite{cameron1995orbits}. This ``row-by-row'' way to compute rowmotion is where the name comes from; see~\cite{striker2012rowmotion}.
\end{example}

\begin{example}
\Cref{fig:rowmotion_example} shows a $7$-element lattice $L$ that is both semidistributive and extremal but is not distributive. There are four join-irreducible elements $j_1,j_2,j_3,j_4$, and we have labeled each Hasse diagram edge with its canonical join-irreducible label. The indexing of the join-irreducible elements comes from the maximum-length chain along the left side of the figure: the labels of the edges going up the chain are $j_4,j_3,j_2,j_1$ (in this order). The blue arrows show the action of rowmotion. For example, there is an arrow from the join-irreducible element $j_2$ to the top element $\hat 1$ because $\mathcal D_L(\hat 1)=\mathcal U_L(j_2)=\{j_1,j_4\}$. We could also show that $\row(j_2)=\hat 1$ by computing rowmotion in slow motion. Indeed,
\begin{align*}
\row(j_2)&=(\tau_{j_4}\circ\tau_{j_3}\circ\tau_{j_2}\circ\tau_{j_1})(j_2) \\
 &=(\tau_{j_4}\circ\tau_{j_3}\circ\tau_{j_2})(j_1) \\ 
 &=(\tau_{j_4}\circ\tau_{j_3})(j_1) \\ 
 &=\tau_{j_4}(j_1) =\hat 1. 
\end{align*}
\end{example} 

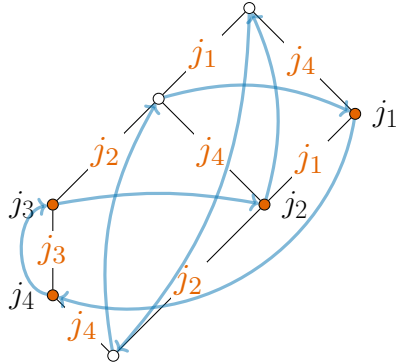
\begin{figure}[ht]
\begin{center}
\scalebox{0.4}{
\begin{tikzpicture}
    \node[draw,circle,fill=white,label={[font=\Huge,label distance=2mm]right:}] (n1) at (0,0) {};
    \draw (5,5)    node[scale=1,circle,draw,fill=red,anchor=center,label={[font=\Huge, scale=1.2,label distance=2mm]right:$j_2$}] (n2) {};
    %node[right=22pt, text=gray, scale=1.5] {$j_2$}
    \draw (8,8)    node[scale=1,circle,draw,fill=red,anchor=center,label={[font=\Huge, scale=1.2,label distance=2mm]right:$j_1$}] (n8) {};
\draw (-2,2)   node[circle,draw,fill=red,anchor=center,label={[font=\Huge, scale=1.2, label distance=2mm]left:$j_4$}] (n3) {};
    \draw (-2,5)   node[circle,draw,fill=red, anchor=center,label={[font=\Huge,scale=1.2, label distance=2mm]left:$j_3$}] (n4) {};
    \draw (1.5,8.5)node[scale=1,circle,draw,fill=white,anchor=center,label={[font=\Huge,label distance=2mm]right:}] (n5) {};
    \draw (3+1.5,8+3+.5) node[scale=1,circle,draw,fill=white,anchor=center,label={[font=\Huge,label distance=2mm]right:}] (n9) {};

    \draw [color=black, line width=1] (n5) -- node[midway, fill=white, text=red, inner sep=1pt, scale=3] {$j_1$} (n9);
    \draw [color=black, line width=1] (n9) -- node[midway, fill=white, text=red, inner sep=1pt, scale=3] {$j_4$} (n8);
    \draw [color=black, line width=1] (n8) -- node[midway, fill=white, text=red, inner sep=1pt, scale=3] {$j_1$} (n2);

    \draw [color=black, line width=1] (n1) -- node[midway, fill=white, text=red, inner sep=1pt, scale=3] {$j_2$} (n2);
    \draw [color=black, line width=1] (n2) -- node[midway, fill=white, text=red, inner sep=1pt, scale=3] {$j_4$} (n5);
    \draw [color=black, line width=1] (n5) -- node[midway, fill=white, text=red, inner sep=1pt, scale=3] {$j_2$} (n4);
    \draw [color=black, line width=1] (n4) -- node[midway, fill=white, text=red, inner sep=1pt, scale=3] {$j_3$} (n3);
    \draw [color=black, line width=1] (n3) -- node[midway, fill=white, text=red, inner sep=1pt, scale=3] {$j_4$} (n1);

\draw[->, line width=3pt, blue, opacity=0.5] (n1) to[bend left=20]  (n5);
\draw[->, line width=3pt, blue, opacity=0.5] (n5) to[bend left=20]  (n8);
\draw[->, line width=3pt, blue, opacity=0.5] (n8) to[bend left=50]  (n3);
\draw[->, line width=3pt, blue, opacity=0.5] (n3) to[bend left=80]  (n4);
\draw[->, line width=3pt, blue, opacity=0.5] (n4) to[bend left=10]  (n2);
\draw[->, line width=3pt, blue, opacity=0.5] (n2) to[bend right=20]  (n9);
\draw[->, line width=3pt, blue, opacity=0.5] (n9) to[bend left=20]  (n1);
    
\end{tikzpicture}
}
\end{center}
\caption{A semidistributive extremal lattice with edges labeled by join-irreducible elements. The arrows represent rowmotion.}
\label{fig:rowmotion_example}  
\end{figure} 

\begin{example} \label{ex:stanley_tamari}
Let $n \geq 1$. Recall from~\cref{sec:intro} that the \dfn{Stanley lattice}, $\Stan{n}$, is the lattice of Dyck paths of length $2n$ ordered by nesting. It is a distributive lattice on $\cat(n)$-many elements, where $\cat(n) \coloneqq \frac{1}{n+1}\binom{2n}{n}$ is the $n$th \dfn{Catalan number}. Meanwhile, the \dfn{Tamari lattice}, $\Tam{n}$, is another lattice on $\cat(n)$-many elements, whose cover relations can be described in terms of flips in triangulations or rotations of trees. For a precise definition of $\Tam{n}$, see \cref{sec:alt} below. It is known that $\Tam{n}$ is a semidistributive and extremal lattice, so it has a rowmotion operation defined on it, and in fact one that can be computed in slow motion. As mentioned in~\cref{sec:intro}, rowmotion has the same orbit structure on $\Stan{n}$ and $\Tam{n}$; for~$n=3$, this is depicted in~\cref{fig:5-element-lattices}. One way to relate rowmotion on $\Stan{n}$ and $\Tam{n}$ is via another operation on a third Catalan lattice. Namely, the \dfn{noncrossing partition lattice} is the collection of set partitions of~$\{1,2,\ldots,n\}$ whose blocks are noncrossing when the numbers $1,2,\ldots,n$ are arranged clockwise around a circle, partially ordered by refinement. It is known that this lattice is complemented, and one particular choice of complement, whose square is rotation, is the \dfn{Kreweras complement}. It is known that both rowmotion on the Stanley lattice~\cite{armstrong2013uniform}, and rowmotion on the Tamari lattice~\cite{barnard2019canonical}, have the same orbit structure as Kreweras complement of noncrossing partitions. This is also depicted, in the case $n=3$, in \cref{fig:5-element-lattices}.
\end{example}

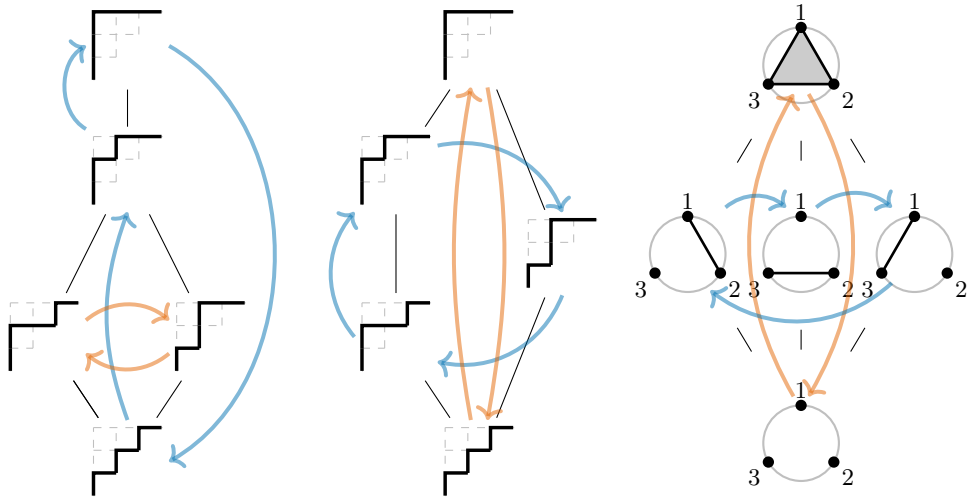
\begin{figure}[ht]
\begin{center}
\begin{tikzpicture}
	[scale=1.1,
	x={(1cm, 0cm)},
	y={(0cm, 1cm)},
	back/.style={loosely dotted, thin},
	edge/.style={color=black!95!black},
	facet/.style={fill=red!95!black,fill opacity=0.800000},
	vertex/.style={inner sep=1pt,circle,draw=green!25!black,fill=green!75!black,thick,anchor=base}]
	\scriptsize
	
	\coordinate (c120) at (0,0);
	\coordinate (c030) at (-1,1.5);
	\coordinate (c021) at (0,3.5);
	\coordinate (c012) at (0,5);
	\coordinate (c111) at (1,1.5);

	\node (n120) at (c120) {\nuPath{0,1,1,1}{0,1,1,1}{.3}{}{}{}{}};
	\node (n030) at (c030) {\nuPath{0,1,1,1}{0,0,2,1}{.3}{}{}{}{}};
	\node (n012) at (c012) {\nuPath{0,1,1,1}{0,0,0,3}{.3}{}{}{}{}};
    \node (n021) at (c021) {\nuPath{0,1,1,1}{0,0,1,2}{.3}{}{}{}{}};
	\node (n111) at (c111) {\nuPath{0,1,1,1}{0,1,0,2}{.3}{}{}{}{}};
    
	\draw[edge,-] (n120) -- (n030); 
	\draw[edge,-] (n030) -- (n021); 
	\draw[edge,-] (n021) -- (n012); 
	\draw[edge,-] (n120) -- (n111);
	\draw[edge,-] (n111) -- (n021);
	\draw[edge,-] (n120) -- (n030);

    \draw[->, line width=1.5pt, blue, opacity=0.5] (0,.5) to[bend left=20]  (0,3);
    \draw[->, line width=1.5pt, blue, opacity=0.5] (-.5,4) to[bend left=60]  (-.5,5);
    \draw[->, line width=1.5pt, blue, opacity=0.5] (.5,5) to[bend left=60]  (.5,0);
    \draw[->, line width=1.5pt, red, opacity=0.5] (-.5,1.7) to[bend left=40]  (.5,1.7);
    \draw[->, line width=1.5pt, red, opacity=0.5] (.5,1.3) to[bend left=40]  (-.5,1.3);
\end{tikzpicture}
\begin{tikzpicture}
	[scale=1.1,
	x={(1cm, 0cm)},
	y={(0cm, 1cm)},
	back/.style={loosely dotted, thin},
	edge/.style={color=black!95!black},
	facet/.style={fill=red!95!black,fill opacity=0.800000},
	vertex/.style={inner sep=1pt,circle,draw=green!25!black,fill=green!75!black,thick,anchor=base}]
	\scriptsize

	\coordinate (c120) at (0,0);
	\coordinate (c030) at (-1,1.5);
	\coordinate (c021) at (-1,3.5);
	\coordinate (c012) at (0,5);
	\coordinate (c111) at (1,2.5);

	\node (n120) at (c120) {\nuPath{0,1,1,1}{0,1,1,1}{.3}{}{}{}{}};
	\node (n030) at (c030) {\nuPath{0,1,1,1}{0,0,2,1}{.3}{}{}{}{}};
	\node (n012) at (c012) {\nuPath{0,1,1,1}{0,0,0,3}{.3}{}{}{}{}};
    \node (n021) at (c021) {\nuPath{0,1,1,1}{0,0,1,2}{.3}{}{}{}{}};
	\node (n111) at (c111) {\nuPath{0,1,1,1}{0,1,0,2}{.3}{}{}{}{}};

	\draw[edge,-] (n120) -- (n030); 
	\draw[edge,-] (n030) -- (n021); 
	\draw[edge,-] (n021) -- (n012); 
	\draw[edge,-] (n120) -- (n111);
	\draw[edge,-] (n111) -- (n012);

    \draw[->, line width=1.5pt, blue, opacity=0.5] (-1.5,1.5) to[bend left=40]  (-1.5,3);
    \draw[->, line width=1.5pt, blue, opacity=0.5] (-.5,3.8) to[bend left=40]  (1,3);
    \draw[->, line width=1.5pt, blue, opacity=0.5] (1,2) to[bend left=40] (-.5,1.2);
    
    \draw[->, line width=1.5pt, red, opacity=0.5] (-.1,.5) to[bend left=10]  (-.1,4.5);
    \draw[->, line width=1.5pt, red, opacity=0.5] (.1,4.5) to[bend left=10]  (.1,.5);

\end{tikzpicture}
\begin{tikzpicture}

  \newcommand{\ncpartition}[1]{%
    \begin{tikzpicture}[scale=0.5]

      \draw[gray!50, line width=0.8pt] (0,0) circle (1);
      \coordinate (n1) at (90:1);
      \coordinate (n2) at (-30:1);
      \coordinate (n3) at (-150:1);
      
      #1

      \node[circle, fill=black, inner sep=1.5pt, label={[font=\footnotesize, label distance=-3.5pt]above:1}] at (n1) {};
      \node[circle, fill=black, inner sep=1.5pt, label={[font=\footnotesize, label distance=-3.5pt]below right:2}] at (n2) {};
      \node[circle, fill=black, inner sep=1.5pt, label={[font=\footnotesize, label distance=-3.5pt]below left:3}] at (n3) {};
    \end{tikzpicture}%
  }

  \node (T) at (0, 5) {%
    \ncpartition{
      \draw[line width=1pt, fill=gray!40] (n1) -- (n2) -- (n3) -- cycle;
    }
  };

  \node (L) at (-1.5, 2.5) {%
    \ncpartition{
      \draw[line width=1pt] (n1) -- (n2);
    }
  };
  
  \node (M) at (0, 2.5) {%
    \ncpartition{
      \draw[line width=1pt] (n2) -- (n3);
    }
  };
  
  \node (R) at (1.5, 2.5) {%
    \ncpartition{
      \draw[line width=1pt] (n1) -- (n3);
    }
  };

  \node (B) at (0, 0) {
    \ncpartition{}
  };

  \draw[shorten <=5pt, shorten >=5pt] (B) -- (L);
  \draw[shorten <=5pt, shorten >=5pt] (B) -- (M);
  \draw[shorten <=5pt, shorten >=5pt] (B) -- (R);

  \draw[shorten <=5pt, shorten >=5pt] (L) -- (T);
  \draw[shorten <=5pt, shorten >=5pt] (M) -- (T);
  \draw[shorten <=5pt, shorten >=5pt] (R) -- (T);

    \draw[->, line width=1.5pt, red, opacity=0.5] (-.1,.5) to[bend left=30]  (-.1,4.5);
    \draw[->, line width=1.5pt, red, opacity=0.5] (.1,4.5) to[bend left=30]  (.1,.5);
    \draw[->, line width=1.5pt, blue, opacity=0.5] (-1,3) to[bend left=40]  (-.2,3);
    \draw[->, line width=1.5pt, blue, opacity=0.5] (.2,3) to[bend left=40]  (1.2,3);
    \draw[->, line width=1.5pt, blue, opacity=0.5] (1.2,2) to[bend left=40]  (-1.2,1.9);
\end{tikzpicture}
\end{center}
\caption{The Stanley (left), Tamari (middle), and noncrossing partition (right) lattices for $n=3$. The black edges are cover relations. In the Stanley and Tamari lattices, the blue and orange arrows represent rowmotion. In the noncrossing partition lattice, these arrows represent Kreweras complement.}
\label{fig:5-element-lattices}
\end{figure} 

Our goal in the remainder of this paper is to vastly generalize \cref{ex:stanley_tamari} by showing that not just the Stanley and Tamari lattices, but in fact a whole family of ``intermediary'' semidistributive extremal lattices (the alt Tamari lattices), have the same rowmotion behavior. However, we will do this not by comparing rowmotion on these lattices to the Kreweras complement. Rather, we will more directly relate rowmotion on all these lattices by defining an intertwining bijection.

\subsection{History}

Here we briefly review the history of rowmotion on distributive, and other, lattices. This provides context for our work, but nothing from this subsection is used in later sections.

Early work on rowmotion includes the study of periods of antichain operators by Brouwer and Schrijver \cite{brouwer1974period} and the further analysis of antichain orbits by Cameron and Fon-Der-Flaass \cite{cameron1995orbits}, especially for the antichains of the rectangle poset. See~\cite[\S7]{thomas2019rowmotion} for a more detailed history.

A particularly important appearance of rowmotion arose in the theory of root posets. Let $\Phi^+$ be the root poset of a finite Weyl group $W$.  Then the $W$-nonnesting partitions are the antichains of~$\Phi^+$.  These correspond to the order ideals of $\Phi^+$ via the map sending an antichain to the order ideal it generates. Under this identification, Panyushev studied an operator on nonnesting partitions that is the same as rowmotion~\cite{panyushev2009orbits}. Note also that $\Stan{n} = J(\Phi^+)$ when $\Phi=A_{n-1}$. Armstrong, Stump, and Thomas used this perspective to construct a uniform bijection from nonnesting partitions to noncrossing partitions that intertwines Panyushev's operator with the Kreweras complement~\cite{armstrong2013uniform}. Thus, rowmotion supplied the dynamical mechanism behind the first uniform bijection between the nonnesting and noncrossing families of Coxeter--Catalan objects.

The broader significance of distributive lattice rowmotion in dynamical algebraic combinatorics comes from the remarkable regularity it displays acting on the order ideals of many natural families of posets, like root and minuscule posets. In favorable cases, rowmotion has a small and predictable order, exhibits the cyclic sieving phenomenon, and exhibits natural homomesies~\cite{armstrong2013uniform,rush2013orbits,propp2015homomesy,striker2012rowmotion,hopkins2024order}. 

Barnard extended the definition of rowmotion to semidistributive lattices \cite{barnard2019canonical}, and Thomas and Williams extended the definition to trim lattices, which includes lattices that are both semidistributive and extremal \cite{thomas2019rowmotion}. Even more general definitions were formulated by Thomas and Williams for independence posets~\cite{thomas2019independence} and by Defant and Williams for semidistrim lattices~\cite{defant2023semidistrim}. 

Rowmotion, especially rowmotion beyond distributive lattices, also has natural connections with the representation theory of quivers and finite-dimensional algebras. Thomas and Williams~\cite{thomas2019rowmotion} showed that if $A$ is a representation-finite algebra such that $\operatorname{mod} A$ has no cycles, then the torsion pairs of~$A$, ordered by inclusion of torsion classes, form a trim lattice. Hence these torsion classes carry an action of rowmotion, which has been related to the Auslander--Reiten translation~\cite{barnard2021dynamical}. For path algebras of Dynkin quivers, this torsion class lattice construction recovers the simply-laced Cambrian lattices, including $\Tam{n}$ in Type A, placing the rowmotion/Kreweras complement story in a representation-theoretic context. Furthermore, Iyama and Marczinzik~\cite{iyama2022distributive} showed that the grade bijection (again, closely related to the Auslander--Reiten translation) acting on isomorphism classes of simple modules of the incidence algebra of a finite distributive lattice $L$ coincides with rowmotion on $L$.  This flavor of result was recently extended to lattices beyond the distributive case~\cite{klasz2025auslander, defant2025echelonmotion}.

\section{Alt \texorpdfstring{$\nu$}{nu}-Tamari lattices} \label{sec:alt} 

The Tamari lattice is a natural partial order on Catalan objects, classically defined on triangulations of a polygon by flipping diagonals or on binary trees by tree rotation. It can also be defined as a partial order on the set of Dyck paths of a fixed size, or equivalently, on the set of lattice paths living in a staircase partition shape.  In 2017, Pr\'eville-Ratelle and Viennot defined the \dfn{$\nu$-Tamari lattice} as an analogue of the Tamari lattice for paths in any partition shape~\cite{preville2017nutamari}. In~2024, Ceballos and Chenevi\`{e}re defined the \dfn{alt~$\nu$-Tamari lattices} as partial orders on the same sets of paths, but with different cover relations depending on an integer vector $\delta$~\cite{ceballos2024altnu}. Two extreme choices of $\delta$ yield the $\nu$-Tamari lattice and the distributive lattice of $\nu$-Dyck paths ordered by inclusion, which we will refer to as the \dfn{$\nu$-Dyck lattice}. Ceballos and Chenevi\`{e}re also provided an alternative construction of alt~$\nu$-Tamari lattices using special trees living in certain shapes called \dfn{stack polyominoes}. We first review both of these constructions, before describing how rowmotion can be described in terms of the tree construction.

\subsection{Definition on paths}

An integer partition can be graphically represented by its partition shape, also called a \dfn{Ferrers diagram} or \dfn{Young diagram}. Its border forms a lattice path $\nu$ from the southwest corner to the northeast corner of the shape, using unit east and north steps, denoted by~$E$ and $N$ respectively. A lattice path that has the same endpoints as $\nu$ and lies weakly above~$\nu$ is called a \dfn{$\nu$-Dyck path}. Equivalently, $\nu$-Dyck paths are all (maximal) lattice paths that can be drawn inside the partition shape.

Following~\cite{ceballos2024altnu}, if the bounding path~$\nu$ is of the form~$\nu = E^{\nu_0} N E^{\nu_1} N \cdots N E^{\nu_n}$, we encode it as the sequence of nonnegative integers $(\nu_0, \nu_1, \dots, \nu_n)$. Here, $n$ is the total number of north steps in~$\nu$, $\nu_0$ is the number of initial east steps, and each $\nu_i$ denotes the number of east steps immediately following the $i$-th north step of $\nu$. Equivalently, $\nu_i$ is the number of columns of height~$n-i$ in the corresponding partition shape. For example, as illustrated in~\cref{fig:alt_v_path}, if we consider the partition shape associated with the partition $\lambda = (6, 6, 5, 1)$, then $\nu$ is the path $ENEEEENENN$, encoded as the sequence $(1, 4, 1, 0, 0)$.

\begin{figure}[ht]    
\begin{center}
\input{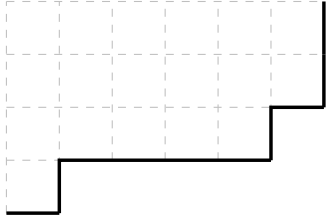}
\end{center}
\caption{The partition shape for $\lambda = (6,6,5,1)$.}
\label{fig:alt_v_path}
\end{figure}

For both the $\nu$-Tamari and the $\nu$-Dyck lattices, cover relations are defined by local \dfn{rotation} operations on $\nu$-Dyck paths, which exchange the east step of a valley with a portion of the path following it. Ceballos and Chenevi\`{e}re introduced \dfn{$\delta$-rotations}, where a parameter $\delta$ determines which portion of the path is exchanged with the east step; they obtained a family of lattices encompassing both the $\nu$-Tamari and $\nu$-Dyck lattices as extreme cases, which we now review~\cite{ceballos2024altnu}.

A sequence $\delta = (\delta_1, \dots, \delta_n)$ of nonnegative integers is an \dfn{increment vector} with respect to $\nu$ if~$\delta_i \leq \nu_i$ for all $1 \leq i \leq n$. Let $\mu$ be a $\nu$-Dyck path and $p$ a lattice point of $\mu$. We define the~\dfn{$\delta$-altitude} $\operatorname{alt}_\delta(p)$ as follows. The altitude of the starting lattice point of $\mu$ is set to zero. Then, an east step reduces the $\delta$-altitude by $1$, while the $i$-th north step of $\mu$ increases it by $\delta_i$.

For a valley $EN$ of $\mu$, let $p$ be the lattice point between the east and north steps. Let $q$ be the next lattice point on $\mu$ such that $\operatorname{alt}_\delta(q)=\operatorname{alt}_\delta(p)$, and let $\mu[p,q]$ be the subpath of $\mu$ starting at~$p$ and ending at $q$. Define $\mu'$ to be the path obtained from $\mu$ by exchanging $\mu[p,q]$ with the east step~$E$ immediately preceding it. The \dfn{$\delta$-rotation} of $\mu$ at the valley $p$ is then given by $\mu \lessdot_\delta \mu'$. An example is illustrated in~\cref{fig:delta_rotation_path}.

\begin{figure}[ht]
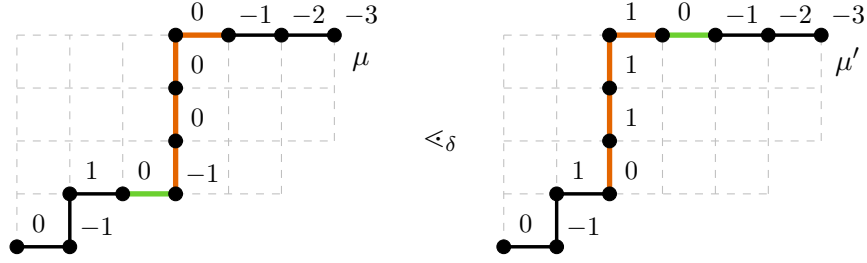
    
\begin{center}
\nuPath{1,4,1,0,0}{1,2,0,0,3}{.7}{}{}{}{\draw (8,2) node[scale=1, black] {$\lessdot_\delta$};

\draw [color=red, line width=2] (3,1)--(3,4)--(4,4);
\draw [color=green, line width=2] (2,1)--(3,1);

\draw (0,0) node[scale=0.5,circle,draw,fill=black,anchor=center,  label={[font=\small ,label distance=0mm]above right:$0$}] {};
\draw (1,0) node[scale=0.5,circle,draw,fill=black,anchor=center,  label={[font=\small ,label distance=-1mm]above right:$-1$}] {};
\draw (1,1) node[scale=0.5,circle,draw,fill=black,anchor=center,  label={[font=\small ,label distance=0mm]above right:$1$}] {};
\draw (2,1) node[scale=0.5,circle,draw,fill=black,anchor=center,  label={[font=\small ,label distance=0mm]above right:$0$}] {};
\draw (3,1) node[scale=0.5,circle,draw,fill=black,anchor=center,  label={[font=\small ,label distance=-1mm]above right:$-1$}] {};
\draw (3,2) node[scale=0.5,circle,draw,fill=black,anchor=center,  label={[font=\small ,label distance=0mm]above right:$0$}] {};
\draw (3,3) node[scale=0.5,circle,draw,fill=black,anchor=center,  label={[font=\small ,label distance=0mm]above right:$0$}] {};
\draw (3,4) node[scale=0.5,circle,draw,fill=black,anchor=center,  label={[font=\small ,label distance=0mm]above right:$0$}] {};
\draw (4,4) node[scale=0.5,circle,draw,fill=black,anchor=center,  label={[font=\small ,label distance=-1mm]above right:$-1$}] {};
\draw (5,4) node[scale=0.5,circle,draw,fill=black,anchor=center,  label={[font=\small ,label distance=-1mm]above right:$-2$}] {};
\draw (6,4) node[scale=0.5,circle,draw,fill=black,anchor=center,  label={[font=\small ,label distance=-1mm]above right:$-3$}] {};
\draw (6.5,3.5) node[scale=1, black] {$\mu$};

}
\nuPath{1,4,1,0,0}{1,1,0,0,4}{.7}{}{}{}{
\draw [color=red, line width=2] (2,1)--(2,4)--(3,4);
\draw [color=green, line width=2] (3,4)--(4,4);

\draw (0,0) node[scale=0.5,circle,draw,fill=black,anchor=center,  label={[font=\small,label distance=0mm]above right:$0$}] {};
\draw (1,0) node[scale=0.5,circle,draw,fill=black,anchor=center,  label={[font=\small,label distance=-1mm]above right:$-1$}] {};
\draw (1,1) node[scale=0.5,circle,draw,fill=black,anchor=center,  label={[font=\small,label distance=0mm]above right:$1$}] {};
\draw (2,1) node[scale=0.5,circle,draw,fill=black,anchor=center,  label={[font=\small,label distance=0mm]above right:$0$}] {};
\draw (3,4) node[scale=0.5,circle,draw,fill=black,anchor=center,  label={[font=\small,label distance=0mm]above right:$0$}] {};
\draw (2,2) node[scale=0.5,circle,draw,fill=black,anchor=center,  label={[font=\small,label distance=0mm]above right:$1$}] {};
\draw (2,3) node[scale=0.5,circle,draw,fill=black,anchor=center,  label={[font=\small,label distance=0mm]above right:$1$}] {};
\draw (2,4) node[scale=0.5,circle,draw,fill=black,anchor=center,  label={[font=\small,label distance=0mm]above right:$1$}] {};
\draw (4,4) node[scale=0.5,circle,draw,fill=black,anchor=center,  label={[font=\small,label distance=-1mm]above right:$-1$}] {};
\draw (5,4) node[scale=0.5,circle,draw,fill=black,anchor=center,  label={[font=\small,label distance=-1mm]above right:$-2$}] {};
\draw (6,4) node[scale=0.5,circle,draw,fill=black,anchor=center,  label={[font=\small,label distance=-1mm]above right:$-3$}] {};

\draw (6.5,3.5) node[scale=1, black] {$\mu'$};
}
\end{center}
\caption{An example of a $\delta$-rotation for $\nu=(1,4,1,0,0)$ and $\delta=(2,1,0,0)$, where the green east step is exchanged with the orange subpath. The number next to each lattice point of the path corresponds to its $\delta$-altitude.}
\label{fig:delta_rotation_path}
\end{figure}

The reflexive transitive closure of $\delta$-rotation forms a semidistributive lattice on the set of $\nu$-Dyck paths called the \dfn{alt $\nu$-Tamari lattice} $\altTam{\nu}{\delta}$~\cite{ceballos2024altnu}. We remark that the cover relations in $\altTam{\nu}{\delta}$ are exactly all $\delta$-rotations. The three alt Tamari lattices for the path $\nu = ENEEN$ are depicted in~\cref{fig:altnu_lattices_ENEEN_paths}.

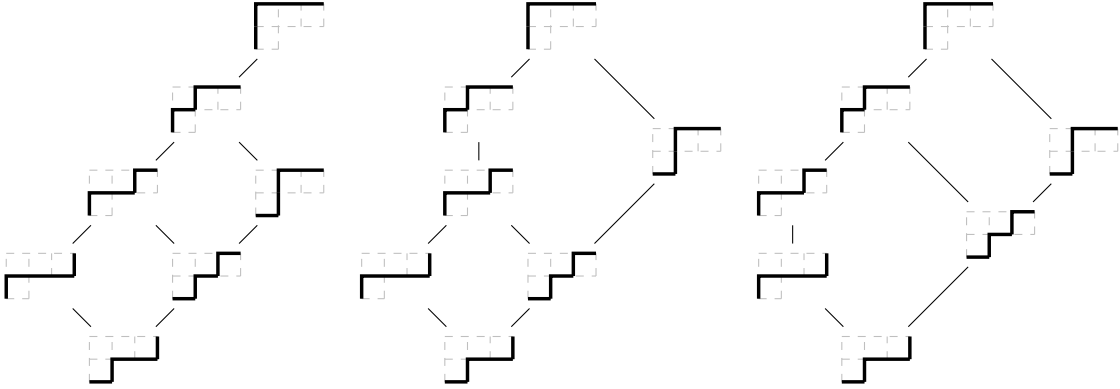
\begin{figure}[ht]
\begin{center}
\begin{tikzpicture}%
	[scale=1.1,
	x={(1cm, 0cm)},
	y={(0cm, 1cm)},
	back/.style={loosely dotted, thin},
	edge/.style={color=black},
	facet/.style={fill=red!95!black,fill opacity=0.800000},
	vertex/.style={inner sep=1pt,circle,draw=green!25!black,fill=green!75!black,thick,anchor=base}]
	\scriptsize
	
	\coordinate (c120) at (0,0);
	\coordinate (c030) at (-1,1);
	\coordinate (c021) at (0,2);
	\coordinate (c012) at (1,3);
	\coordinate (c003) at (2,4);
	
	\coordinate (c111) at (1,1);
	\coordinate (c102) at (2,2);
 
	\node (n120) at (c120) {\nuPath{1,2,0}{1,2,0}{.3}{}{}{}{}};
	% \node (n120) at (c120) {120};
	\node (n030) at (c030) {\nuPath{1,2,0}{0,3,0}{.3}{}{}{}{}};
	\node (n021) at (c021) {\nuPath{1,2,0}{0,2,1}{.3}{}{}{}{}};
	\node (n012) at (c012) {\nuPath{1,2,0}{0,1,2}{.3}{}{}{}{}};
	\node (n003) at (c003) {\nuPath{1,2,0}{0,0,3}{.3}{}{}{}{}};
	
	\node (n111) at (c111) {\nuPath{1,2,0}{1,1,1}{.3}{}{}{}{}};
	\node (n102) at (c102) {\nuPath{1,2,0}{1,0,2}{.3}{}{}{}{}};
	
	\draw[edge,-] (n120) -- (n030); 
	\draw[edge,-] (n030) -- (n021); 
	\draw[edge,-] (n021) -- (n012); 
	\draw[edge,-] (n012) -- (n003); 
	
	\draw[edge,-] (n120) -- (n111);
	\draw[edge,-] (n111) -- (n021);
	\draw[edge,-] (n111) -- (n102);
	
	\draw[edge,-] (n102) -- (n012);
\end{tikzpicture}
\begin{tikzpicture}%
	[scale=1.1,
	x={(1cm, 0cm)},
	y={(0cm, 1cm)},
	back/.style={loosely dotted, thin},
	edge/.style={color=black},
	facet/.style={fill=red!95!black,fill opacity=0.800000},
	vertex/.style={inner sep=1pt,circle,draw=green!25!black,fill=green!75!black,thick,anchor=base}]
	\scriptsize
	
	\coordinate (c120) at (0,0);
	\coordinate (c030) at (-1,1);
	\coordinate (c021) at (0,2);
	\coordinate (c012) at (0,3);
	\coordinate (c003) at (1,4);
	
	\coordinate (c111) at (1,1);
	\coordinate (c102) at (2.5,2.5);
	
	\node (n120) at (c120) {\nuPath{1,2,0}{1,2,0}{.3}{}{}{}{}};
	% \node (n120) at (c120) {120};
	\node (n030) at (c030) {\nuPath{1,2,0}{0,3,0}{.3}{}{}{}{}};
	\node (n021) at (c021) {\nuPath{1,2,0}{0,2,1}{.3}{}{}{}{}};
	\node (n012) at (c012) {\nuPath{1,2,0}{0,1,2}{.3}{}{}{}{}};
	\node (n003) at (c003) {\nuPath{1,2,0}{0,0,3}{.3}{}{}{}{}};
	
	\node (n111) at (c111) {\nuPath{1,2,0}{1,1,1}{.3}{}{}{}{}};
	\node (n102) at (c102) {\nuPath{1,2,0}{1,0,2}{.3}{}{}{}{}};
	
	\draw[edge,-] (n120) -- (n030); 
	\draw[edge,-] (n030) -- (n021); 
	\draw[edge,-] (n021) -- (n012); 
	\draw[edge,-] (n012) -- (n003); 
	
	\draw[edge,-] (n120) -- (n111);
	\draw[edge,-] (n111) -- (n102);
	\draw[edge,-] (n102) -- (n003);
	
	\draw[edge,-] (n111) -- (n021);
\end{tikzpicture}
\begin{tikzpicture}%
	[scale=1.1,
	x={(1cm, 0cm)},
	y={(0cm, 1cm)},
	back/.style={loosely dotted, thin},
	edge/.style={color=black},
	facet/.style={fill=red!95!black,fill opacity=0.800000},
	vertex/.style={inner sep=1pt,circle,draw=green!25!black,fill=green!75!black,thick,anchor=base}]
	\scriptsize
	
	\coordinate (c120) at (0,0);
	\coordinate (c030) at (-1,1);
	\coordinate (c021) at (-1,2);
	\coordinate (c012) at (0,3);
	\coordinate (c003) at (1,4);
	
	\coordinate (c111) at (1.5,1.5);
	\coordinate (c102) at (2.5,2.5);
	
	\node (n120) at (c120) {\nuPath{1,2,0}{1,2,0}{.3}{}{}{}{}};
	% \node (n120) at (c120) {120};
	\node (n030) at (c030) {\nuPath{1,2,0}{0,3,0}{.3}{}{}{}{}};
	\node (n021) at (c021) {\nuPath{1,2,0}{0,2,1}{.3}{}{}{}{}};
	\node (n012) at (c012) {\nuPath{1,2,0}{0,1,2}{.3}{}{}{}{}};
	\node (n003) at (c003) {\nuPath{1,2,0}{0,0,3}{.3}{}{}{}{}};
	
	\node (n111) at (c111) {\nuPath{1,2,0}{1,1,1}{.3}{}{}{}{}};
	\node (n102) at (c102) {\nuPath{1,2,0}{1,0,2}{.3}{}{}{}{}};
	
	\draw[edge,-] (n120) -- (n030); 
	\draw[edge,-] (n030) -- (n021); 
	\draw[edge,-] (n021) -- (n012); 
	\draw[edge,-] (n012) -- (n003); 
	
	\draw[edge,-] (n120) -- (n111);
	\draw[edge,-] (n111) -- (n102);
	\draw[edge,-] (n102) -- (n003);
	
	\draw[edge,-] (n111) -- (n012);
\end{tikzpicture}

	% %%draw text label
	% 	\foreach \a/\b/\c/\d in {#5}{
	% 		\draw[\d](\a,\b) node[scale=.9,anchor=west]{\c};
\end{center}
\caption{Examples of alt $\nu$-Tamari lattices $\altTam{\nu}{\delta}$ for $\nu=ENEEN=(1,2,0)$. Left: the $\nu$-Dyck lattice, for $\delta=(0,0)$. Middle: the lattice for $\delta=(1,0)$. Right: the~$\nu$-Tamari lattice, for $\delta=(2,0)$.}
\label{fig:altnu_lattices_ENEEN_paths}
\end{figure}

\begin{example}
The case of $\altTam{\nu}{\delta}$ where $\delta_i = \nu_i$ for all $1 \leq i \leq n$ recovers the $\nu$-Tamari lattice~$\Tam{\nu}$ of Pr\'eville-Ratelle and Viennot~\cite{preville2017nutamari}. (In turn, for appropriate choices of $\nu$, the $\nu$-Tamari lattices $\Tam{\nu}$ recover the \dfn{$m$-Tamari lattices}~\cite{bergeron2012higher} and \dfn{rational Tamari lattices}; see also~\cref{sec:rational} below.) On the other hand, setting $\delta_i = 0$ for all $i$, $\altTam{\nu}{\delta}$ becomes the distributive lattice of $\nu$-Dyck paths ordered by nesting, sometimes called the $\nu$-Dyck lattice. In particular, when~$\nu=(NE)^n$ is the staircase path, we get the classical Tamari and Stanley lattices, $\Tam{n}$ and $\Stan{n}$, respectively. Moreover, for this staircase $\nu$, the alt~$\nu$-Tamari lattices provide $2^{n-1}$ semidistributive lattices on Catalan many elements interpolating between the Tamari and Stanley lattices; see~\cref{fig:alt_tamaris_14}.
\end{example}

In order to know that rowmotion is defined on the alt~$\nu$-Tamari lattices, we need to observe that they are semidistributive. In fact, as mentioned in \cref{sec:background}, the lattices we work with in this paper will all be both semidistributive and extremal.

\begin{proposition} \label{prop:alt_semi_ext}
Every alt $\nu$-Tamari lattice is a semidistributive and extremal lattice.
\end{proposition}

\begin{proof}
It is well-known that the classical Tamari lattice $\Tam{n}$ is semidistributive and extremal; see, e.g.,~\cite{markowsky1992primes}. By \cref{prop:intervals}, it suffices then to note that every $\altTam{\nu}{\delta}$ is an interval in some~$\Tam{n}$, which is established in~\cite{ceballos2024altnu} and \cite{preville2017nutamari}.
\end{proof}

\subsection{Definition on trees} \label{sec:alt_nu_trees}

In this subsection, we describe another, equivalent way to think of the elements of alt $\nu$-Tamari lattices, as certain trees. This tree description will be more convenient for describing the action of rowmotion on $\altTam{\nu}{\delta}$, and so in subsequent sections we will most often view the elements of $\altTam{\nu}{\delta}$ as trees, not lattice paths.

A \dfn{polyomino} is a finite union of unit squares in $\mathbb{Z}^2$ that we call \dfn{cells} or \dfn{boxes}. The cells along a vertical line form a \dfn{column} and those along a horizontal line form a~\dfn{row}. A polyomino is \dfn{convex} if each of its rows and columns forms a contiguous block of cells, and it is~\dfn{L-convex} if any two cells can be connected by a path of cells with at most one change of direction, in which case we call it a \dfn{moon polyomino}. Equivalently, a moon polyomino is a convex polyomino whose rows (or columns) are nested. A \dfn{stack polyomino} is a moon polyomino whose largest row is the top row. In particular, a partition shape is a stack polyomino.  If $\nu$ is the bounding path of the partition, we denote the corresponding stack polyomino by $F_\nu$. Note that a polyomino is a stack polyomino if and only if it can be obtained from a partition shape $F_\nu$ by permuting its columns in such a way that the sequence of column lengths is weakly increasing then weakly decreasing. This choice of column permutation is equivalent to the choice of an increment vector $\delta$, where $\delta_i$ is the number of columns of height $n-i$ that remain to the right of the largest columns. We thus denote by $F_{\delta,\nu}$ the stack polyomino obtained from $F_\nu$ by permuting columns according to the increment vector $\delta$, as illustrated in~\cref{fig:altv_path}. Furthermore, let $L_{\delta,\nu}$ denote the set of lattice points inside $F_{\delta,\nu}$.

\begin{figure}[ht]
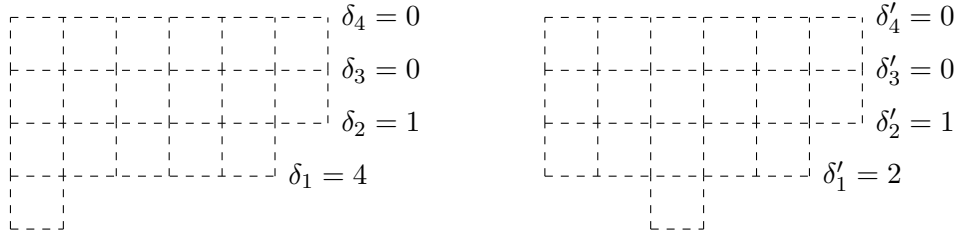

\begin{center}
%\nuPath{1,4,1,0,0}{1,4,1,0,0}{.7}{}{}{}{}
\altnuShape{1,4,1,0,0}{4,1,0,0}{0,0,0,0,0}{0.7}{}{}{}{
\draw (6,1) node[scale=1, black] {$\delta_1=4$};
\draw (7,2) node[scale=1, black] {$\delta_2=1$};
\draw (7,3) node[scale=1, black] {$\delta_3=0$};
\draw (7,4) node[scale=1, black] {$\delta_4=0$};
}
\hspace{1cm}
\altnuShape{1,4,1,0,0}{2,1,0,0}{0,0,0,0,0}{0.7}{}{}{}{

\draw (6,1) node[scale=1, black] {$\delta'_1=2$};
\draw (7,2) node[scale=1, black] {$\delta'_2=1$};
\draw (7,3) node[scale=1, black] {$\delta'_3=0$};
\draw (7,4) node[scale=1, black] {$\delta'_4=0$};

}

        %\nuPath{3,3,3,0}{3,3,3,0}{.4}{}{}{}{
        %\draw (10,3) node[scale=1, black] {$\nu$};
        %}
        %\nuPath{9,1,1,0}{9,1,1,0}{.4}{}{}{}{
        %\draw (12,3) node[scale=1, black] {$\check{\nu}$};
        %\draw (-1,3) node[scale=1, orange] {$\hat{\nu}$};
        %\draw [color=orange, line width=2, dashed] (0,3)--(2,3)--(2,2)--(4,2)--(4,1)--(6,1)--(6,0)--(9,0);
        %}
\end{center}
\caption{Left: The Ferrers diagram $F_\nu$ for $\nu=(1,4,1,0,0)$ is the stack polyomino~$F_{\delta,\nu}$ for $\delta=(4,1,0,0)$. Right: The stack polyomino $F_{\delta',\nu}$ for the same $\nu$ and~$\delta'=(2,1,0,0)$ is obtained from permuting the columns of $F_\nu$ according to $\delta'$.}
\label{fig:altv_path}
\end{figure}

Given a stack polyomino $F_{\delta,\nu}$, we can define a family of trees living inside $F_{\delta,\nu}$, which are in bijection with $\nu$-Dyck paths~\cite{ceballos2024altnu}. Two integer points $p$ and $q$ living inside $F_{\delta,\nu}$ are \dfn{$\delta$-incompatible} if one is strictly southwest of the other and the smallest rectangle containing $p$ and $q$ lies completely inside~$F_{\delta,\nu}$, and they are \dfn{$\delta$-compatible} otherwise. A \dfn{$(\delta, \nu)$-tree} is a maximal collection of pairwise~$\delta$-compatible points in~$L_{\delta,\nu}$. When connecting each point in a $(\delta,\nu)$-tree $T$ to the next point in $T$ in the north or west direction, if one exists, we obtain a (rooted binary) tree, hence the name. An example can be found on the right of~\cref{fig:right_flush}.

\begin{figure}[ht]
\begin{center}
\input{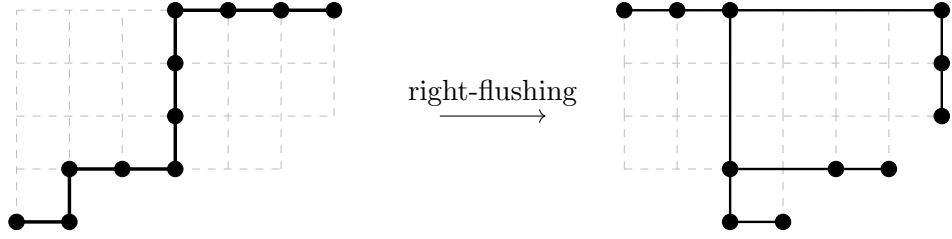}
\end{center}
\caption{Right-flushing of a $\nu$-Dyck path to a $(\delta,\nu)$-tree $T$ for $\nu=(1,4,1,0,0)$, $\delta=(2,1,0,0)$. Both have $2$, $3$, $1$, $1$, and $4$ vertices from bottom to top.}
\label{fig:right_flush}
\end{figure}

For each $\nu$-Dyck path $\mu$, there is a unique $(\delta,\nu)$-tree $T$ that contains the same number of lattice points in each row as $\mu$. The $(\delta, \nu)$-tree can be constructed by a greedy algorithm that inserts the points row by row, from bottom to top, in the rightmost available position that is compatible with all the previous points. In particular, each point that is not the leftmost in its row forbids all positions above it in its column. We give an example of the flushing map in~\cref{fig:right_flush}. This provides a bijection between $\nu$-Dyck paths and $(\delta,\nu)$-trees, called the \dfn{right-flushing} map, which was defined in~\cite{ceballos2024altnu}, extending the map defined in~\cite{ceballos2020nutamari} between $\nu$-Dyck paths and $\nu$-trees.

\begin{figure}[ht]
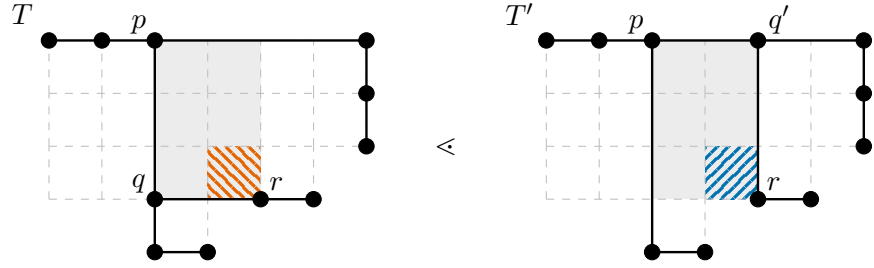

\begin{center}
\altnuTree{1,4,1,0,0}{2,1,0,0}{1,2,0,0,3}{0.7}{}{}{}{
\draw (-0.5,4.5) node[scale=1, black] {$T$};
\draw (1.7,4.3) node[scale=1, black] {$p$};
\draw (1.7,1.3) node[scale=1, black] {$q$};
\draw (4.3,1.3) node[scale=1, black] {$r$};
\fill[gray!15] (2,1) rectangle (4,4);
\fill[pattern={Lines[angle=-45, distance=3pt]}, pattern color=red, opacity=1, line width=0.4mm] (3,1) rectangle (4,2);

\draw (7.5,2) node[scale=1, black] {$\lessdot$};
}
\altnuTree{1,4,1,0,0}{2,1,0,0}{1,1,0,0,4}{0.7}{}{}{}{
\draw (-0.5,4.5) node[scale=1, black] {$T'$};
\draw (1.7,4.3) node[scale=1, black] {$p$};
\draw (4.4,4.4) node[scale=1, black] {$q'$};
\draw (4.3,1.3) node[scale=1, black] {$r$};
\fill[gray!15] (2,1) rectangle (4,4);
\fill[pattern={Lines[angle=45, distance=3pt]}, pattern color=blue, opacity=1, line width=0.4mm] (3,1) rectangle (4,2);
}

% SMALL EXAMPLE
%\altnuTree{1,1,1,0}{0,1,0}{0,1,2,0}{0.8}{}{}{}{
%\draw (3.5,3) node[scale=1, black] {$T$};
%\draw (0.7,2.3) node[scale=1, black] {$p$};
%\draw (0.7,1.3) node[scale=1, black] {$q$};
%\draw (2.3,1.3) node[scale=1, black] {$r$};

%\draw (4.5,1.5) node[scale=1, black] {$\lessdot$};
%}
%\altnuTree{1,1,1,0}{0,1,0}{0,0,3,0}{0.8}{}{}{}{
%\draw (3.5,3) node[scale=1, black] {$T'$};
%\draw (0.7,2.3) node[scale=1, black] {$p$};
%\draw (2.4,2.4) node[scale=1, black] {$q'$};
%\draw (2.3,1.3) node[scale=1, black] {$r$};

%}
\end{center}
\caption{A right rotation in a $(\delta,\nu)$-tree $T$ for $\nu=(1,4,1,0,0)$ and $\delta=(2,1,0,0)$. The box $b_{T, T'}$ labeling the cover relation $T \lessdot T'$ is striped orange or blue.}
\label{fig:delta_rotation_tree}
\end{figure}

Under this bijection, $\delta$-rotations of $\nu$-Dyck paths translate to right rotations of binary trees, thus providing an alternative description of the alt $\nu$-Tamari lattice $\altTam{\nu}{\delta}$ on the set of $(\delta,\nu)$-trees. Let $T$ be a $(\delta,\nu)$-tree and $q$ a node of $T$ connected to a node $p$ above it and a node $r$ to its right. Let $q'$ be the fourth (top right) corner of the smallest rectangle containing $p$, $q$, and $r$, and $T'$ be the tree obtained from $T$ by replacing $q$ with $q'$. We say that $T \lessdot T'$ is a \dfn{right rotation}, and the inverse operation is called a \dfn{left rotation}. In this case, we say that $q$ is an \dfn{ascent} of $T$ and that $q'$ is a \dfn{descent} of $T'$. We illustrate the rotation of $(\delta,\nu)$-trees in~\cref{fig:delta_rotation_tree}.

\begin{example}
The three alt~$\nu$-Tamari lattices for $\nu=ENEEN$ in terms of $(\delta,\nu)$-trees are depicted in~\cref{fig:altnu_lattices_ENEEN_trees}; compare with the path-based versions in~\cref{fig:altnu_lattices_ENEEN_paths}.
\end{example}

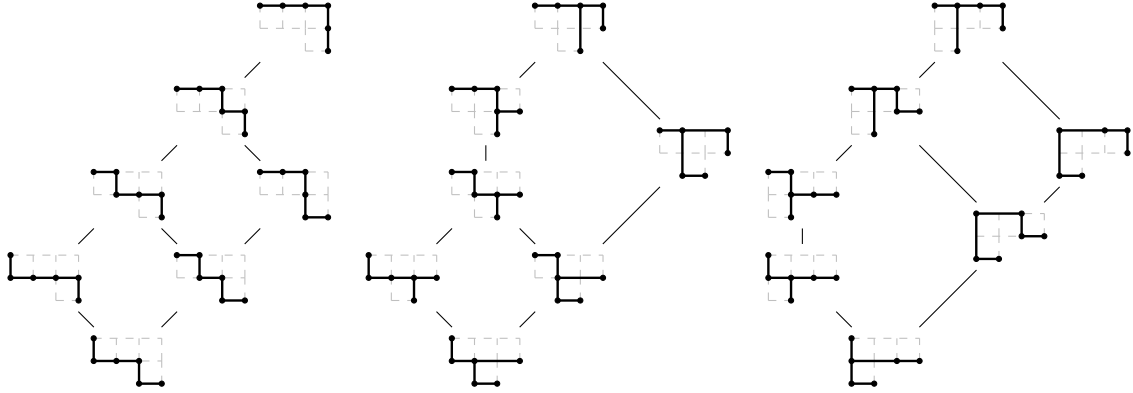
\begin{figure}[htb]
\begin{center}
\begin{tikzpicture}%
	[scale=1.1,
	x={(1cm, 0cm)},
	y={(0cm, 1cm)},
	back/.style={loosely dotted, thin},
	edge/.style={color=black},
	facet/.style={fill=red!95!black,fill opacity=0.800000},
	vertex/.style={inner sep=1pt,circle,draw=green!25!black,fill=green!75!black,thick,anchor=base}]
	\scriptsize
	
	\coordinate (c120) at (0,0);
	\coordinate (c030) at (-1,1);
	\coordinate (c021) at (0,2);
	\coordinate (c012) at (1,3);
	\coordinate (c003) at (2,4);
	
	\coordinate (c111) at (1,1);
	\coordinate (c102) at (2,2);
 
	\node (n120) at (c120) {\altnuTree{1,2,0}{0,0}{1,2,0}{.3}{}{}{}{}};
	% \node (n120) at (c120) {120};
	\node (n030) at (c030) {\altnuTree{1,2,0}{0,0}{0,3,0}{.3}{}{}{}{}};
	\node (n021) at (c021) {\altnuTree{1,2,0}{0,0}{0,2,1}{.3}{}{}{}{}};
	\node (n012) at (c012) {\altnuTree{1,2,0}{0,0}{0,1,2}{.3}{}{}{}{}};
	\node (n003) at (c003) {\altnuTree{1,2,0}{0,0}{0,0,3}{.3}{}{}{}{}};
	
	\node (n111) at (c111) {\altnuTree{1,2,0}{0,0}{1,1,1}{.3}{}{}{}{}};
	\node (n102) at (c102) {\altnuTree{1,2,0}{0,0}{1,0,2}{.3}{}{}{}{}};
	
	\draw[edge,-] (n120) -- (n030); 
	\draw[edge,-] (n030) -- (n021); 
	\draw[edge,-] (n021) -- (n012); 
	\draw[edge,-] (n012) -- (n003); 
	
	\draw[edge,-] (n120) -- (n111);
	\draw[edge,-] (n111) -- (n021);
	\draw[edge,-] (n111) -- (n102);
	
	\draw[edge,-] (n102) -- (n012);
\end{tikzpicture}
\begin{tikzpicture}%
	[scale=1.1,
	x={(1cm, 0cm)},
	y={(0cm, 1cm)},
	back/.style={loosely dotted, thin},
	edge/.style={color=black},
	facet/.style={fill=red!95!black,fill opacity=0.800000},
	vertex/.style={inner sep=1pt,circle,draw=green!25!black,fill=green!75!black,thick,anchor=base}]
	\scriptsize
	
	\coordinate (c120) at (0,0);
	\coordinate (c030) at (-1,1);
	\coordinate (c021) at (0,2);
	\coordinate (c012) at (0,3);
	\coordinate (c003) at (1,4);
	
	\coordinate (c111) at (1,1);
	\coordinate (c102) at (2.5,2.5);
	
	\node (n120) at (c120) {\altnuTree{1,2,0}{1,0}{1,2,0}{.3}{}{}{}{}};
	% \node (n120) at (c120) {120};
	\node (n030) at (c030) {\altnuTree{1,2,0}{1,0}{0,3,0}{.3}{}{}{}{}};
	\node (n021) at (c021) {\altnuTree{1,2,0}{1,0}{0,2,1}{.3}{}{}{}{}};
	\node (n012) at (c012) {\altnuTree{1,2,0}{1,0}{0,1,2}{.3}{}{}{}{}};
	\node (n003) at (c003) {\altnuTree{1,2,0}{1,0}{0,0,3}{.3}{}{}{}{}};
	
	\node (n111) at (c111) {\altnuTree{1,2,0}{1,0}{1,1,1}{.3}{}{}{}{}};
	\node (n102) at (c102) {\altnuTree{1,2,0}{1,0}{1,0,2}{.3}{}{}{}{}};
	
	\draw[edge,-] (n120) -- (n030); 
	\draw[edge,-] (n030) -- (n021); 
	\draw[edge,-] (n021) -- (n012); 
	\draw[edge,-] (n012) -- (n003); 
	
	\draw[edge,-] (n120) -- (n111);
	\draw[edge,-] (n111) -- (n102);
	\draw[edge,-] (n102) -- (n003);
	
	\draw[edge,-] (n111) -- (n021);
\end{tikzpicture}
\begin{tikzpicture}%
	[scale=1.1,
	x={(1cm, 0cm)},
	y={(0cm, 1cm)},
	back/.style={loosely dotted, thin},
	edge/.style={color=black},
	facet/.style={fill=red!95!black,fill opacity=0.800000},
	vertex/.style={inner sep=1pt,circle,draw=green!25!black,fill=green!75!black,thick,anchor=base}]
	\scriptsize
	
	\coordinate (c120) at (0,0);
	\coordinate (c030) at (-1,1);
	\coordinate (c021) at (-1,2);
	\coordinate (c012) at (0,3);
	\coordinate (c003) at (1,4);
	
	\coordinate (c111) at (1.5,1.5);
	\coordinate (c102) at (2.5,2.5);
	
	\node (n120) at (c120) {\altnuTree{1,2,0}{2,0}{1,2,0}{.3}{}{}{}{}};
	% \node (n120) at (c120) {120};
	\node (n030) at (c030) {\altnuTree{1,2,0}{2,0}{0,3,0}{.3}{}{}{}{}};
	\node (n021) at (c021) {\altnuTree{1,2,0}{2,0}{0,2,1}{.3}{}{}{}{}};
	\node (n012) at (c012) {\altnuTree{1,2,0}{2,0}{0,1,2}{.3}{}{}{}{}};
	\node (n003) at (c003) {\altnuTree{1,2,0}{2,0}{0,0,3}{.3}{}{}{}{}};
	
	\node (n111) at (c111) {\altnuTree{1,2,0}{2,0}{1,1,1}{.3}{}{}{}{}};
	\node (n102) at (c102) {\altnuTree{1,2,0}{2,0}{1,0,2}{.3}{}{}{}{}};
	
	\draw[edge,-] (n120) -- (n030); 
	\draw[edge,-] (n030) -- (n021); 
	\draw[edge,-] (n021) -- (n012); 
	\draw[edge,-] (n012) -- (n003); 
	
	\draw[edge,-] (n120) -- (n111);
	\draw[edge,-] (n111) -- (n102);
	\draw[edge,-] (n102) -- (n003);
	
	\draw[edge,-] (n111) -- (n012);
\end{tikzpicture}

	% %%draw text label
	% 	\foreach \a/\b/\c/\d in {#5}{
	% 		\draw[\d](\a,\b) node[scale=.9,anchor=west]{\c};
\end{center}
\caption{Examples of alt~$\nu$-Tamari lattices $\altTam{\nu}{\delta}$ for $\nu=ENEEN=(1,2,0)$. Left: the $\nu$-Dyck lattice, for $\delta=(0,0)$. Middle: the lattice for $\delta=(1,0)$. Right: the $\nu$-Tamari lattice, for $\delta=(2,0)$.}
\label{fig:altnu_lattices_ENEEN_trees}
\end{figure}

\subsection{The edge labeling and rowmotion of alt \texorpdfstring{$\nu$}{nu}-Tamari lattices}

Using the description of alt~$\nu$-Tamari lattices in terms of $(\delta,\nu)$-trees, M\"{u}ller introduced in~\cite{mueller2026combinatorial} a labeling of the cover relations~$T\lessdot T'$ by certain boxes $b_{T,T'}$ of the stack polyomino $F_{\delta,\nu}$. On the other hand, since alt~$\nu$-Tamari lattices are semidistributive, their cover relations admit a canonical labeling by join-irreducible elements, as described in~\cref{sec:background}. More precisely, each cover relation~$T\lessdot T'$ is labeled by a join-irreducible element $j_{T,T'}$. As shown in~\cite{mueller2026combinatorial}, these two labelings coincide. We use this box labeling to obtain a natural description of rowmotion on alt~$\nu$-Tamari lattices.

Reusing the notation from~\Cref{sec:alt_nu_trees}, let $T \lessdot T'$ be a cover relation in the alt~$\nu$-Tamari lattice~$\altTam{\nu}{\delta}$ described in terms of trees, and let $q$ and $q'$ be the corresponding ascent of $T$ and descent of~$T'$ that are exchanged. The nodes $q$ and $q'$ are the bottom-left and top-right corners of a rectangle whose bottom-right corner $r$ is also a node of both $T$ and $T'$. We label the cover relation~$T \lessdot T'$ with the box $b_{T,T'}$ of $F_{\delta,\nu}$ whose bottom-right corner is $r$, as illustrated in~\cref{fig:delta_rotation_tree}, where the rectangle is shaded gray and the box $b_{T,T'}$ is striped orange. We call $b_{T,T'}$ an \dfn{ascent box} of $T$ and a \dfn{descent box} of $T'$.

A box $b$ is an ascent box of $T$ if and only if its bottom-right corner $r$ is the second point of its row, from left to right. Similarly, $b$ is a descent box of $T'$ if and only if its bottom-right corner $r$ is the second point of its column, from top to bottom; equivalently, $r$ is the leftmost point of its row, and the parent of $r$ is not the leftmost point of its row. We identify the position of a box with the position of its bottom-right corner. In particular, a $(\delta,\nu)$-tree cannot have two ascent boxes in the same row or in the same column, and the same holds for descent boxes.

The join-irreducible elements of~$\altTam{\nu}{\delta}$ can be naturally identified with the boxes of the shape~$F_{\delta,\nu}$. Indeed, a $(\delta,\nu)$-tree is join-irreducible in~$\altTam{\nu}{\delta}$ if and only if it has exactly one descent. The corresponding box in $F_{\delta,\nu}$ is then its unique descent box. This correspondence leads to the following result.

\begin{theorem}[{M\"{u}ller~\cite[Lemma 2.10 and Corollary 2.18]{mueller2026combinatorial}}]
The join-irreducible elements in the alt $\nu$-Tamari lattice $\altTam{\nu}{\delta}$ are in bijective correspondence with the boxes of the shape $F_{\delta, \nu}$. Under this bijection, the join-irreducible element $j_{T,T'}$ associated with a cover relation $T\lessdot T'$ is mapped to the box $b_{T, T'}$. 
\end{theorem}

\begin{example}
For $\nu=(1,1,1,0)$ and $\delta=(1,1,0)$, we displayed in~\cref{fig:ex_ji_trees} the six join-irreducible $(\delta,\nu)$-trees, together with the corresponding box of the shape $F_{\delta, \nu}$.
\end{example}

\begin{figure}[ht]
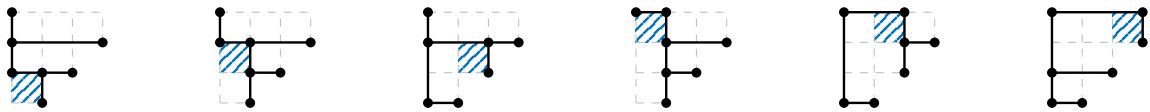

\begin{center}
\hfill
        \altnuTree{1,1,1,0}{1,1,0}{0,2,1,0}{0.4}{}{}{}{
\fill[pattern={Lines[angle=45, distance=3pt]}, pattern color=blue, opacity=1, line width=0.3mm](0,0) rectangle (1,1);

        } \hfill
        \altnuTree{1,1,1,0}{1,1,0}{0,1,2,0}{0.4}{}{}{}{
\fill[pattern={Lines[angle=45, distance=3pt]}, pattern color=blue, opacity=1, line width=0.3mm](0,1) rectangle (1,2);

        } \hfill
        \altnuTree{1,1,1,0}{1,1,0}{1,0,2,0}{0.4}{}{}{}{
\fill[pattern={Lines[angle=45, distance=3pt]}, pattern color=blue, opacity=1, line width=0.3mm](1,1) rectangle (2,2);

        }    \hfill     
        \altnuTree{1,1,1,0}{1,1,0}{0,1,1,1}{0.4}{}{}{}{
\fill[pattern={Lines[angle=45, distance=3pt]}, pattern color=blue, opacity=1, line width=0.3mm](0,2) rectangle (1,3);

        }     \hfill    
        \altnuTree{1,1,1,0}{1,1,0}{1,0,1,1}{0.4}{}{}{}{
\fill[pattern={Lines[angle=45, distance=3pt]}, pattern color=blue, opacity=1, line width=0.3mm](1,2) rectangle (2,3);

        }     \hfill    
        \altnuTree{1,1,1,0}{1,1,0}{1,1,0,1}{0.4}{}{}{}{
\fill[pattern={Lines[angle=45, distance=3pt]}, pattern color=blue, opacity=1, line width=0.3mm](2,2) rectangle (3,3);

        } \hfill
\end{center}
\caption{The join-irreducible $(\delta,\nu)$-trees for $\nu=(1,1,1,0)$, $\delta =(1,1,0)$.}
\label{fig:ex_ji_trees}
\end{figure}

Recall from~\cref{sec:background} that, for a semidistributive lattice $L$, rowmotion is defined as the operator that sends an element $x \in L$ with upper labels set~$\mathcal U(x)$ to the unique element $y \in L$ with lower labels set $\mathcal D(y)=\mathcal U(x)$. In the case of alt~$\nu$-Tamari lattices, the upper label set $\mathcal U(T)$ (resp.~lower label set $\mathcal D(T)$) of a $(\delta,\nu)$-tree~$T$ is identified with the sets of ascent (resp.~descent) boxes of $T$. This gives our first description of rowmotion on $(\delta,\nu)$-trees, as pictured in~\cref{fig:rowmotion_via_trees,fig:rowmotion_orbits_ENEEN}.

\begin{figure}[ht]
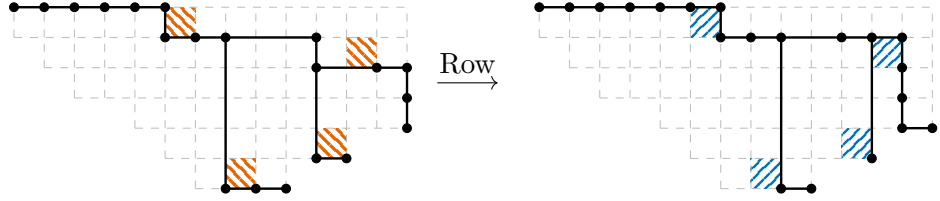

\begin{center}
   \altnuTree{3,3,3,2,1,1,0}{2,2,0,0,0,0}{2,1,0,0,2,3,5}{0.4}{}{}{}{
\fill[pattern={Lines[angle=-45, distance=3pt]}, pattern color=red, opacity=1, line width=0.4mm] (5,5) rectangle (6,6);

\fill[pattern={Lines[angle=-45, distance=3pt]}, pattern color=red, opacity=1, line width=0.4mm] (11,4) rectangle (12,5);

\fill[pattern={Lines[angle=-45, distance=3pt]}, pattern color=red, opacity=1, line width=0.4mm] (10,1) rectangle (11,2);

\fill[pattern={Lines[angle=-45, distance=3pt]}, pattern color=red, opacity=1, line width=0.4mm] (7,0) rectangle (8,1);

        \draw[->] (14,3.5) -- (16,3.5) node[midway, above] {Row};
        }
        \altnuTree{3,3,3,2,1,1,0}{2,2,0,0,0,0}{1,0,1,0,0,5,6}{0.4}{}{}{}{
\fill[pattern={Lines[angle=45, distance=3pt]}, pattern color=blue, opacity=1, line width=0.3mm](5,5) rectangle (6,6);

\fill[pattern={Lines[angle=45, distance=3pt]}, pattern color=blue, opacity=1, line width=0.3mm](11,4) rectangle (12,5);

\fill[pattern={Lines[angle=45, distance=3pt]}, pattern color=blue, opacity=1, line width=0.3mm](10,1) rectangle (11,2);

\fill[pattern={Lines[angle=45, distance=3pt]}, pattern color=blue, opacity=1, line width=0.3mm](7,0) rectangle (8,1);

        } 
\end{center}
\caption{Rowmotion via $(\delta,\nu)$-trees. The ascent boxes of the tree on the left are the descent boxes of the tree on the right.}
\label{fig:rowmotion_via_trees}
\end{figure}

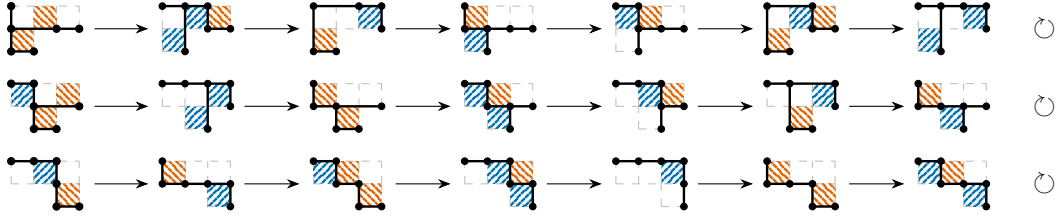
\begin{figure}[ht]
\begin{center}
\begin{tikzpicture}
[node distance=2cm,
    >={Stealth},
]

\node (f1) [inner sep=4pt] {
    \altnuTree{1,2,0}{2,0}{1,2,0}{.3}{}{}{}{
        \fill[pattern={Lines[angle=-45, distance=2pt]}, pattern color=red, opacity=1, line width=0.3mm] (0,0) rectangle +(1,1);
        \fill[pattern={Lines[angle=-45, distance=2pt]}, pattern color=red, opacity=1, line width=0.3mm] (1,1) rectangle +(1,1);
    }};
\node (f2) [right of=f1] {
    \altnuTree{1,2,0}{2,0}{0,1,2}{.3}{}{}{}{
        \fill[pattern={Lines[angle=45, distance=2pt]}, pattern color=blue, opacity=1, line width=0.3mm] (0,0) rectangle +(1,1);
        \fill[pattern={Lines[angle=45, distance=2pt]}, pattern color=blue, opacity=1, line width=0.3mm] (1,1) rectangle +(1,1);
        
        \fill[pattern={Lines[angle=-45, distance=2pt]}, pattern color=red, opacity=1, line width=0.3mm] (2,1) rectangle +(1,1);
    }};
\node (f3) [right of=f2] {
    \altnuTree{1,2,0}{2,0}{1,0,2}{.3}{}{}{}{
        \fill[pattern={Lines[angle=45, distance=2pt]}, pattern color=blue, opacity=1, line width=0.3mm] (2,1) rectangle +(1,1);
        
        \fill[pattern={Lines[angle=-45, distance=2pt]}, pattern color=red, opacity=1, line width=0.3mm] (0,0) rectangle +(1,1);
    }};
\node (f4) [right of=f3] {
    \altnuTree{1,2,0}{2,0}{0,3,0}{.3}{}{}{}{
        \fill[pattern={Lines[angle=45, distance=2pt]}, pattern color=blue, opacity=1, line width=0.3mm] (0,0) rectangle +(1,1);
        
        \fill[pattern={Lines[angle=-45, distance=2pt]}, pattern color=red, opacity=1, line width=0.3mm] (0,1) rectangle +(1,1);
    }};
\node (f5) [right of=f4] {
    \altnuTree{1,2,0}{2,0}{0,2,1}{.3}{}{}{}{
        \fill[pattern={Lines[angle=45, distance=2pt]}, pattern color=blue, opacity=1, line width=0.3mm] (0,1) rectangle +(1,1);
        
        \fill[pattern={Lines[angle=-45, distance=2pt]}, pattern color=red, opacity=1, line width=0.3mm] (1,1) rectangle +(1,1);
    }};
\node (f6) [right of=f5] {
    \altnuTree{1,2,0}{2,0}{1,1,1}{.3}{}{}{}{
        \fill[pattern={Lines[angle=45, distance=2pt]}, pattern color=blue, opacity=1, line width=0.3mm] (1,1) rectangle +(1,1);
        
        \fill[pattern={Lines[angle=-45, distance=2pt]}, pattern color=red, opacity=1, line width=0.3mm] (2,1) rectangle +(1,1);
        \fill[pattern={Lines[angle=-45, distance=2pt]}, pattern color=red, opacity=1, line width=0.3mm] (0,0) rectangle +(1,1);
    }};
\node (f7) [right of=f6] {
    \altnuTree{1,2,0}{2,0}{0,0,3}{.3}{}{}{}{
        \fill[pattern={Lines[angle=45, distance=2pt]}, pattern color=blue, opacity=1, line width=0.3mm] (2,1) rectangle +(1,1);
        \fill[pattern={Lines[angle=45, distance=2pt]}, pattern color=blue, opacity=1, line width=0.3mm] (0,0) rectangle +(1,1);
    }};

\draw[->] (f1) -- (f2);
\draw[->] (f2) -- (f3);
\draw[->] (f3) -- (f4);
\draw[->] (f4) -- (f5);
\draw[->] (f5) -- (f6);
\draw[->] (f6) -- (f7);

\node[right=3mm of f7] {$\circlearrowright$};
% \draw[->, bend left=8] (f7.south) to (f1.south);
\end{tikzpicture}
\begin{tikzpicture}
[node distance=2cm,
    >={Stealth},
]

\node (f1) [inner sep=4pt] {
    \altnuTree{1,2,0}{1,0}{1,1,1}{.3}{}{}{}{
        \fill[pattern={Lines[angle=45, distance=2pt]}, pattern color=blue, opacity=1, line width=0.3mm] (0,1) rectangle +(1,1);
        
        \fill[pattern={Lines[angle=-45, distance=2pt]}, pattern color=red, opacity=1, line width=0.3mm] (1,0) rectangle +(1,1);
        \fill[pattern={Lines[angle=-45, distance=2pt]}, pattern color=red, opacity=1, line width=0.3mm] (2,1) rectangle +(1,1);
    }};
\node (f2) [right of=f1] {
    \altnuTree{1,2,0}{1,0}{0,0,3}{.3}{}{}{}{
        \fill[pattern={Lines[angle=45, distance=2pt]}, pattern color=blue, opacity=1, line width=0.3mm] (1,0) rectangle +(1,1);
        \fill[pattern={Lines[angle=45, distance=2pt]}, pattern color=blue, opacity=1, line width=0.3mm] (2,1) rectangle +(1,1);
    }};
\node (f3) [right of=f2] {
    \altnuTree{1,2,0}{1,0}{1,2,0}{.3}{}{}{}{
        \fill[pattern={Lines[angle=-45, distance=2pt]}, pattern color=red, opacity=1, line width=0.3mm] (1,0) rectangle +(1,1);
        \fill[pattern={Lines[angle=-45, distance=2pt]}, pattern color=red, opacity=1, line width=0.3mm] (0,1) rectangle +(1,1);
    }};
\node (f4) [right of=f3] {
    \altnuTree{1,2,0}{1,0}{0,2,1}{.3}{}{}{}{
        \fill[pattern={Lines[angle=45, distance=2pt]}, pattern color=blue, opacity=1, line width=0.3mm] (1,0) rectangle +(1,1);
        \fill[pattern={Lines[angle=45, distance=2pt]}, pattern color=blue, opacity=1, line width=0.3mm] (0,1) rectangle +(1,1);
        
        \fill[pattern={Lines[angle=-45, distance=2pt]}, pattern color=red, opacity=1, line width=0.3mm] (1,1) rectangle +(1,1);
    }};
\node (f5) [right of=f4] {
    \altnuTree{1,2,0}{1,0}{0,1,2}{.3}{}{}{}{
        \fill[pattern={Lines[angle=45, distance=2pt]}, pattern color=blue, opacity=1, line width=0.3mm] (1,1) rectangle +(1,1);
        
        \fill[pattern={Lines[angle=-45, distance=2pt]}, pattern color=red, opacity=1, line width=0.3mm] (2,1) rectangle +(1,1);
    }};
\node (f6) [right of=f5] {
    \altnuTree{1,2,0}{1,0}{1,0,2}{.3}{}{}{}{
        \fill[pattern={Lines[angle=45, distance=2pt]}, pattern color=blue, opacity=1, line width=0.3mm] (2,1) rectangle +(1,1);
        
        \fill[pattern={Lines[angle=-45, distance=2pt]}, pattern color=red, opacity=1, line width=0.3mm] (1,0) rectangle +(1,1);
    }};
\node (f7) [right of=f6] {
    \altnuTree{1,2,0}{1,0}{0,3,0}{.3}{}{}{}{
        \fill[pattern={Lines[angle=45, distance=2pt]}, pattern color=blue, opacity=1, line width=0.3mm] (1,0) rectangle +(1,1);
        
        \fill[pattern={Lines[angle=-45, distance=2pt]}, pattern color=red, opacity=1, line width=0.3mm] (0,1) rectangle +(1,1);
    }};

\draw[->] (f1) -- (f2);
\draw[->] (f2) -- (f3);
\draw[->] (f3) -- (f4);
\draw[->] (f4) -- (f5);
\draw[->] (f5) -- (f6);
\draw[->] (f6) -- (f7);

\node[right=3mm of f7] {$\circlearrowright$};
% \draw[->, bend left=8] (f7.south) to (f1.south);
\end{tikzpicture}
\begin{tikzpicture}
[node distance=2cm,
    >={Stealth},
]

\node (f1) [inner sep=4pt] {
    \altnuTree{1,2,0}{0,0}{1,0,2}{0.3}{}{}{}{
        \fill[pattern={Lines[angle=45, distance=2pt]}, pattern color=blue, opacity=1, line width=0.3mm] (1,1) rectangle (2,2);
        
        \fill[pattern={Lines[angle=-45, distance=2pt]}, pattern color=red, opacity=1, line width=0.3mm] (2,0) rectangle (3,1);
    }};
\node (f2) [right of=f1] {
    \altnuTree{1,2,0}{0,0}{0,3,0}{.3}{}{}{}{
        \fill[pattern={Lines[angle=45, distance=2pt]}, pattern color=blue, opacity=1, line width=0.3mm] (2,0) rectangle (3,1);
        
        \fill[pattern={Lines[angle=-45, distance=2pt]}, pattern color=red, opacity=1, line width=0.3mm] (0,1) rectangle (1,2);
    }};
\node (f3) [right of=f2] {
    \altnuTree{1,2,0}{0,0}{1,1,1}{.3}{}{}{}{
        \fill[pattern={Lines[angle=45, distance=2pt]}, pattern color=blue, opacity=1, line width=0.3mm] (0,1) rectangle (1,2);
        
        \fill[pattern={Lines[angle=-45, distance=2pt]}, pattern color=red, opacity=1, line width=0.3mm] (1,1) rectangle (2,2);
        \fill[pattern={Lines[angle=-45, distance=2pt]}, pattern color=red, opacity=1, line width=0.3mm] (2,0) rectangle (3,1);
    }};
\node (f4) [right of=f3] {
    \altnuTree{1,2,0}{0,0}{0,1,2}{.3}{}{}{}{
        \fill[pattern={Lines[angle=45, distance=2pt]}, pattern color=blue, opacity=1, line width=0.3mm] (1,1) rectangle (2,2);
        \fill[pattern={Lines[angle=45, distance=2pt]}, pattern color=blue, opacity=1, line width=0.3mm] (2,0) rectangle (3,1);
        
        \fill[pattern={Lines[angle=-45, distance=2pt]}, pattern color=red, opacity=1, line width=0.3mm] (2,1) rectangle (3,2);
    }};
\node (f5) [right of=f4] {
    \altnuTree{1,2,0}{0,0}{0,0,3}{.3}{}{}{}{
        \fill[pattern={Lines[angle=45, distance=2pt]}, pattern color=blue, opacity=1, line width=0.3mm] (2,1) rectangle (3,2);
    }};
\node (f6) [right of=f5] {
    \altnuTree{1,2,0}{0,0}{1,2,0}{.3}{}{}{}{
        \fill[pattern={Lines[angle=-45, distance=2pt]}, pattern color=red, opacity=1, line width=0.3mm] (2,0) rectangle (3,1);
        \fill[pattern={Lines[angle=-45, distance=2pt]}, pattern color=red, opacity=1, line width=0.3mm] (0,1) rectangle (1,2);
    }};
\node (f7) [right of=f6] {
    \altnuTree{1,2,0}{0,0}{0,2,1}{.3}{}{}{}{
        \fill[pattern={Lines[angle=45, distance=2pt]}, pattern color=blue, opacity=1, line width=0.3mm] (2,0) rectangle (3,1);
        \fill[pattern={Lines[angle=45, distance=2pt]}, pattern color=blue, opacity=1, line width=0.3mm] (0,1) rectangle (1,2);
        
        \fill[pattern={Lines[angle=-45, distance=2pt]}, pattern color=red, opacity=1, line width=0.3mm] (1,1) rectangle (2,2);
    }};

\draw[->] (f1) -- (f2);
\draw[->] (f2) -- (f3);
\draw[->] (f3) -- (f4);
\draw[->] (f4) -- (f5);
\draw[->] (f5) -- (f6);
\draw[->] (f6) -- (f7);

\node[right=3mm of f7] {$\circlearrowright$};
% \draw[->, bend left=8] (f7.south) to (f1.south);
\end{tikzpicture}
\end{center}
\caption{Rowmotion has exactly one orbit of size seven for all the alt~$\nu$-Tamari lattices in~\cref{fig:altnu_lattices_ENEEN_trees}. For each tree, ascent boxes are striped orange while descent boxes are striped blue.}
\label{fig:rowmotion_orbits_ENEEN}
\end{figure}

As we showed in \cref{prop:alt_semi_ext}, the alt~$\nu$-Tamari lattices are not just semidistributive but also extremal. Hence, as discussed in~\cref{sec:background}, we can alternatively compute rowmotion \emph{in slow motion}~\cite{thomas2019rowmotion}. In other words, rowmotion can be expressed as a sequence of toggles performed in a prescribed order, as we now describe.

If $b$ is a box in $F_{\delta, \nu}$ and $T$ is a $(\delta,\nu)$-tree, then the toggle $\tau_b(T)$ is obtained from $T$ by applying the corresponding rotation, when $b$ is an ascent or descent box of $T$. Otherwise $\tau_b(T) = T$. Since we sometimes involve several increment vectors at the same time, we will refer to the toggles in~$\altTam{\nu}{\delta}$ as \dfn{$\delta$-toggles} to avoid confusion. 

\begin{lemma} \label{lem:rowmotion_slow_motion}
Label the boxes of~$F_{\delta,\nu}$ with the numbers $1,2,\dots,\ell$ from top to bottom and right to left, i.e., in ``Arabic/Hebrew reading order.'' Then we have that rowmotion on the alt $\nu$-Tamari lattice $\altTam{\nu}{\delta}$ is the following composition of $\delta$-toggles: 
\[ \row(T)=\tau_\ell \circ \dots \circ \tau_2 \circ \tau_1(T). \]
\end{lemma}

\begin{proof}
It suffices to see that there is a chain of maximal length whose sequence of labels reads $\ell, \ell - 1, \dots, 1$ from bottom to top.
    
Recall that the smallest $(\delta, \nu)$-tree corresponds to the path $\nu$ under the flushing bijection. Each north step of $\nu$ is followed by $\nu_i$ east steps, and we have $\delta_i \leq \nu_i$. Let $k$ be the index of the first non-initial north step of $\nu$, that is, the leftmost north step of $\nu$ involved in a valley. The rotation of $\nu$ at this leftmost valley will be equivalent to changing this valley into a peak. In this case, the area under the path has increased by exactly $1$.

Moreover, for each $i \geq k$, the $i$-th north step is still followed by at least $\delta_i$ east steps. Hence, we can further proceed to rotate at the leftmost valley of the resulting path, and we obtain a maximal chain from the bottom element $\nu$ to the top element of the lattice, where each cover relation increases the area by $1$. Thus, this is a chain of length $\ell$, hence of maximal length. Finally, under the flushing bijection, the sequence of box labels along this chain is precisely $\ell, \ell - 1, \dots, 1$.
\end{proof}

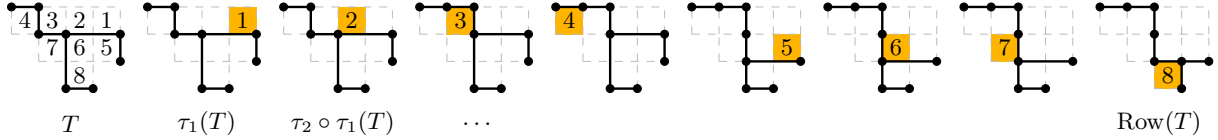
\begin{figure}[ht]
\begin{center}
\begin{tikzpicture}
\def\h {1.8}

\node (n0) at (0,0) {
\altnuTree{1,2,1,0}{1,0,0}{1,0,2,1}{0.36}{}{}{}{
\node at (3.5,2.5) {\footnotesize 1};
\node at (2.5,2.5) {\footnotesize 2};
\node at (1.5,2.5) {\footnotesize 3};
\node at (0.5,2.5) {\footnotesize 4};
\node at (3.5,1.5) {\footnotesize 5};
\node at (2.5,1.5) {\footnotesize 6};
\node at (1.5,1.5) {\footnotesize 7};
\node at (2.5,0.5) {\footnotesize 8};
}
};

\node (n1) at (1*\h,0) {
\altnuTree{1,2,1,0}{1,0,0}{1,0,2,1}{0.36}{}{}{}{
\fill[orange] (3,2) rectangle (4,3);
\node at (3.5,2.5) {\footnotesize 1};
% \node at (2.5,2.5) {\footnotesize 2};
% \node at (1.5,2.5) {\footnotesize 3};
% \node at (0.5,2.5) {\footnotesize 4};
% \node at (3.5,1.5) {\footnotesize 5};
% \node at (2.5,1.5) {\footnotesize 6};
% \node at (1.5,1.5) {\footnotesize 7};
% \node at (2.5,0.5) {\footnotesize 8};
}
};

\node (n2) at (2*\h,0) {
\altnuTree{1,2,1,0}{1,0,0}{1,0,2,1}{0.36}{}{}{}{
\fill[orange] (2,2) rectangle (3,3);
% \node at (3.5,2.5) {\footnotesize 1};
\node at (2.5,2.5) {\footnotesize 2};
% \node at (1.5,2.5) {\footnotesize 3};
% \node at (0.5,2.5) {\footnotesize 4};
% \node at (3.5,1.5) {\footnotesize 5};
% \node at (2.5,1.5) {\footnotesize 6};
% \node at (1.5,1.5) {\footnotesize 7};
% \node at (2.5,0.5) {\footnotesize 8};
}
};

\node (n3) at (3*\h,0) {
\altnuTree{1,2,1,0}{1,0,0}{1,0,1,2}{0.36}{}{}{}{
\fill[orange] (1,2) rectangle (2,3);
% \node at (3.5,2.5) {\footnotesize 1};
% \node at (2.5,2.5) {\footnotesize 2};
\node at (1.5,2.5) {\footnotesize 3};
% \node at (0.5,2.5) {\footnotesize 4};
% \node at (3.5,1.5) {\footnotesize 5};
% \node at (2.5,1.5) {\footnotesize 6};
% \node at (1.5,1.5) {\footnotesize 7};
% \node at (2.5,0.5) {\footnotesize 8};
}
};

\node (n4) at (4*\h,0) {
\altnuTree{1,2,1,0}{1,0,0}{1,0,1,2}{0.36}{}{}{}{
\fill[orange] (0,2) rectangle (1,3);
% \node at (3.5,2.5) {\footnotesize 1};
% \node at (2.5,2.5) {\footnotesize 2};
% \node at (1.5,2.5) {\footnotesize 3};
\node at (0.5,2.5) {\footnotesize 4};
% \node at (3.5,1.5) {\footnotesize 5};
% \node at (2.5,1.5) {\footnotesize 6};
% \node at (1.5,1.5) {\footnotesize 7};
% \node at (2.5,0.5) {\footnotesize 8};
}
};

\node (n5) at (5*\h,0) {
\altnuTree{1,2,1,0}{1,0,0}{1,1,0,2}{0.36}{}{}{}{
\fill[orange] (3,1) rectangle (4,2);
% \node at (3.5,2.5) {\footnotesize 1};
% \node at (2.5,2.5) {\footnotesize 2};
% \node at (1.5,2.5) {\footnotesize 3};
% \node at (0.5,2.5) {\footnotesize 4};
\node at (3.5,1.5) {\footnotesize 5};
% \node at (2.5,1.5) {\footnotesize 6};
% \node at (1.5,1.5) {\footnotesize 7};
% \node at (2.5,0.5) {\footnotesize 8};
}
};

\node (n6) at (6*\h,0) {
\altnuTree{1,2,1,0}{1,0,0}{1,1,0,2}{0.36}{}{}{}{
\fill[orange] (2,1) rectangle (3,2);
% \node at (3.5,2.5) {\footnotesize 1};
% \node at (2.5,2.5) {\footnotesize 2};
% \node at (1.5,2.5) {\footnotesize 3};
% \node at (0.5,2.5) {\footnotesize 4};
% \node at (3.5,1.5) {\footnotesize 5};
\node at (2.5,1.5) {\footnotesize 6};
% \node at (1.5,1.5) {\footnotesize 7};
% \node at (2.5,0.5) {\footnotesize 8};
}
};

\node (n7) at (7*\h,0) {
\altnuTree{1,2,1,0}{1,0,0}{1,1,0,2}{0.36}{}{}{}{
\fill[orange] (1,1) rectangle (2,2);
% \node at (3.5,2.5) {\footnotesize 1};
% \node at (2.5,2.5) {\footnotesize 2};
% \node at (1.5,2.5) {\footnotesize 3};
% \node at (0.5,2.5) {\footnotesize 4};
% \node at (3.5,1.5) {\footnotesize 5};
% \node at (2.5,1.5) {\footnotesize 6};
\node at (1.5,1.5) {\footnotesize 7};
% \node at (2.5,0.5) {\footnotesize 8};
}
};

\node (n8) at (8*\h,0) {
\altnuTree{1,2,1,0}{1,0,0}{0,2,0,2}{0.36}{}{}{}{
\fill[orange] (2,0) rectangle (3,1);
% \node at (3.5,2.5) {\footnotesize 1};
% \node at (2.5,2.5) {\footnotesize 2};
% \node at (1.5,2.5) {\footnotesize 3};
% \node at (0.5,2.5) {\footnotesize 4};
% \node at (3.5,1.5) {\footnotesize 5};
% \node at (2.5,1.5) {\footnotesize 6};
% \node at (1.5,1.5) {\footnotesize 7};
\node at (2.5,0.5) {\footnotesize 8};
}
};

\node at (0,-1) {\footnotesize $T$};
\node at (1*\h,-1) {\footnotesize $\tau_1(T)$};
\node at (2*\h,-1) {\footnotesize $\tau_2\circ \tau_1(T)$};
\node at (3*\h,-1) {\footnotesize $\dots$};
\node at (8*\h,-1) {\footnotesize $\row(T)$};

%\node at (7*\h,-1) {\footnotesize $\tau^{\operatorname{up}}(T)$};

\end{tikzpicture}

    
\end{center}
\caption{Rowmotion in slow motion: $\row(T)=\tau_8 \circ \dots \circ \tau_1 (T)$. Note that $T$ has ascent boxes $3$ and $8$, which are exactly the descent boxes of the last tree.}
\label{fig:slowmotion}
\end{figure}

An example of rowmotion computed in slow motion is shown in \cref{fig:slowmotion}.

\section{The rowmotion-intertwining bijection} \label{sec:phi}

The purpose of this section is to prove that, for a fixed path $\nu$, all alt~$\nu$-Tamari lattices $\altTam{\nu}{\delta}$ have identical rowmotion dynamics (\cref{thm:intro_main}). We construct a bijection from $\altTam{\nu}{\delta}$ to~$\altTam{\nu}{\widetilde \delta}$ where $\widetilde \delta$ is obtained from $\delta$ by decreasing its last nonzero entry by one. We show that our bijection commutes with rowmotion. Iterating this procedure eventually yields the distributive~$\nu$-Dyck lattice $\altTam{\nu}{\delta^{\min}}$, where $\delta^{\min}\coloneqq (0,\dots,0)$. Consequently, every alt $\nu$-Tamari lattice has the same rowmotion dynamics as $\altTam{\nu}{\delta^{\min}}$.

\subsection{The map \texorpdfstring{$\varphi$}{phi}}

Let $\delta=(\delta_1,\dots,\delta_n)\neq (0,\dots,0)$ be a fixed increment vector with respect to $\nu$, and let $\widetilde \delta$ be the increment vector obtained by decreasing the last nonzero entry of $\delta$ by one. Since we will work with toggle and rowmotion maps defined on both $\altTam{\nu}{\delta}$ and $\altTam{\nu}{\widetilde \delta}$, we will often use tildes to denote maps defined on the latter. For example, we denote rowmotion on this lattice by $\widetilde \row\colon \altTam{\nu}{\widetilde \delta} \to \altTam{\nu}{\widetilde \delta}$. Our goal is to construct a bijection $\varphi\colon \altTam{\nu}{\delta} \to \altTam{\nu}{\widetilde \delta}$ such that the following diagram commutes:
\begin{center}
\begin{tikzcd}
\altTam{\nu}{\delta} \arrow[r, "\varphi"] \arrow[d, "\row"'] & \altTam{\nu}{\widetilde \delta} \arrow[d, "\widetilde \row"] \\
\altTam{\nu}{\delta} \arrow[r, "\varphi"'] & \altTam{\nu}{\widetilde \delta}
\end{tikzcd}    
\end{center}
Equivalently, we wish to show that $\widetilde \row \circ \varphi = \varphi \circ \row$. Although the map $\varphi$ depends on $\delta$, we suppress this dependence in the notation throughout this section.

\begin{remark}
Our description of $\varphi$ is purely combinatorial. Nevertheless, we would like to point out that its discovery was motivated by the underlying geometric structure of the alt $\nu$-Tamari lattice. Its Hasse diagram is the edge graph of a polytopal complex induced by an arrangement of tropical hyperplanes~\cite{ceballos2024canonical}. When $\delta=(\delta_1,\dots,\delta_n)$ with $\delta_{n-1}\neq 0$, decreasing the last nonzero entry~$\delta_{n-1}$ by one corresponds to ``moving'' one of the hyperplanes in the arrangement. The map $\varphi$ is then obtained by relabeling the vertices accordingly. Two examples are illustrated in \cref{fig:tropical_varphi,fig:tropical_varphi2}, where the node labels represent rowmotion orbits (cf.~\cref{fig:alt_tamaris_14}). This yields the combinatorial description of~$\varphi$ in this particular case, which we subsequently extended to the general setting.
\end{remark}

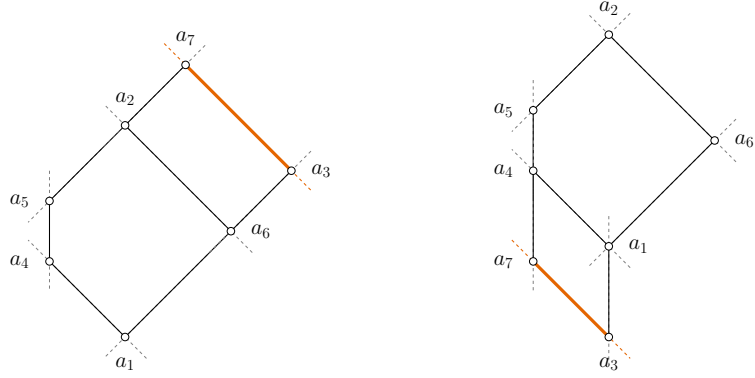
\begin{figure}[ht]
\begin{center}
\scalebox{0.4}{%
\begin{tikzpicture}
% 1. Nur die Koordinaten definieren
\coordinate (n1) at (0,0);
\coordinate (n2) at (5,5);
\coordinate (n22) at (7,7);
\coordinate (n222) at (9,9);
    
\coordinate (n3) at (-2,2);
\coordinate (n4) at (-2,5);
\coordinate (n5) at (2-.5,8+.5);

\coordinate (n55) at (4-.5,10+.5);
\coordinate (n555) at (6-.5,12+.5);
    
\coordinate (n6) at (2-4,8+4);
\coordinate (n66) at (4-4,10+4);
\coordinate (n666) at (6-4,12+4);

\coordinate (n7) at (-2-2.5,5+4.5);
\coordinate (n8) at (-2-2.5,2+2.5);
\coordinate (n9) at (-2-2.5,5+2.5);

% 2. Stile definieren
\tikzset{
  extended line/.style={
    color=black, 
    line width=1pt,
    preaction={
      draw,
      gray,
      dashed,
      line width=0.6pt,
      shorten <= -1cm,
      shorten >= -1cm
    }
  },
  red extended line/.style={
    color=red, 
    line width=3pt,
    preaction={
      draw,
      red,
      dashed,
      line width=1pt,
      shorten <= -1cm,
      shorten >= -1cm
    }
  }
}

% 3. Linien zeichnen
\draw [extended line] (n5) -- (n4);
\draw [extended line] (n5) -- (n6);
\draw [extended line] (n6) -- (n7);
\draw [extended line] (n7) -- (n9);
\draw [extended line] (n4) -- (n9);
\draw [extended line] (n5) -- (n55);
\draw [extended line] (n6) -- (n66);
\draw [red extended line, preaction={draw, red, dashed}] (n66) -- (n55);

% 4. Knoten mit etwas größeren Kreisen
\node [draw, circle, fill=white, inner sep=2.5pt, label={[label distance=4mm]below:{\huge $a_1$}}] at (n4) {};
\node [draw, circle, fill=white, inner sep=2.5pt, label={[label distance=4mm]above:{\huge $a_2$}}] at (n6) {};
\node [draw, circle, fill=white, inner sep=2.5pt, label={[label distance=4mm]right:{\huge $a_{3}$}}] at (n55) {};
\node [draw, circle, fill=white, inner sep=2.5pt, label={[label distance=4mm]left:{\huge $a_4$}}] at (n9) {};
\node [draw, circle, fill=white, inner sep=2.5pt, label={[label distance=4mm]left:{\huge $a_5$}}] at (n7) {};
\node [draw, circle, fill=white, inner sep=2.5pt, label={[label distance=4mm]right:{\huge $a_6$}}] at (n5) {};
\node [draw, circle, fill=white, inner sep=2.5pt, label={[label distance=4mm]above:{\huge $a_7$}}] at (n66) {};

\end{tikzpicture}\hspace{5cm}
\begin{tikzpicture}
% 1. Nur die Koordinaten definieren
\coordinate (n1) at (0,0);
\coordinate (n2) at (5,5);
\coordinate (n22) at (7,7);
\coordinate (n222) at (-1,-1);
    
\coordinate (n3) at (-2,2);
\coordinate (n4) at (-2,5);
\coordinate (n5) at (2-.5,8+.5);

\coordinate (n55) at (4-.5,10+.5);
\coordinate (n555) at (-2,0);
    
\coordinate (n6) at (2-4,8+4);
\coordinate (n66) at (4-4,10+4);
\coordinate (n666) at (-4.5,2.5);

\coordinate (n7) at (-2-2.5,5+4.5);
\coordinate (n8) at (-2-2.5,2+2.5);
\coordinate (n9) at (-2-2.5,5+2.5);

% 2. Stile definieren
\tikzset{
  extended line/.style={
    color=black, 
    line width=1pt,
    preaction={
      draw,
      gray,
      dashed,
      line width=0.6pt,
      shorten <= -1cm,
      shorten >= -1cm
    }
  },
  red extended line/.style={
    color=red, 
    line width=3pt,
    preaction={
      draw,
      red,
      dashed,
      line width=1pt,
      shorten <= -1cm,
      shorten >= -1cm
    }
  }
}

% 3. Linien zeichnen
\draw [extended line] (n5) -- (n4);
\draw [extended line] (n4) -- (n3);
\draw [extended line] (n5) -- (n6);
\draw [extended line] (n6) -- (n7);
\draw [extended line] (n7) -- (n9);
\draw [extended line] (n9) -- (n8);
\draw [red extended line] (n8) -- (n3);
\draw [extended line] (n4) -- (n9);

% 4. Knoten mit etwas größeren Kreisen
\node [draw, circle, fill=white, inner sep=2.5pt, label={[label distance=4mm]right:{\huge $a_1$}}] at (n4) {};
\node [draw, circle, fill=white, inner sep=2.5pt, label={[label distance=4mm]above:{\huge $a_2$}}] at (n6) {};
\node [draw, circle, fill=white, inner sep=2.5pt, label={[label distance=4mm]below:{\huge $a_3$}}] at (n3) {};
\node [draw, circle, fill=white, inner sep=2.5pt, label={[label distance=4mm]left:{\huge $a_4$}}] at (n9) {};
\node [draw, circle, fill=white, inner sep=2.5pt, label={[label distance=4mm]left:{\huge $a_5$}}] at (n7) {};
\node [draw, circle, fill=white, inner sep=2.5pt, label={[label distance=4mm]right:{\huge $a_6$}}] at (n5) {};
\node [draw, circle, fill=white, inner sep=2.5pt, label={[label distance=4mm]left:{\huge $a_7$}}] at (n8) {};

\end{tikzpicture}
}
\end{center}
\caption{Geometric intuition behind the map $\varphi$ in the case $\delta_{n-1} \neq 0$. Alt $\nu$-Tamari lattice for $\nu=(1,2,0)$ with $\delta=(2,0)$ (left) and $\delta=(1,0)$ (right).}
\label{fig:tropical_varphi}
\end{figure}

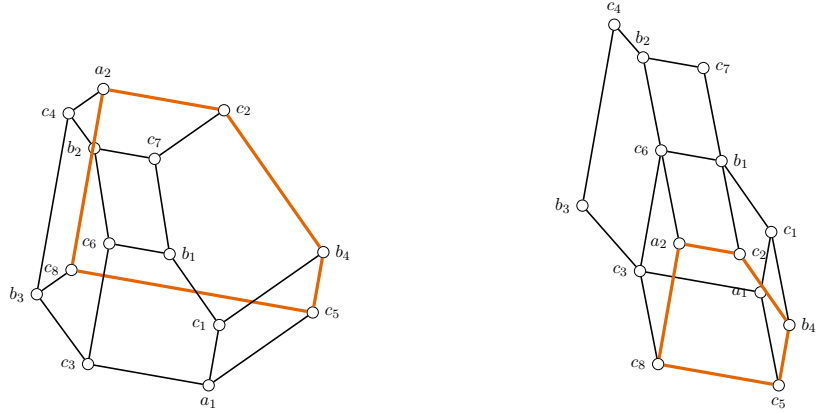
\begin{figure}[ht]
\begin{center}
\scalebox{0.6}{%
\begin{tikzpicture}%
	[z={(0cm, 4.5cm)},
	y={(-4.5cm, 0cm)},
	x={(2.2cm, 2.2cm)},
	scale=0.3,rotate=-10]

%% Coordinates of the vertices:
\coordinate (c0003) at (3,4,3);
\coordinate (c0012) at (2,4,3);
\coordinate (c0021) at (1,3,3);
\coordinate (c0030) at (0,2,2);
\coordinate (c0102) at (3,2,3);
\coordinate (c0111) at (1,2,3);
\coordinate (c0120) at (0,1,2);
\coordinate (c0201) at (3,0,1);
\coordinate (c0210) at (0,0,1);
\coordinate (c1002) at (3,4,0);
\coordinate (c1011) at (2,4,0);
\coordinate (c1020) at (0,2,0);
\coordinate (c1101) at (3,0,0);
\coordinate (c1110) at (0,0,0);

%% Drawing the vertices with labels
\node[draw,circle,fill=white,inner sep=2.5pt, label={[label distance=0mm]above:{$a_2$}}] (n0003) at (c0003) {};
\node[draw,circle,fill=white,inner sep=2.5pt, label={[label distance=0mm]left:{$c_4$}}] (n0012) at (c0012) {};
\node[draw,circle,fill=white,inner sep=2.5pt, label={[label distance=0mm]left:{$b_2$}}] (n0021) at (c0021) {};
\node[draw,circle,fill=white,inner sep=2.5pt, label={[label distance=0mm]left:{$c_6$}}] (n0030) at (c0030) {};
\node[draw,circle,fill=white,inner sep=2.5pt, label={[label distance=0mm]right:{$c_2$}}] (n0102) at (c0102) {};
\node[draw,circle,fill=white,inner sep=2.5pt, label={[label distance=0mm]above:{$c_7$}}] (n0111) at (c0111) {};
\node[draw,circle,fill=white,inner sep=2.5pt, label={[label distance=0mm]right:{$b_1$}}] (n0120) at (c0120) {};
\node[draw,circle,fill=white,inner sep=2.5pt, label={[label distance=0mm]right:{$b_4$}}] (n0201) at (c0201) {};
\node[draw,circle,fill=white,inner sep=2.5pt, label={[label distance=0mm]left:{$c_1$}}] (n0210) at (c0210) {};
\node[draw,circle,fill=white,inner sep=2.5pt, label={[label distance=0mm]left:{$c_8$}}] (n1002) at (c1002) {};
\node[draw,circle,fill=white,inner sep=2.5pt, label={[label distance=0mm]left:{$b_3$}}] (n1011) at (c1011) {};
\node[draw,circle,fill=white,inner sep=2.5pt, label={[label distance=0mm]left:{$c_3$}}] (n1020) at (c1020) {};
\node[draw,circle,fill=white,inner sep=2.5pt, label={[label distance=0mm]right:{$c_5$}}] (n1101) at (c1101) {};
\node[draw,circle,fill=white,inner sep=2.5pt, label={[label distance=0mm]below:{$a_1$}}] (n1110) at (c1110) {};

%% Drawing edges in the back (individual lines)
\draw [color=red, line width=2] (n1002) -- (n0003);
\draw [color=black, line width=1] (n1011) -- (n1002);
\draw [color=red, line width=2] (n1101) -- (n1002);

%% Drawing edges in the front (individual lines)
\draw [color=black, line width=1] (n0012) -- (n0003);
\draw [color=black, line width=1] (n0021) -- (n0012);
\draw [color=black, line width=1] (n0111) -- (n0021);
\draw [color=black, line width=1] (n0111) -- (n0102);
\draw [color=red, line width=2] (n0102) -- (n0003);
\draw [color=black, line width=1] (n1110) -- (n1020);
\draw [color=black, line width=1] (n1020) -- (n1011);
\draw [color=black, line width=1] (n1110) -- (n1101);
\draw [color=black, line width=1] (n1110) -- (n0210);
\draw [color=black, line width=1] (n0210) -- (n0120);
\draw [color=black, line width=1] (n0120) -- (n0030);
\draw [color=black, line width=1] (n0030) -- (n0021);
\draw [color=black, line width=1] (n0120) -- (n0111);
\draw [color=black, line width=1] (n1020) -- (n0030);
\draw [color=black, line width=1] (n1011) -- (n0012);
\draw [color=red, line width=2] (n1101) -- (n0201);
\draw [color=red, line width=2] (n0201) -- (n0102);
\draw [color=black, line width=1] (n0210) -- (n0201);

\end{tikzpicture}
\hspace{4cm}
\begin{tikzpicture}%
	[z={(0cm, 4.5cm)},
	y={(-4.5cm, 0cm)},
	x={(2cm, 2cm)},
	scale=0.3,rotate=-10]

%% Coordinates of the vertices:
\coordinate (c0003) at (3,6,6);
\coordinate (c0012) at (2,5,6);
\coordinate (c0021) at (1,4,5);
\coordinate (c0030) at (0,3,4);
\coordinate (c0102) at (2,4,6);
\coordinate (c0111) at (1,3,5);
\coordinate (c0120) at (0,2,4);
\coordinate (c0201) at (1,2,4);
\coordinate (c0210) at (0,1,3);
\coordinate (c1002) at (3,6,3);
\coordinate (c1011) at (1,4,3);
\coordinate (c1020) at (0,3,2);
\coordinate (c1101) at (1,2,3);
\coordinate (c1110) at (0,1,2);

%% Drawing the vertices with labels
\node[draw,circle,fill=white,inner sep=2.5pt, label={[label distance=0mm]above:{$c_4$}}] (n0003) at (c0003) {};
\node[draw,circle,fill=white,inner sep=2.5pt, label={[label distance=0mm]above:{$b_2$}}] (n0012) at (c0012) {};
\node[draw,circle,fill=white,inner sep=2.5pt, label={[label distance=0mm]left:{$c_6$}}] (n0021) at (c0021) {};
\node[draw,circle,fill=white,inner sep=2.5pt, label={[label distance=0mm]left:{$a_2$}}] (n0030) at (c0030) {};
\node[draw,circle,fill=white,inner sep=2.5pt, label={[label distance=0mm]right:{$c_7$}}] (n0102) at (c0102) {};
\node[draw,circle,fill=white,inner sep=2.5pt, label={[label distance=0mm]right:{$b_1$}}] (n0111) at (c0111) {};
\node[draw,circle,fill=white,inner sep=2.5pt, label={[label distance=0mm]right:{$c_2$}}] (n0120) at (c0120) {};
\node[draw,circle,fill=white,inner sep=2.5pt, label={[label distance=0mm]right:{$c_1$}}] (n0201) at (c0201) {};
\node[draw,circle,fill=white,inner sep=2.5pt, label={[label distance=0mm]right:{$b_4$}}] (n0210) at (c0210) {};
\node[draw,circle,fill=white,inner sep=2.5pt, label={[label distance=0mm]left:{$b_3$}}] (n1002) at (c1002) {};
\node[draw,circle,fill=white,inner sep=2.5pt, label={[label distance=0mm]left:{$c_3$}}] (n1011) at (c1011) {};
\node[draw,circle,fill=white,inner sep=2.5pt, label={[label distance=0mm]left:{$c_8$}}] (n1020) at (c1020) {};
\node[draw,circle,fill=white,inner sep=2.5pt, label={[label distance=0mm]left:{$a_1$}}] (n1101) at (c1101) {};
\node[draw,circle,fill=white,inner sep=2.5pt, label={[label distance=0mm]below:{$c_5$}}] (n1110) at (c1110) {};

%% Drawing edges in the back (individual lines)
\draw [color=black, line width=1] (n1110) -- (n1101);
\draw [color=black, line width=1] (n1101) -- (n1011);
\draw [color=black, line width=1] (n1101) -- (n0201);

%% Drawing edges in the front (individual lines)
\draw [color=red, line width=2] (n1110) -- (n1020);
\draw [color=red, line width=2] (n1110) -- (n0210);
\draw [color=black, line width=1] (n1020) -- (n1011);
\draw [color=red, line width=2] (n1020) -- (n0030);
\draw [color=black, line width=1] (n1011) -- (n1002);
\draw [color=black, line width=1] (n1011) -- (n0021);
\draw [color=black, line width=1] (n1002) -- (n0003);
\draw [color=black, line width=1] (n0210) -- (n0201);
\draw [color=red, line width=2] (n0210) -- (n0120);
\draw [color=black, line width=1] (n0201) -- (n0111);
\draw [color=black, line width=1] (n0120) -- (n0111);
\draw [color=red, line width=2] (n0120) -- (n0030);
\draw [color=black, line width=1] (n0111) -- (n0102);
\draw [color=black, line width=1] (n0111) -- (n0021);
\draw [color=black, line width=1] (n0102) -- (n0012);
\draw [color=black, line width=1] (n0030) -- (n0021);
\draw [color=black, line width=1] (n0021) -- (n0012);
\draw [color=black, line width=1] (n0012) -- (n0003);

\end{tikzpicture}
}
\end{center}
\caption{Geometric intuition behind the map $\varphi$ in the case $\delta_{n-1} \neq 0$. Alt $\nu$-Tamari lattice for $\nu=(1,1,1,0)$ with $\delta=(1,1,0)$ (left) and $\delta=(1,0,0)$ (right).}
\label{fig:tropical_varphi2}
\end{figure}

\begin{remark}
In the case of a staircase lattice path $\nu$, it is also possible to think of the map $\varphi$ as a ``piece'' of the map $D$ (or rather its inverse) from Striker--Williams~\cite[Theorem 5.4]{striker2012rowmotion} that conjugates rowmotion to promotion.
\end{remark}

To define $\varphi$, we combine a slow motion modification of rowmotion with a flushing bijection. More precisely, consider the stack polyominoes $F=F_{\delta,\nu}$ and $\widetilde F = F_{\widetilde \delta,\nu}$; see \cref{fig:two_stack_shapes} for an example. Let $\delta_k$ be the last nonzero entry of $\delta$. We denote by $F\up$ and $F\down$ the subshapes of $F$ consisting of the boxes above and below height $k$, respectively. Define $\widetilde F\up$ and $\widetilde F\down$ analogously using the same height $k$.

Note that~$\widetilde F\up=F\up$ while~$\widetilde F\down$ is obtained by translating~$F\down$ one step to the right. In other words, the shape $\widetilde F$ is obtained from $F$ by moving the lower part $F\down$ one step to the right.  We also define the vertical part $F\ver$ to consist of the boxes of $F$ in the same columns as those in~$F\down$, including those in $F\down$. Define $\widetilde F\ver$ similarly. Note that, up to translating one step to the right, the two vertical shapes coincide, $F\ver = \widetilde F\ver$.

\begin{figure}[ht]
\begin{center}
\begin{tikzpicture}
    
\node at (0,0) {
% \altnuShapegray{1,2,1,0}{1,0,0}{}{0.7}{}{}{}{
\altnuShapegraybonus{1,2,1,0}{1,0,0}{}{0.7}{}{}{}{}{
\draw (5,1) node[scale=1, black] {$\delta_1=1$};
\draw (5,2) node[scale=1, black] {$\delta_2=0$};
\draw (5,3) node[scale=1, black] {$\delta_3=0$};
\draw[very thick, red] (1,1) -- (4,1);
\draw (-1,1.5) node {$F$};
\node at (6,2) {};
}    
};

\node at (8,0) {
\altnuShapegraybonus{1,2,1,0}{0,0,0}{}{0.7}{}{}{}{}{
% \altnuShapegray{1,2,1,0}{0,0,0}{}{0.7}{}{}{}{
\draw (5,1) node[scale=1, black] {$\widetilde \delta_1=0$};
\draw (5,2) node[scale=1, black] {$\widetilde \delta_2=0$};
\draw (5,3) node[scale=1, black] {$\widetilde \delta_3=0$};
\draw[very thick, red] (1,1) -- (4,1);
\draw (-1,1.5) node {$\widetilde F$};
\node at (6,2) {};
}
};

\node at (0,-3) {
\altnuShapegraybonus{1,2,1,0}{1,0,0}{}{0.7}{}{}{}{}{
% \altnuShapegray{1,2,1,0}{1,0,0}{}{0.7}{}{}{}{
\filldraw[fill=blue!60,draw=blue!80!black] (1.1,1.1) -- (3.9,1.1) -- (3.9,2.9) -- (0.1,2.9) -- (0.1,2.1) -- (1.1,2.1) -- cycle;
\filldraw[fill=red!70,draw=red!90!black] (2.1,0.1) rectangle (2.9,0.9);
\filldraw[fill=purple!70,draw=purple!90!black, opacity=0.6] (2.05,0.05) rectangle (2.95,2.95);
\node[blue] at (-1,2) {$F^{\operatorname{up}}$};
\node[red] at (-1,0.5) {$F^{\operatorname{down}}$};
\node[purple] at (2.5,-0.5) {$F^{\operatorname{vert}}$};
\node at (6,2) {};
}
};

\node at (8,-3) {
\altnuShapegraybonus{1,2,1,0}{0,0,0}{}{0.7}{}{}{}{}{
% \altnuShapegray{1,2,1,0}{0,0,0}{}{0.7}{}{}{}{
\filldraw[fill=blue!60,draw=blue!80!black] (1.1,1.1) -- (3.9,1.1) -- (3.9,2.9) -- (0.1,2.9) -- (0.1,2.1) -- (1.1,2.1) -- cycle;
\filldraw[fill=red!70,draw=red!90!black] (3.1,0.1) rectangle (3.9,0.9);
\filldraw[fill=purple!70,draw=purple!90!black, opacity=0.6] (3.05,0.05) rectangle (3.95,2.95);
\node[blue] at (-1,2) {$\widetilde F^{\operatorname{up}}$};
\node[red] at (-1,0.5) {$\widetilde F^{\operatorname{down}}$};
\node[purple] at (3.5,-0.5) {$\widetilde F^{\operatorname{vert}}$};
\node at (6,2) {};
}
};

\end{tikzpicture}

    
\end{center}
\caption{The stack polyominoes $F=F_{\delta,\nu}$ and $\widetilde F=F_{\widetilde \delta,\nu}$, together with their upper, lower, and vertical parts. Here $\nu=(1,2,1,0)$, $\delta=(1,0,0)$, and $\widetilde \delta=(0,0,0)$.}
\label{fig:two_stack_shapes}
\end{figure}
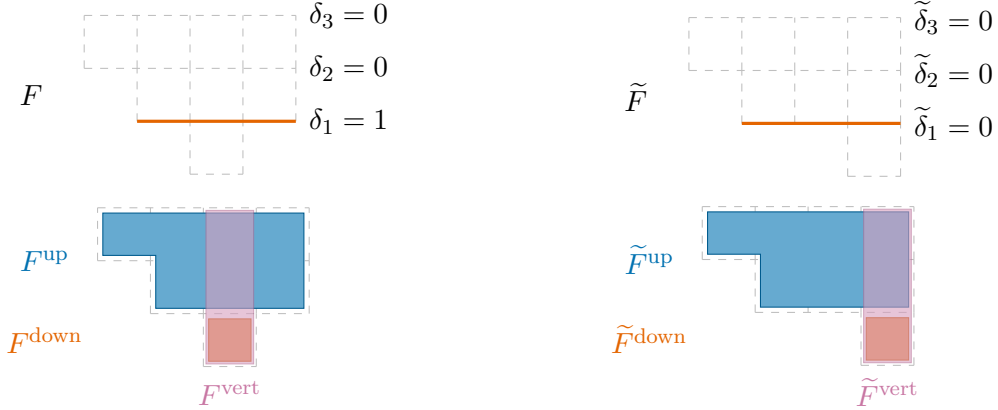

Recall from \cref{lem:rowmotion_slow_motion} that we can compute rowmotion in slow motion as follows. Let $T$ be a~$(\delta,\nu)$-tree. Label the boxes of~$F_{\delta,\nu}$ with the numbers $1,2,\dots,\ell$ from top to bottom and right to left. Denote by $\tau_i$ the $\delta$-toggle corresponding to box $i$. Then 
$\row(T)=\tau_\ell \circ \dots \circ \tau_2 \circ \tau_1(T)$.
An example, which corresponds to $F=F_{\delta,\nu}$ from \cref{fig:two_stack_shapes}, is shown in \cref{fig:slowmotion}.

Now let $\tau\up$ and $\tau\down$ denote the compositions of the $\delta$-toggles corresponding to the boxes in~$F\up$ and~$F\down$, respectively, taken in top-to-bottom, right-to-left order. Define $\widetilde \tau\up$ and $\widetilde \tau\down$ analogously using $\widetilde \delta$-toggles. Continuing our running example, the upper part $F\up$ consists of the boxes $1,\dots,7$, while the lower part $F\down$ consists of the box $8$, as shown in \cref{fig:two_stack_shapes}. Hence, in this example, $\tau\down=\tau_8$ and $\tau\up= \tau_7 \circ \dots \circ \tau_1$. Notice in general, $\row = \tau\down \circ \tau\up$.

We now define $\varphi\colon \altTam{\nu}{\delta} \to \altTam{\nu}{\widetilde \delta}$ by $\varphi \coloneqq \flush \circ \tau\up$, where $\flush\colon \altTam{\nu}{\delta} \to \altTam{\nu}{\widetilde \delta}$ is the (unique) bijection that preserves the number of nodes at each height, i.e., what was called ``right-flushing'' in \cref{sec:alt}. Note that, since $\varphi$ is the composition of two bijections, it is a bijection. Continuing our running example, an instance of this $\varphi$ is shown in \cref{fig:example_map_varphi}.\footnote{In this and subsequent figures in this section, we show boxes corresponding to upper-covers (respectively, lower-covers) as shaded in orange (resp.,~blue), which will be relevant when we discuss up/down-degrees in \cref{sec:ddeg}.}

\begin{figure}[ht]
\begin{center}
\begin{tikzpicture}[>=stealth]

  % === 1. TOP LEFT CORNER ===
  \node (TL) at (0,0) {%
    \altnuTree{1,2,1,0}{1,0,0}{1,0,2,1}{0.4}{}{}{}{
      \fill[pattern={Lines[angle=-45, distance=2pt]}, pattern color=red, opacity=1, line width=0.3mm] (2,0) rectangle +(1,1);
      \fill[pattern={Lines[angle=-45, distance=2pt]}, pattern color=red, opacity=1, line width=0.3mm] (1,2) rectangle +(1,1);
    }%
  };

  % === 2. TOP MIDDLE ===
  \node (TM) at (4,0) {%
    \altnuTree{1,2,1,0}{1,0,0}{1,1,0,2}{0.4}{}{}{}{
      \fill[pattern={Lines[angle=45, distance=2pt]}, pattern color=blue, opacity=1, line width=0.3mm] (1,2) rectangle +(1,1);
        
      \fill[pattern={Lines[angle=-45, distance=2pt]}, pattern color=red, opacity=1, line width=0.3mm] (2,0) rectangle +(1,1);
    }%
  };

  % === 3. TOP RIGHT CORNER ===
  \node (TR) at (8,0) {%
    \altnuTree{1,2,1,0}{0,0,0}{1,1,0,2}{0.4}{}{}{}{
      \fill[pattern={Lines[angle=45, distance=2pt]}, pattern color=blue, opacity=1, line width=0.3mm] (1,2) rectangle +(1,1);
        
      \fill[pattern={Lines[angle=-45, distance=2pt]}, pattern color=red, opacity=1, line width=0.3mm] (3,0) rectangle +(1,1);
    }%
  };

\draw[->, line width=1pt]
  (TL) to[bend left=20]
  node[midway, above] {$\varphi$}
  (TR);
\draw[->, line width=1pt]
  (TL) to
  node[midway, above] {$\tau^{\operatorname{up}}$}
  (TM);
\draw[->, line width=1pt]
  (TM) to
  node[midway, above] {$\flush$}
  (TR);

\end{tikzpicture}
\end{center}
\caption{An example of the map $\varphi$.}
\label{fig:example_map_varphi}
\end{figure}
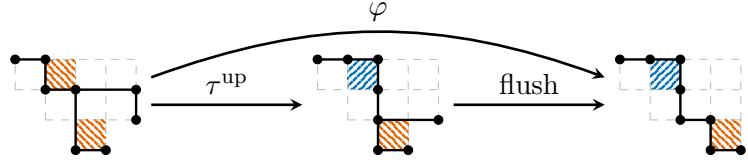

The following theorem, which shows that $\varphi$ intertwines rowmotion, is the main result of this section.

\begin{theorem} \label{thm:row_varphi_commute}
We have $\widetilde \row \circ \varphi = \varphi \circ \row$.
\end{theorem}

Continuing our running example, we depict an instance of \cref{thm:row_varphi_commute} in~\Cref{fig:varphi_row_commute}.

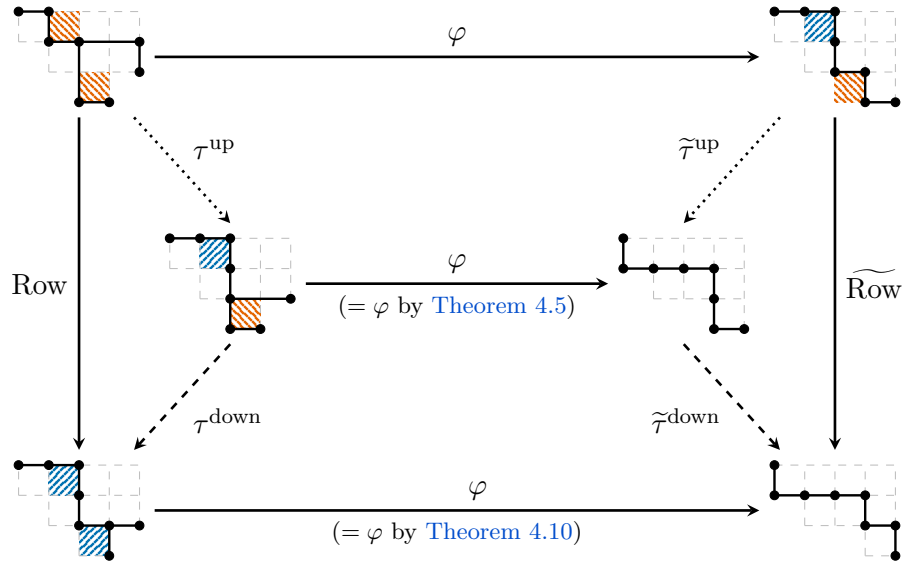
\begin{figure}[ht]
\begin{center}
\begin{tikzpicture}[>=stealth]

  % === 1. TOP LEFT CORNER ===
  \node (TL) at (0,6) {%
    \altnuTree{1,2,1,0}{1,0,0}{1,0,2,1}{0.4}{}{}{}{
      \fill[pattern={Lines[angle=-45, distance=2pt]}, pattern color=red, opacity=1, line width=0.3mm] (2,0) rectangle +(1,1);
      \fill[pattern={Lines[angle=-45, distance=2pt]}, pattern color=red, opacity=1, line width=0.3mm] (1,2) rectangle +(1,1);
    }%
  };

  % === 2. MIDDLE LEFT ===
  \node (ML) at (2,3) {%
    \altnuTree{1,2,1,0}{1,0,0}{1,1,0,2}{0.4}{}{}{}{
      \fill[pattern={Lines[angle=45, distance=2pt]}, pattern color=blue, opacity=1, line width=0.3mm] (1,2) rectangle +(1,1);
        
      \fill[pattern={Lines[angle=-45, distance=2pt]}, pattern color=red, opacity=1, line width=0.3mm] (2,0) rectangle +(1,1);
    }%
  };

  % === 4. BOTTOM LEFT CORNER ===
  \node (BL) at (0,0) {%
    \altnuTree{1,2,1,0}{1,0,0}{0,2,0,2}{0.4}{}{}{}{
      \fill[pattern={Lines[angle=45, distance=2pt]}, pattern color=blue, opacity=1, line width=0.3mm] (1,2) rectangle +(1,1);
      \fill[pattern={Lines[angle=45, distance=2pt]}, pattern color=blue, opacity=1, line width=0.3mm] (2,0) rectangle +(1,1);
    }%
  };

  % === 3. TOP RIGHT CORNER ===
  \node (TR) at (10,6) {%
    \altnuTree{1,2,1,0}{0,0,0}{1,1,0,2}{0.4}{}{}{}{
      \fill[pattern={Lines[angle=45, distance=2pt]}, pattern color=blue, opacity=1, line width=0.3mm] (1,2) rectangle +(1,1);
        
      \fill[pattern={Lines[angle=-45, distance=2pt]}, pattern color=red, opacity=1, line width=0.3mm] (2,0) rectangle +(1,1);
    }%
  };

  % === 3. MIDDLE RIGHT ===
  \node (MR) at (8,3) {%
    \altnuTree{1,2,1,0}{0,0,0}{1,0,3,0}{0.4}{}{}{}{
      % \fill[pattern={Lines[angle=45, distance=2pt]}, pattern color=blue, opacity=1, line width=0.3mm] (2,1) rectangle +(1,1);

      % \fill[pattern={Lines[angle=-45, distance=2pt]}, pattern color=red, opacity=1, line width=0.3mm] (3,0) rectangle +(1,1);
    }%
  };

  % === 6. BOTTOM RIGHT CORNER ===
  \node (BR) at (10,0) {%
    \altnuTree{1,2,1,0}{0,0,0}{0,1,3,0}{0.4}{}{}{}{
      % \fill[pattern={Lines[angle=45, distance=2pt]}, pattern color=blue, opacity=1, line width=0.3mm] (3,0) rectangle +(1,1);
      % \fill[pattern={Lines[angle=45, distance=2pt]}, pattern color=blue, opacity=1, line width=0.3mm] (2,1) rectangle +(1,1);
    }%
  };

\draw[->, line width=1pt]
  (TL) to
  node[midway, above] {$\varphi$}
  (TR);
\draw[->, line width=1pt, dotted]
  (TL) to
  node[midway, above right] {$\tau^{\operatorname{up}}$}
  (ML.north);
\draw[->, line width=1pt] 
  (BL) to 
  node[midway, above right] {$\varphi$}
  node[midway, below] {\footnotesize ($= \varphi$ by ~\Cref{prop:varphi_tau_down_commute})} 
  (BR);
\draw[->, line width=1pt] 
  (TR) to 
  node[midway, right] {$\widetilde \row$} 
  (BR);
\draw[->, line width=1pt] 
  (TL) to 
  node[midway, left] {$\row$} 
  (BL);

\draw[->, line width=1pt, dashed] 
  (ML.south) to 
  node[midway, below right] {$\tau^{\operatorname{down}}$} 
  (BL);

\draw[->, line width=1pt, dotted]
  (TR) to
  node[midway, above left] {$\widetilde\tau^{\operatorname{up}}$}
  (MR.north);
\draw[->, line width=1pt, dashed] 
  (MR.south) to 
  node[midway, below left] {$\widetilde\tau^{\operatorname{down}}$} 
  (BR);  

\draw[->, line width=1pt]
  (ML) to
  node[midway, above] {$\varphi$}
  node[midway, below] {\footnotesize ($= \varphi$ by ~\Cref{prop:varphi_tau_up_commute})}
  (MR);
\end{tikzpicture}
\end{center}
\caption{An illustration of \cref{thm:row_varphi_commute}: $\widetilde \row \circ \varphi = \varphi\circ \row$.}
\label{fig:varphi_row_commute}
\end{figure}

\subsection{Proof of \texorpdfstring{\cref{thm:row_varphi_commute}}{Theorem 4.3}}

Our strategy is to prove \cref{thm:row_varphi_commute} by establishing the following two commutation relations: $\widetilde{\tau}\up \circ \varphi = \varphi \circ \tau\up$ and $\widetilde{\tau}\down \circ \varphi = \varphi \circ \tau\down$, which are stated in \cref{prop:varphi_tau_up_commute,prop:varphi_tau_down_commute}, respectively.

As a first step, we show that the maps $\tau\up$ and $\flush$ commute. This yields an alternative description of $\varphi$. See~\cref{fig:flush_tau_up_commute} for an example.

\begin{figure}[ht]
\begin{center}
\begin{tikzpicture}[>=stealth]

  % === 1. TOP LEFT CORNER ===
  \node (TL) at (0,3) {%
    \altnuTree{1,2,1,0}{1,0,0}{1,0,2,1}{0.4}{}{}{}{
      % \fill[pattern={Lines[angle=-45, distance=2pt]}, pattern color=red, opacity=1, line width=0.3mm] (2,0) rectangle +(1,1);
      \fill[pattern={Lines[angle=-45, distance=2pt]}, pattern color=red, opacity=1, line width=0.3mm] (1,2) rectangle +(1,1);
    }%
  };

  % === 2. MIDDLE LEFT ===
  \node (BL) at (0,0) {%
    \altnuTree{1,2,1,0}{1,0,0}{1,1,0,2}{0.4}{}{}{}{
      \fill[pattern={Lines[angle=45, distance=2pt]}, pattern color=blue, opacity=1, line width=0.3mm] (1,2) rectangle +(1,1);
        
      % \fill[pattern={Lines[angle=-45, distance=2pt]}, pattern color=red, opacity=1, line width=0.3mm] (2,0) rectangle +(1,1);
    }%
  };

  % === 3. TOP RIGHT CORNER ===
  \node (TR) at (6,3) {%
    \altnuTree{1,2,1,0}{0,0,0}{1,0,2,1}{0.4}{}{}{}{
      \fill[pattern={Lines[angle=-45, distance=2pt]}, pattern color=red, opacity=1, line width=0.3mm] (1,2) rectangle +(1,1);
    }%
  };

  % === 4. MIDDLE RIGHT ===
  \node (BR) at (6,0) {%
    \altnuTree{1,2,1,0}{0,0,0}{1,1,0,2}{0.4}{}{}{}{
     \fill[pattern={Lines[angle=45, distance=2pt]}, pattern color=blue, opacity=1, line width=0.3mm] (1,2) rectangle +(1,1);
    }%
  };

\draw[->, line width=1pt]
  (TL) to
  node[midway, above] {$\flush$}
  (TR);
\draw[->, line width=1pt]
  (TL) to
  node[midway, right] {$\tau^{\operatorname{up}}$}
  (BL);

\draw[->, line width=1pt]
  (TR) to
  node[midway, left] {$\widetilde\tau^{\operatorname{up}}$}
  (BR);

\draw[->, line width=1pt]
  (BL) to
  node[midway, above] {$\flush$}
  (BR);
\end{tikzpicture}
\end{center}
\caption{An example of \cref{lem:varphi_widetilde_varphi}: $\varphi=\flush \circ \tau\up=\widetilde{\tau}\up \circ \flush$.}
\label{fig:flush_tau_up_commute}
\end{figure}
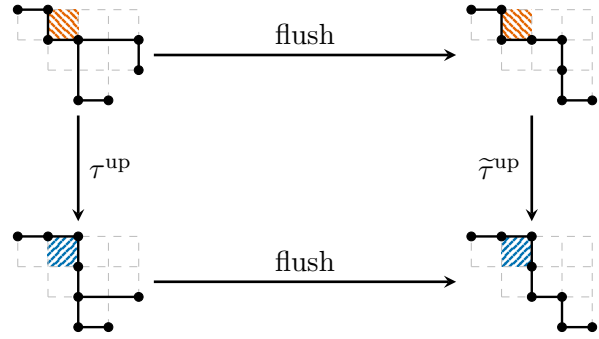

\begin{lemma} \label{lem:varphi_widetilde_varphi}
We have $ \varphi = \flush \circ \tau\up = \widetilde{\tau}\up \circ \flush$.
\end{lemma}

\begin{proof}
Let $\varphi \coloneqq \flush \circ \tau\up$ and $\widetilde \varphi \coloneqq \widetilde \tau\up \circ \flush$. Let $T$ be a $(\delta,\nu)$-tree. Observe that the nodes of $T$ lying strictly below the line separating $F\down$ and $F\up$ are unaffected by both $\flush$ and $\tau\up$. Thus, it suffices to compare the actions of $\varphi$ and $\widetilde{\varphi}$ on the nodes contained in $F\up$.

Since $F\up$ terminates on the right with a full vertical step, the restriction of $T$ to $F\up$ behaves as in the classical $\nu$-Dyck lattice setting, except that certain columns are forbidden. More precisely, for every node of $T$ in the lower part, the column directly above it must be empty unless the node is the leftmost node in its row. Ignoring these forbidden columns, the nodes of $T$ contained in $F\up$ form a northwest lattice path from the lower-right corner to the upper-left corner.

Under this identification, the action of $\tau\up$ agrees with rowmotion on a $\nu$-Dyck lattice. In this case, rowmotion exchanges valleys and peaks: a northwest path has a valley at $(i,j)$ if and only if its image under rowmotion has a peak at $(i+1,j+1)$.

On the other hand, the map $\flush$ keeps track of the forbidden columns and moves each node to the first available position to its left. This operation does not change the underlying northwest path; it only changes the columns in which the nodes are placed. Consequently, $\widetilde{\tau}\up(\flush(T))$ is obtained from $\tau\up(T)$ by shifting each node to the first available position to its left (equivalently, applying $\flush$). This is exactly the action of $\varphi$, and therefore $\varphi(T)=\widetilde{\varphi}(T)$.
\end{proof}

\begin{proposition} \label{prop:varphi_tau_up_commute}
We have $\widetilde{\tau}\up \circ \varphi = \varphi \circ \tau\up$.
\end{proposition}

\begin{proof}
By \cref{lem:varphi_widetilde_varphi}, we have $\widetilde{\tau}\up \circ \flush = \flush \circ \tau\up$. Therefore,
\begin{align*}
\widetilde{\tau}\up \circ \varphi
&= \widetilde{\tau}\up \circ (\flush \circ \tau\up) \\
&= (\widetilde{\tau}\up \circ \flush) \circ \tau\up \\
&= (\flush \circ \tau\up) \circ \tau\up \\
&= \varphi \circ \tau\up. \qedhere
\end{align*}
\end{proof}

As explained in the proof of~\cref{lem:varphi_widetilde_varphi}, the nodes of $T$ contained in $F\up$ determine a lattice path from the lower-right corner to the upper-left corner, with some forbidden gaps along its horizontal steps. Let $h_1,h_2,\dots,h_\ell$ be the heights of the horizontal segments of this path contained in $F\ver$, listed from bottom to top; see \cref{fig:varphi_map}. For convenience, define $h_0$ as follows. If the path has a horizontal segment immediately to the right of $F\ver$, then $h_0$ is the height of that segment. Otherwise, let $h_0\coloneqq h-1$, where $h$ denotes the height of the horizontal line separating~$F\up$ and~$F\down$.

\begin{figure}[ht]
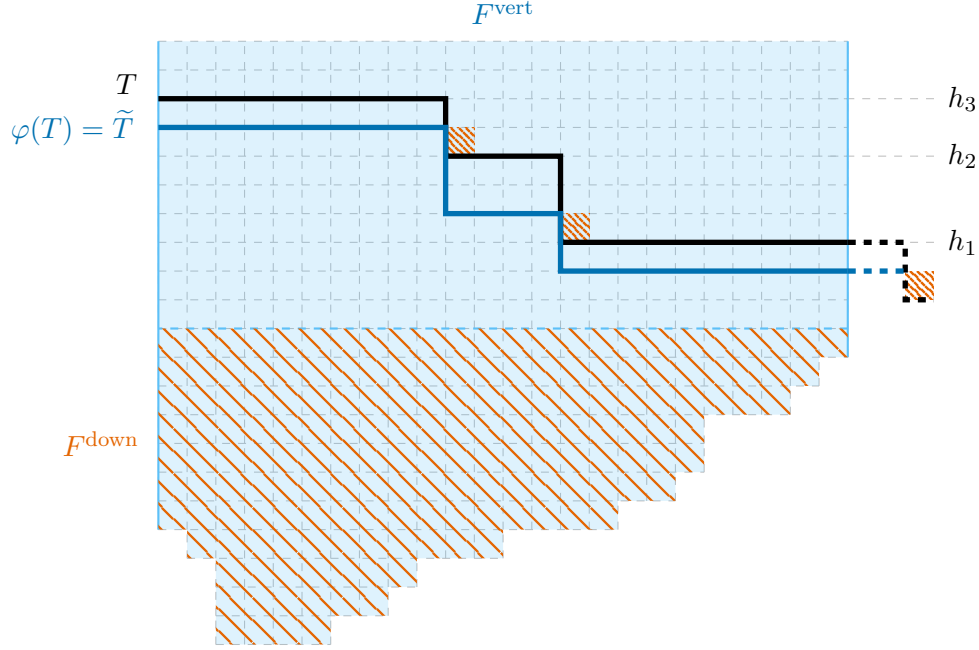

\begin{center}
% \altnuShapegray{4,2,1,4,5,2,1,0,3,1,1,0,0,0,0,0,0,0,0,0,0,0}{2,1,3,4,2,1,0,3,1,1,0,0,0,0,0,0,0,0,0,0,0}{}{0.38}{}{}{}{
\altnuShapegraybonus{4,2,1,4,5,2,1,0,3,1,1,0,0,0,0,0,0,0,0,0,0,0}{2,1,3,4,2,1,0,3,1,1,0,0,0,0,0,0,0,0,0,0,0}{}{0.38}{}{}{}{}{

\draw (-1.1,19.5) node[scale=1, black] {$T$};
\draw (-3,18) node[scale=1, blue] {$\varphi(T)=\widetilde{T}$};

\node (h1) at (28,14) {$h_1$};
\node (h2) at (28,17) {$h_2$};
\node (h3) at (28,19) {$h_3$};
% \draw (28,14) node[scale=1, black] {$h_1$};
% \draw (28,17) node[scale=1, black] {$h_2$};
% \draw (28,19) node[scale=1, black] {$h_3$};
\draw[\gray, dashed] (h1) +(-1, 0) -- +(-5,0);
\draw[\gray, dashed] (h2) +(-1, 0) -- +(-5,0);
\draw[\gray, dashed] (h3) +(-1, 0) -- +(-5,0);

\draw[skyblue,thick] (0,4) -- (0,21);
    \draw[skyblue,thick] (24,10) -- (24,21);
    \fill[skyblue, fill opacity=0.2] (0,4) -- (0,21) -- (24,21) -- (24,10) -- (23,10) -- (23,9) -- (22,9) -- (22,8) -- (19,8) -- (19,6) -- (18,6) -- (18,5) -- (16,5) -- (16,4) -- (12,4) -- (12,3) -- (9,3) -- (9,2) -- (8,2) -- (8,1) -- (6,1) -- (6,0) -- (2,0) -- (2,3) -- (1,3) -- (1,4) -- (0,4);

    \draw[skyblue,dashed, thick] (24,11) -- (0,11);

    \fill[pattern={Lines[angle=-45, distance=2pt]}, pattern color=red, opacity=1, line width=0.3mm] (10,17) rectangle +(1,1);
    \fill[pattern={Lines[angle=-45, distance=2pt]}, pattern color=red, opacity=1, line width=0.3mm] (14,14) rectangle +(1,1);
    
    \fill[pattern={Lines[angle=-45, distance=2pt]}, pattern color=red, opacity=1, line width=0.3mm] (26,12) rectangle +(1,1);

    \draw [color=black, line width=2] (0,19)--(10,19)--(10,17)--(14,17)--(14,14)--(24,14);
    \draw [color=black, dashed, line width=2] (24,14)--(26,14)--(26,12)--(27,12);

    \draw [color=blue, line width=2] (0,18)--(10,18)--(10,15)--(14,15)--(14,13)--(24,13);
    \draw [color=blue, dashed, line width=2] (24,13)--(26,13);

    \fill[pattern={Lines[angle=-45, distance=6pt]}, pattern color=red, opacity=1, line width=0.3mm] (0,4) -- (0,11) -- (24,11) -- (24,10) -- (23,10) -- (23,9) -- (22,9) -- (22,8) -- (19,8) -- (19,6) -- (18,6) -- (18,5) -- (16,5) -- (16,4) -- (12,4) -- (12,3) -- (9,3) -- (9,2) -- (8,2) -- (8,1) -- (6,1) -- (6,0) -- (2,0) -- (2,3) -- (1,3) -- (1,4) -- (0,4);

\node[blue] at (12,22) {$F^\text{vert}$};
\node[red] at (-2,7) {$F^\text{down}$};    
    }
\end{center}
\caption{Illustration of the restriction of $\varphi$ to the vertical part (\cref{lem:varphi_map}).}
\label{fig:varphi_map}
\end{figure}

Our next goal is to prove \cref{prop:varphi_tau_down_commute}, which states that the maps $\varphi$ and $\tau\down$ commute. The proof is more involved than that of \cref{prop:varphi_tau_up_commute}. Since $\tau\down$ only affects nodes in the vertical part, it suffices to analyze the action of the two compositions on the vertical part. The next lemma provides an explicit description of the action of $\varphi$ on the vertical part; see \cref{fig:varphi_map} for an illustration. An explicit example is shown in~\cref{fig:vertical_part_vertical_shift}.

\begin{figure}[ht]
\begin{center}
\begin{tikzpicture}[>=stealth]

  % === 1. TOP LEFT CORNER ===
  \node (T) at (0,9) {%
    \altnuTree{4,2,1,4,5,2,1,0,3,1,1,6,0,1,0,0,0,0,0,0,0,0}{2,1,3,4,2,1,0,3,1,1,5,0,0,0,0,0,0,0,0,0,0}{1,1,2,1,0,1,1,0,1,0,2,1,2,0,7,0,0,3,0,8,0,0}{0.38}{2/19/blue,3/19/blue,4/19/blue,5/19/blue,6/19/blue,7/19/blue,12/19/blue,12/17/red,13/17/red,15/17/red,16/17/red,16/14/purple,17/14/purple,18/14/purple,19/14/purple,22/14/purple,23/14/purple}{}{}{
    \node at (-1,0) {};
    \draw[skyblue,thick] (2,4) -- (2,21);
    \draw[skyblue,thick] (26,10) -- (26,21);
    \fill[skyblue, fill opacity=0.2] (2,4) -- (2,21) -- (26,21) -- (26,10) -- (25,10) -- (25,9) -- (24,9) -- (24,8) -- (21,8) -- (21,6) -- (20,6) -- (20,5) -- (18,5) -- (18,4) -- (14,4) -- (14,3) -- (11,3) -- (11,2) -- (10,2) -- (10,1) -- (8,1) -- (8,0) -- (4,0) -- (4,3) -- (3,3) -- (3,4) -- (2,4);        \draw[dashed,blue, ultra thick] (2,18) -- (12,18);
    \draw[dashed,red, ultra thick] (12,15) -- (16,15);
    \draw[dashed,purple,ultra thick] (16,13) -- (28,13);
    }%
  };

  % === 2. MIDDLE LEFT ===
  \node (B) at (0,0) {%
    \altnuTree{4,2,1,4,5,2,1,0,3,1,1,6,0,1,0,0,0,0,0,0,0,0}{2,1,3,4,2,1,0,3,1,1,4,0,0,0,0,0,0,0,0,0,0}{1,1,2,1,0,1,1,0,1,0,2,0,2,7,0,3,0,0,8,0,1,0}{0.38}{3/18/blue,4/18/blue,5/18/blue,6/18/blue,7/18/blue,8/18/blue,13/18/blue,13/15/red,14/15/red,16/15/red,17/15/red,17/13/purple,18/13/purple,19/13/purple,20/13/purple,23/13/purple,24/13/purple}{}{}{
    \node at (32,0) {};
    \draw[skyblue,thick] (3,4) -- (3,21);
    \draw[skyblue,thick] (27,10) -- (27,21);
    \fill[skyblue, fill opacity=0.2] (3,4) -- (3,21) -- (27,21) -- (27,10) -- (26,10) -- (26,9) -- (25,9) -- (25,8) -- (22,8) -- (22,6) -- (21,6) -- (21,5) -- (19,5) -- (19,4) -- (15,4) -- (15,3) -- (12,3) -- (12,2) -- (11,2) -- (11,1) -- (9,1) -- (9,0) -- (5,0) -- (5,3) -- (4,3) -- (4,4) -- (3,4);
    }%
  };

\draw[->, line width=1pt]
  (T) to
  node[midway, right] {$\varphi$}
  (B);

\end{tikzpicture}
\end{center}
\caption{An example of \cref{lem:varphi_map}. Here $T$ and $\widetilde T=\varphi(T)$ coincide on the vertical part, up to moving the horizontal lines in the upper part downwards.}
\label{fig:vertical_part_vertical_shift}
\end{figure}
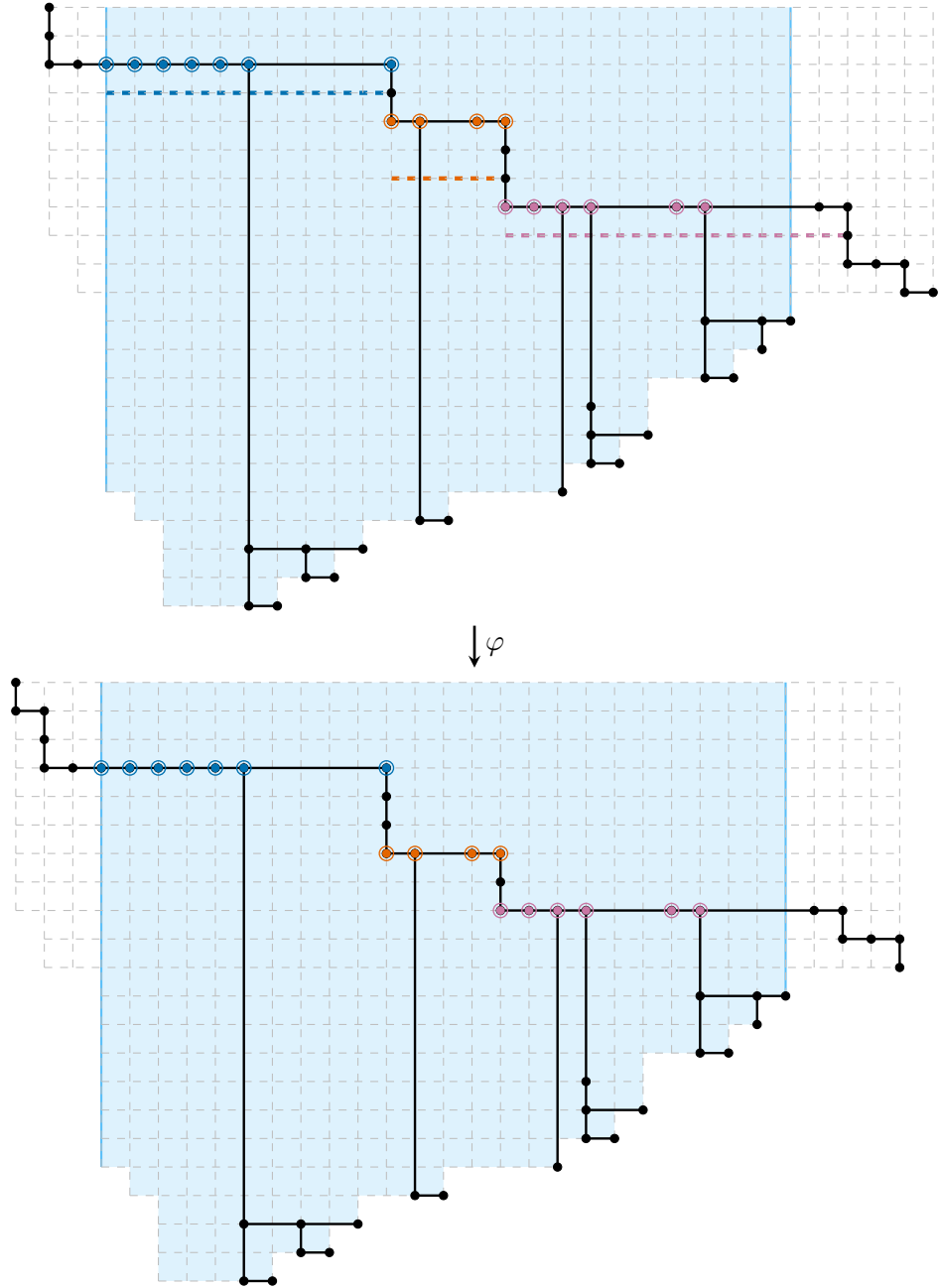

\begin{lemma}[Restriction of $\varphi$ to the vertical part] \label{lem:varphi_map} 
Let $T$ be a $(\delta,\nu)$-tree and let $\widetilde T \coloneqq \varphi(T)$. Then~$T$ and $\widetilde T$ coincide on the vertical part $F\ver=\widetilde F\ver$, except for the nodes belonging to the upper part. More precisely, for each $i\in\{1,\dots,\ell\}$ and every node of $T$ at height $h_i$, there is a node at height $h_{i-1}+1$ in $\widetilde T$ within the same column.
\end{lemma}

\begin{proof}
The map $\varphi=\flush \circ \tau\up$ is obtained by applying rowmotion to the northwest path determined by the nodes of $T$ in $F\up$, and then shifting each resulting node to the first available non-forbidden position to its left.

Under rowmotion, a valley of the path at height $h_i$ is transformed into a peak one unit above and one unit to the right. Therefore, the subsequent left shift places the corresponding node in the same column and at height $h_{i}+1$. In~\cref{fig:varphi_map}, the resulting path with peaks at these positions is shown in blue. Consequently, for every peak at height $h_i$ in $T$ there is a peak at height $h_{i-1}+1$ in $\widetilde T$ within the same column. This property also holds for the positions of the nodes on the horizontal lines of the restriction of $T$ and $\widetilde T$ to the vertical part. The other nodes (coming from north steps) are completely determined.  
\end{proof}

The following lemma compares $T'\coloneqq \tau\down(T)$ and $\widetilde T' \coloneqq \widetilde{\tau}\down(\widetilde T)$.

\begin{lemma} \label{lem:taudown_varphi}
Let $T$ be a $(\delta,\nu)$-tree and let $\widetilde T\coloneqq \varphi(T)$. Set $T'\coloneqq \tau\down(T)$ and $ \widetilde T' \coloneqq \widetilde{\tau}\down(\widetilde T)$. Then $T'$ and $\widetilde T'$ coincide on the vertical part $F\ver =\widetilde F\ver$, except for the nodes belonging to the upper part. More precisely, for each $i\in\{1,\dots,\ell\}$ and every node of $T'$ at height $h_i$, there is a node at height $h_{i-1}+1$ in $\widetilde T'$ within the same column.
\end{lemma}

\begin{proof}
For simplicity, let ($\star$) denote the property stated in the lemma, relating a pair of trees~$(S,\widetilde S)$. By \cref{lem:varphi_map}, the pair $(T,\widetilde T)$ satisfies ($\star$). Since $\tau\down$ and $\widetilde{\tau}\down$ are compositions of toggles associated with boxes in the lower part, it suffices to show that ($\star$) is preserved by a single toggle.

If a toggle is the identity, then property ($\star$) is clearly preserved. If not, we consider two cases. If a toggle does not involve any node in the upper part, then it acts identically on both trees and~($\star$) is clearly preserved.

Otherwise, the toggle exchanges the lower-left and upper-right corners of a rectangle $R$ whose bottom lies strictly in the lower part and whose top lies at height $h_i$ in the upper part, for some $i$. The corresponding toggle on $\widetilde T$ exchanges the corners of a rectangle $\widetilde R$ whose bottom coincides with that of $R$ and whose top lies at height $h_{i-1}+1$. Since ($\star$) precisely states that nodes at height $h_i$ in one tree correspond to nodes at height $h_{i-1}+1$ in the other, the effect of the toggle is compatible with the correspondence. Hence, if $(T,\widetilde T)$ satisfies ($\star$), then so does $(\tau(T),\widetilde{\tau}(\widetilde T))$.

It follows by induction over the sequence of toggles defining $\tau\down$ and $\widetilde{\tau}\down$ that $(T',\widetilde T')$ satisfies property~($\star$).
\end{proof}

\begin{lemma} \label{lem:varphi_Tprime}
Let $T$ be a $(\delta,\nu)$-tree and let $\widetilde T \coloneqq \varphi(T)$. Set $T' \coloneqq \tau\down(T)$ and $\widetilde T' \coloneqq \widetilde{\tau}\down(\widetilde T)$. Then $\varphi(T')$ and $\widetilde T'$ coincide on the vertical part $F\ver=\widetilde F\ver$.
\end{lemma}

\begin{proof}
By~\cref{lem:taudown_varphi}, the pair $(T',\widetilde T')$ satisfies the same displacement property as $(T,\widetilde T)$: for each~$i\in\{1,\dots,\ell\}$, every node of $T'$ at height $h_i$ corresponds to a node of $\widetilde T'$ at height $h_{i-1}+1$ in the same column. Applying \cref{lem:varphi_map} to $T'$ therefore moves every node of $T'$ at height $h_i$ precisely to height $h_{i-1}+1$. Hence the image $\varphi(T')$ coincides with $\widetilde T'$ on $F\ver=\widetilde F\ver$.
\end{proof}

\begin{lemma} \label{lem:varphi_taudown_commute_verticalpart}
Let $T$ be a $(\delta,\nu)$-tree. Then $\widetilde{\tau}\down \circ \varphi(T)$ and $ \varphi \circ \tau\down(T)$ coincide on the vertical part $F\ver = \widetilde F\ver$.
\end{lemma}

\begin{proof}
Let $\widetilde T\coloneqq \varphi(T)$, $T'\coloneqq \tau\down(T)$, and $\widetilde T'\coloneqq\widetilde{\tau}\down(\widetilde T)$. By \cref{lem:varphi_Tprime}, the trees $\varphi(T')=\varphi\circ\tau\down(T)$ and $\widetilde T' = \widetilde{\tau}\down \circ \varphi(T)$ coincide on the vertical part $F\ver=\widetilde F\ver$.
\end{proof}

\begin{proposition} \label{prop:varphi_tau_down_commute}
We have $\widetilde\tau\down \circ \varphi = \varphi \circ \tau\down$.
\end{proposition}

\begin{proof}
Let $T$ be a $(\delta,\nu)$-tree. By \cref{lem:varphi_taudown_commute_verticalpart}, the trees $\widetilde{\tau}\down \circ \varphi(T)$ and $\varphi \circ \tau\down(T)$ coincide on the vertical part. It therefore remains to compare the nodes outside the vertical part. Since~$\tau\down$ and~$\widetilde{\tau}\down$ are compositions of toggles supported in the lower part, neither map affects any node outside the vertical part. Moreover, the action of $\varphi$ outside the vertical part is independent of whether~$\tau\down$ has been applied beforehand. Consequently, the two compositions also coincide outside the vertical part. We conclude that $\widetilde{\tau}\down \circ \varphi(T) =
\varphi \circ \tau\down(T)$, which proves the proposition.
\end{proof}

We are now ready to prove \cref{thm:row_varphi_commute}.

\begin{proof}[Proof of \cref{thm:row_varphi_commute}]
We have
\begin{align*}
\widetilde{\row} \circ \varphi &= (\widetilde{\tau}\down \circ \widetilde{\tau}\up) \circ \varphi \\
&= \widetilde{\tau}\down \circ (\widetilde{\tau}\up \circ \varphi) \\
&= \widetilde{\tau}\down \circ (\varphi \circ \tau\up) \qquad\text{by \cref{prop:varphi_tau_up_commute}} \\
&= (\widetilde{\tau}\down \circ \varphi) \circ \tau\up \\
&= (\varphi \circ \tau\down) \circ \tau\up \qquad\text{by \cref{prop:varphi_tau_down_commute}} \\
&= \varphi \circ (\tau\down \circ \tau\up) \\
&= \varphi \circ \row. \qedhere
\end{align*}
\end{proof}

\begin{remark}
Note that the flush bijection, by itself, does \emph{not} intertwine rowmotion, since  
\[\widetilde{\row} \circ \flush = \widetilde{\tau}\down \circ \widetilde{\tau}\up \circ \flush = \flush \circ \tau\up \circ \tau\down \neq \flush \circ \tau\down \circ \tau\up = \flush \circ \row.\]
Here we are using the fact that $\tau\up$ and $\tau\down$ do not commute in general.
\end{remark}

\subsection{The up- and down-degree statistics} \label{sec:ddeg}

In this subsection, we show that the bijection $\varphi$ has the desired behavior with respect to the up- and down-degree statistics. More precisely, let~$\udeg\colon L \to \mathbb{N}$ and $\ddeg\colon L \to \mathbb{N}$ denote the \dfn{up-degree} and \dfn{down-degree} statistics on a lattice~$L$. Thus, $\udeg(x)$ is the number of elements of $L$ that cover $x$, while $\ddeg(x)$ is the number of elements of $L$ covered by $x$. Recalling the language of \cref{sec:alt}, for a $(\delta,\nu)$-tree $T$ in an alt~$\nu$-Tamari lattice $\altTam{\nu}{\delta}$, we have that $\udeg(T)$ is the number of ascent boxes of $T$ and $\ddeg(T)$ is the number of descent boxes. The main result of this subsection is the following.

\begin{theorem} \label{thm:varphi_udeg}
The bijection $\varphi\colon \altTam{\nu}{\delta} \to \altTam{\nu}{\widetilde{\delta}}$ preserves orbit averages of the up-degree statistic. More precisely,
\[ \sum_{y \in \{\row^j(x)\colon j\ge 0\}} \udeg(y) = \sum_{y \in \{\widetilde{\row}^j(\varphi(x))\colon j\ge 0\}} \udeg(y) \]
for every $x \in \altTam{\nu}{\delta}$.
\end{theorem}

We first prove the following easy corollary.
\begin{corollary} \label{cor:varphi_ddeg}
The bijection $\varphi\colon \altTam{\nu}{\delta} \to \altTam{\nu}{\widetilde{\delta}}$ preserves orbit averages of the down-degree statistic. More precisely,
\[ \sum_{y \in \{\row^j(x)\colon j\ge 0\}} \ddeg(y) = \sum_{y \in \{\widetilde{\row}^j(\varphi(x))\colon j\ge 0\}} \ddeg(y) \]
for every $x \in \altTam{\nu}{\delta}$.
\end{corollary}

\begin{proof}
By \cref{thm:varphi_udeg}, it suffices to use the identity $\ddeg(y)=\udeg(\row^{-1}(y))$, which holds for every element $y$. Indeed,
\begin{align*}
\sum_{y \in \{\row^j(x)\colon j\ge 0\}} \ddeg(y) &= \sum_{y \in \{\row^j(x)\colon j\ge 0\}} \udeg(\row^{-1}(y)) \\
&= \sum_{y \in \{\row^j(x)\colon j\ge 0\}} \udeg(y),
\end{align*}
since the orbit of $y$ coincides with the orbit of $\row^{-1}(y)$. Applying \cref{thm:varphi_udeg}, we obtain
\begin{align*}
\sum_{y \in \{\row^j(x)\colon j\ge 0\}} \udeg(y) &= \sum_{y \in \{\widetilde{\row}^j(\varphi(x))\colon j\ge 0\}} \udeg(y) \\
&= \sum_{y \in \{\widetilde{\row}^j(\varphi(x))\colon j\ge 0\}} \udeg(\widetilde{\row}^{-1}(y)) \\
&= \sum_{y \in \{\widetilde{\row}^j(\varphi(x))\colon j\ge 0\}} \ddeg(y),
\end{align*}
where the last equality follows from the analogous identity $\ddeg(y)= \udeg(\widetilde{\row}^{-1}(y))$ on $\altTam{\nu}{\widetilde{\delta}}$. This proves the claim.
\end{proof}

To prove \cref{thm:varphi_udeg}, we decompose the up-degree statistic into two contributions, corresponding to the upper and lower parts of the stack polyomino. Let $T\in \altTam{\nu}{\delta}$ be a $(\delta,\nu)$-tree, and let $F=F_{\delta,\nu}$ be the associated stack polyomino. Recall the decomposition of $F$ into its upper and lower parts, $F\up$ and $F\down$. Recall that the up-degree $\udeg(T)$ is the number of ascent boxes of~$T$, which label the upper covers of $T$. The \dfn{upper up-degree} $\udeg\up(T)$ is defined as the number of ascent boxes of $T$ lying in $F\up$, while the \dfn{lower up-degree} $\udeg\down(T)$ is the number of ascent boxes of $T$ lying in $F\down$. Hence
\[\udeg(T) = \udeg\down(T) + \udeg\up(T).\]
Applying the analogous definitions to $(\widetilde{\delta},\nu)$-trees and the decomposition of $\widetilde F$ into $\widetilde F\up$ and $\widetilde F\down$, we obtain
\[\udeg(\varphi(T)) = \udeg\down(\varphi(T)) + \udeg\up(\varphi(T)).\]

\begin{lemma} \label{lem:udeg_invariance}
The following statements hold for every $(\delta,\nu)$-tree $T$.
\begin{enumerate}
\item \label{item:udeg_invariance_down} The maps $\flush$ and $\tau\up$ preserve the lower up-degree:
\[\udeg\down(T) = \udeg\down(\flush(T)),\]
and
\[\udeg\down(T) = \udeg\down(\tau\up(T)).\]
\item \label{item:udeg_invariance_up} The maps $\flush$ and $\tau\down$ preserve the upper up-degree:
\[\udeg\up(T) = \udeg\up(\flush(T)),\]
and
\[\udeg\up(T) = \udeg\up(\tau\down(T)).\]
\end{enumerate}
\end{lemma}

\begin{proof}
The nodes of $T$ that can be increasingly flipped are precisely the leftmost nodes in rows containing at least two nodes, excluding the top row. Since $\flush$ preserves the number of nodes at each height, it preserves the number of rows whose leftmost node can be increasingly flipped. Therefore, $\flush$ preserves both $\udeg\up$ and $\udeg\down$.

Next, $\tau\up$ is a composition of toggles corresponding to boxes in $F\up$. Such toggles do not affect the configuration of nodes strictly below the separation line between $F\up$ and $F\down$. Therefore, $\udeg\down$ is preserved under $\tau\up$.

Similarly, $\tau\down$ is a composition of toggles corresponding to boxes in $F\down$. It therefore suffices to show that each such toggle preserves $\udeg\up$. A toggle in $F\down$ either acts trivially on $F\up$ or adds or removes a node that is not the leftmost node of its row. Consequently, it does not affect the set of increasingly flippable nodes in $F\up$, and hence preserves $\udeg\up$.
\end{proof}

\begin{proposition} \label{prop:udeg_split}
For every $(\delta,\nu)$-tree $T$, the map $\varphi$ preserves the lower up-degree, while the upper up-degree is transformed according to rowmotion:
\[\udeg\down(\varphi(T)) = \udeg\down(T); \qquad \udeg\up(\varphi(T)) = \udeg\up(\row(T)).\]
\end{proposition}

\begin{proof}
By definition, $\varphi(T) = \flush(\tau\up(T))$. Both $\tau\up$ and $\flush$ preserve the lower up-degree by~\cref{lem:udeg_invariance}\eqref{item:udeg_invariance_down}. Hence,
\[\udeg\down(\varphi(T)) = \udeg\down(T).\]
Moreover, using that rowmotion decomposes as $\row = \tau\down \circ \tau\up$, we obtain
\begin{align*}
\varphi(T) &= \flush(\tau\up(T)) \\
&= \flush \circ (\tau\down)^{-1} \bigl(\tau\down \circ \tau\up(T)\bigr) \\
&= \flush \circ (\tau\down)^{-1}(\row(T)).
\end{align*}
Both $(\tau\down)^{-1}$ and $\flush$ preserve the upper up-degree by \cref{lem:udeg_invariance}\eqref{item:udeg_invariance_up}. Therefore,
\[\udeg\up(\varphi(T)) = \udeg\up(\row(T)),\]
as desired.
\end{proof}

\begin{proof}[Proof of \cref{thm:varphi_udeg}]
Since $\varphi$ intertwines rowmotion by \cref{thm:row_varphi_commute}, we have
\[\sum_{y \in \{\widetilde{\row}^j(\varphi(x))\colon j\ge 0\}} \udeg(y) = \sum_{y \in \{\row^j(x)\colon j\ge 0\}} \udeg(\varphi(y)).\]
Moreover, 
\begin{align*}
\sum_{y \in \{\row^j(x)\colon j\ge 0\}} \udeg(\varphi(y)) &= \sum_{y \in \{\row^j(x)\colon j\ge 0\}} \udeg\down(\varphi(y)) + \sum_{y \in \{\row^j(x)\colon j\ge 0\}} \udeg\up(\varphi(y)) \\
&= \sum_{y \in \{\row^j(x)\colon j\ge 0\}} \udeg\down(y) + \sum_{y \in \{\row^j(x)\colon j\ge 0\}} \udeg\up(\row(y)) \\
&= \sum_{y \in \{\row^j(x)\colon j\ge 0\}} \udeg\down(y) + \sum_{y \in \{\row^j(x)\colon j\ge 0\}} \udeg\up(y) \\
&= \sum_{y \in \{\row^j(x)\colon j\ge 0\}} \udeg(y),
\end{align*}
where in the second equality we used \cref{prop:udeg_split}, and in the third equality we reindexed the orbit of rowmotion for the second term. Combining the two displayed identities yields
\[\sum_{y \in \{\widetilde{\row}^j(\varphi(x))\colon j\ge 0\}} \udeg(y) = \sum_{y \in \{\row^j(x)\colon j\ge 0\}} \udeg(y), \]
as required.
\end{proof}

\begin{remark}
A careful analysis of our proof of \cref{cor:varphi_ddeg} shows that $\varphi$ preserves not just the rowmotion orbit sums of the down-degree statistic, but also the rowmotion orbit sums of the refined statistics where we keep track of the number of descent boxes in each row separately.
\end{remark}

\subsection{Completing the proof of \texorpdfstring{\cref{thm:intro_main}}{the main theorem}}

To complete the proof of our main result, \cref{thm:intro_main}, we simply compose these $\varphi$ as needed to connect any two alt $\nu$-Tamari lattices.

\begin{proof}[Proof of \cref{thm:intro_main}]
Let $L_1$ and $L_2$ be two alt $\nu$-Tamari lattices; say $L_1=\altTam{\nu}{\delta_1}$ and $L_2=\altTam{\nu}{\delta_2}$. Let $\delta_{\min} \coloneqq (0,0,\ldots,0)$, so that $\altTam{\nu}{\delta_{\min}}$ is the distributive $\nu$-Dyck lattice. Define a sequence $\delta_1=\delta^1,\delta^2,\ldots,\delta^m=\delta_{\min}$ of increment vectors so that $\delta^{i+1}$ is obtained from $\delta^{i}$ by decrementing by one the last nonzero entry, and let $\varphi_i\colon \altTam{\nu}{\delta^i}\to\altTam{\nu}{\delta^{i+1}}$ be the corresponding $\varphi$ from above. Then define $\Phi_1\colon\altTam{\nu}{\delta_1}\to\altTam{\nu}{\delta_{\min}}$ as the composition $\Phi_1\coloneqq \varphi_{m-1}\circ \cdots \circ \varphi_1$. Similarly, define $\Phi_2\colon\altTam{\nu}{\delta_2}\to\altTam{\nu}{\delta_{\min}}$. Finally, let $\Phi\colon L_1=\altTam{\nu}{\delta_1} \to \altTam{\nu}{\delta_2} = L_2$ be $\Phi\coloneqq (\Phi_2)^{-1}\circ\Phi_1$. Each of the constituent $\phi$'s (and their inverses) intertwines rowmotion by~\cref{thm:row_varphi_commute} and has the required interaction with the down-degree statistic by~\cref{cor:varphi_ddeg}. So $\Phi$ does as well, proving the theorem.
\end{proof}

\begin{remark}
The proof we have given of \cref{thm:intro_main} requires us to compose many of the $\varphi$'s (and their inverses) to define the bijection $\Phi$ between two alt $\nu$-Tamari lattices $L_1$ and $L_2$. It could be useful to have a simpler description of $\Phi$. We believe this is sometimes possible. For example, suppose that $\widetilde{\delta}$ is obtained from $\delta$ by decrementing \emph{some} nonzero entry by one, but not necessarily the last nonzero entry. Then we believe that basically the same definition of $\varphi\colon \altTam{\nu}{\delta}\to\altTam{\nu}{\widetilde\delta}$ as $\varphi\coloneqq \flush \circ \tau\up$ should work, as long as one carefully and correctly defines the upper-part $F\up$. The reason we preferred to consider only the case where $\widetilde{\delta}$ is obtained from $\delta$ by decrementing the \emph{last} nonzero entry is because then the $\delta$-toggles corresponding to boxes in $F\up$ behave essentially as distributive lattice toggles, which are easier to analyze. Nevertheless, the more general $\varphi$'s we sketched in this remark could be useful when considering the extension to the cross Tamari case (see \cref{sec:cross}).
\end{remark}

\section{Rational Tamari lattices} \label{sec:rational}

For a generic choice of lattice path~$\nu$, the alt~$\nu$-Tamari lattices are not particularly well behaved. For example, while there is a determinantal formula for the number of $\nu$-Dyck paths for generic~$\nu$, there is no simple product formula. Additionally, rowmotion tends to behave in a rather disorderly fashion on generic alt~$\nu$-Tamari lattices. For certain special families of lattice paths~$\nu$, however, rowmotion exhibits a remarkably orderly behavior. We study these special families of alt~$\nu$-Tamari lattices in this section.

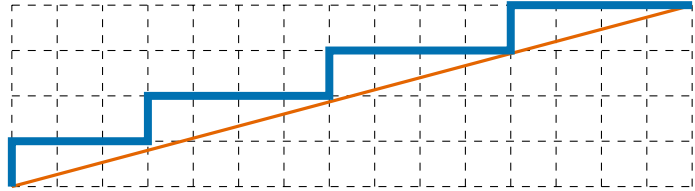
\begin{figure}[ht]
\begin{center}
\begin{tikzpicture}[scale = 0.6]
\draw[step=1cm, dashed] (0,0) grid (15,4);
\draw[red,very thick] (0,0) -- (15,4);
\draw[blue,line width=3pt] (0,0) -- ++(0,1) -- ++(3,0) -- ++(0,1) -- ++(4,0) -- ++(0,1) -- ++(4,0) -- ++(0,1) -- ++(4,0); 
\end{tikzpicture}
\end{center}
\caption{For $(a,b)=(4,15)$, the rational Dyck paths are those above the path~$\nu$ drawn in blue above.}
\label{fig:rat_path_ex}
\end{figure}
 
Fix coprime positive integers $a < b$, and let $\nu$ be the lowest northeast path that connects the lower-left and upper-right corners of an $a\times b$ rectangle and stays weakly above the diagonal line connecting these two corners. 
For example, if~$(a,b)=(4,15)$, then~$\nu$ is the lattice path $\nu = NE^3NE^4NE^4NE^4$ corresponding to the partition shape $\lambda=(11,7,3,0)$, as shown in~\cref{fig:rat_path_ex}. We call the corresponding set of $\nu$-Dyck paths the set of \dfn{(a,b)-rational Dyck paths}, or just the set of rational Dyck paths when $(a,b)$ is clear. The corresponding (alt) $\nu$-Tamari lattices are called the \dfn{rational (alt) Tamari lattices}. It is well-known (see e.g.~\cite{bizley1954longtitle}) that the number of $(a,b)$-rational Dyck paths is given by the corresponding \dfn{rational Catalan number}
\[\cat(a,b) \coloneqq \frac{1}{a+b}\binom{a+b}{a}.\]
Notice that when~$(a,b)=(n,n+1)$ for an integer $n\geq 1$, we recover the classical Catalan numbers~$\cat(n) = \frac{1}{n+1}\binom{2n}{n}$ and the classical set of Dyck paths of semilength $n$. (See \cite{stanley2015catalan} for over 200 combinatorial interpretations of the Catalan numbers.) More generally, when~$(a,b)=(n,nm+1)$ for integers $n,m\geq 1$, the numbers $\cat(n,nm+1)=\frac{1}{mn+1}\binom{(m+1)n}{n}$ are called \dfn{Fuss--Catalan numbers} and these (alt) Tamari lattices are the \dfn{(alt) $m$-Tamari lattices}~\cite{bergeron2012higher}. See~\cite[Chapter~1]{stump2025cataland} for a discussion of the history and significance of Fuss (and rational) generalizations of Catalan objects in the context of Coxeter--Catalan combinatorics.

We will show how rowmotion behaves well on rational alt Tamari lattices, building on previous work in~\cite{thomas2019rowmotion, defant2024tamari} and addressing some conjectures posed there. The key strategy is to use \cref{thm:intro_main} to transfer to the distributive case, which is the easiest to analyze combinatorially.

\subsection{Cyclic sieving}

We first focus on the order and orbit structure of rowmotion acting on rational alt Tamari lattices. The concept of cyclic sieving~\cite{reiner2004cyclic, sagan2011cyclic} gives a compact way to record the orbit structure of a cyclic group action on a combinatorial set. If $X$ is a finite set acted upon by a cyclic group $C=\langle c \rangle\simeq \mathbb{Z}/N\mathbb{Z}$ and $f\in \mathbb{N}[q]$ is a polynomial with nonnegative integer coefficients, then we say that the triple $(X,C,f)$ exhibits \dfn{cyclic sieving} if for all $k\geq 0$, the number of $x\in X$ fixed by $c^k$ is $\#X^{c^k}=f(\zeta^{k})$ where $\zeta=e^{2\pi i/N}$ is a primitive $N$-th root of unity. Note that in the case where the sieving polynomial $f(q)$ has a product formula as a rational function in~$q$, this implies in particular that each symmetry class of $X$ under the $C$ action is enumerated by a product formula.

Let $[n]_q=1+q+\cdots+q^{n-1} \coloneqq \frac{1-q^n}{1-q}$ denote the standard \dfn{$q$-number}, $[n]_q! \coloneqq [n]_q\cdot[n-1]_q\cdots[1]_q$ the \dfn{$q$-factorial}, and $\binom{n}{k}_q \coloneqq \frac{[n]_q!}{[k]_q![n-k]_q!}$ the \dfn{$q$-binomial}. For coprime integers $a<b$, it is well-known (see, e.g.,~\cite[\S4]{armstrong2016rational}) that the \dfn{rational $q$-Catalan number}
\[\cat_q(a,b) \coloneqq \frac{1}{[a+b]_q}\binom{a+b}{b}_q\]
is a polynomial in $q$ with nonnegative integer coefficients.

\begin{theorem}[{See~\cite[Conjecture 5.15]{defant2024tamari} and \cite[\S7.6]{thomas2019rowmotion}}]
For any coprime integers $a<b$, the order of rowmotion acting on any $(a,b)$ rational alt Tamari lattice $L$ divides $a+b-1$. Moreover, the triple $(L,\mathbb{Z}/(a+b-1)\mathbb{Z},\cat_q(a,b))$ exhibits cyclic sieving, where the generator of $\mathbb{Z}/(a+b-1)\mathbb{Z}$ acts on $L$ as $\row$.
\end{theorem}

\begin{proof}
By \cref{thm:intro_main}, it is enough to prove this in the distributive case, i.e., for the lattice of~$(a,b)$-rational Dyck paths ordered by nesting. The proof of the desired cyclic sieving result for this distributive lattice is basically explained in~\cite[\S7.5]{thomas2019rowmotion}, but let us give more details here.

As explained in \cite[\S5.1]{armstrong2013rational}, the set of $(a,b)$-rational Dyck paths is in bijection with a certain collection of set partitions of $[a+b-1]=\{1,2,\ldots,a+b-1\}$ called \dfn{$(a,b)$-homogeneous noncrossing set partitions}. Let us review this construction. Let $D$ be an $(a,b)$-rational Dyck path. Label its steps $1,2,\ldots,a+b$ in order from the lower-left to upper-right corners of the $a\times b$ rectangle. Slightly above each north step, draw a diagonal line with the same slope as the line connecting the lower-left and upper-right corners of the $a\times b$ rectangle. This divides the labels $1,2,\ldots,a+b-1$ into a noncrossing set partition $\mu$ with exactly $a$ blocks (we ignore the label $a+b$). See \cref{fig:rational_dyck_path} for an example in the case $(a,b)=(5,8)$. In~\cite[\S5.1]{armstrong2013rational} it is proved that this map $D\mapsto \mu$ is injective, and the image is called the set of $(a,b)$-homogeneous noncrossing set partitions. (Note that, in the classical Catalan case $(a,b)=(n,n+1)$, the $(a,b)$-homogeneous noncrossing set partitions are exactly the \dfn{noncrossing perfect matchings} of $[2n]$.)

Moreover, in~\cite[\S5.2]{armstrong2013rational}, Armstrong, Rhoades, and Williams show that \dfn{promotion} of the rational Dyck path $D$ corresponds to \dfn{rotation} of the homogeneous noncrossing partition $\mu$. Here promotion of Dyck paths is a certain ``left-to-right'' composition of the distributive lattice toggles. The by now well-understood technique of Striker and Williams~\cite{striker2012rowmotion} allows one to show that this ``left-to-right'' composition of toggles (i.e., promotion) is conjugate, in the group generated by toggles, to the ``top-to-bottom'' composition of toggles (i.e., rowmotion). Hence, rowmotion in the distributive lattice of rational Dyck paths is in equivariant bijection with rotation of homogeneous noncrossing partitions. This implies that the order of rowmotion divides $a+b-1$.

Finally, we conclude by noting that Bodnar and Rhoades~\cite[Theorem 5.4]{bodnar2016rational} proved that rotation of $(a,b)$ homogeneous noncrossing partitions exhibits cyclic sieving with the sieving polynomial being the rational $q$-Catalan number $\cat_q(a,b)$. This completes the proof.
\end{proof}

\begin{figure}[ht]
\begin{center}
\scalebox{0.8}{\begin{tikzpicture}
    \draw [color=black, line width=2] (0,0)--(0,2)--(1,2)--(1,3)--(3,3)--(3,4)--(6,4)--(6,5)--(8,5);
    
    \filldraw[line width=1, black, fill=gray!30] 
    (0,2) to[bend right=30] (1,2) 
         to[bend right=10] (3,3) 
         to[bend left=50] cycle;

    \filldraw[line width=1, black, fill=gray!30] 
    (0,1) to[bend right=20] (5,4) 
         to[bend right=30] (6,4) 
         to[bend left=20] cycle;
         
    \draw [color=gray, line width=1, dashed] (0,0)--(8,5);
    \draw [color=red, line width=1, dashed] (1,2)--(2.6,3);
    \draw [color=red, line width=1, dashed] (3,3)--(4.6,4);
    \draw [color=red, line width=1, dashed] (0,1)--(4.8,4);
    \draw [color=red, line width=1, dashed] (6,4)--(7.6,5);

    %\draw[color=black, line width=1] (0,2) to[bend right=40]  (1,2);
    \draw[color=black, line width=1] (1,3) to[bend right=40]  (2,3);
    \draw[color=black, line width=1] (3,4) to[bend right=40]  (4,4);
    %\draw[color=gray, line width=1, dashed] (5,4) to[bend right=40]  (6,4);
    \draw[color=black, line width=1] (6,5) to[bend right=40]  (7,5);

\draw (0,0) node [scale=0.8, circle, draw, fill=black, anchor=center] {};
    \draw (0,1) node [scale=0.5, circle, draw, fill=black, anchor=center, label=north west:{1}] {};
    \draw (0,2) node [scale=0.5, circle, draw, fill=black, anchor=center, label=north west:{2}] {};
    \draw (1,2) node [scale=0.5, circle, draw, fill=black, anchor=center, label=north west:{3}] {};
    \draw (1,3) node [scale=0.5, circle, draw, fill=black, anchor=center, label=north west:{4}] {};
    \draw (2,3) node [scale=0.5, circle, draw, fill=black, anchor=center, label=north west:{5}] {};
    \draw (3,3) node [scale=0.5, circle, draw, fill=black, anchor=center, label=north west:{6}] {};
    \draw (3,4) node [scale=0.5, circle, draw, fill=black, anchor=center, label=north west:{7}] {};
    \draw (4,4) node [scale=0.5, circle, draw, fill=black, anchor=center, label=north west:{8}] {};
    \draw (5,4) node [scale=0.5, circle, draw, fill=black, anchor=center, label=north west:{9}] {};
    \draw (6,4) node [scale=0.5, circle, draw, fill=black, anchor=center, label=north west:{10}] {};
    \draw (6,5) node [scale=0.5, circle, draw, fill=black, anchor=center, label=north west:{11}] {};
    \draw (7,5) node [scale=0.5, circle, draw, fill=black, anchor=center, label=north west:{12}] {};
    \draw (8,5) node [scale=0.8, circle, draw, fill=black, anchor=center] {};

\end{tikzpicture}
\newcommand{\ncpartition}[1]{%
  \begin{tikzpicture}[scale=2.5]
    \draw[gray!50, line width=0.8pt] (0,0) circle (1);

    \foreach \i in {1,...,12} {
      \coordinate (n\i) at ({90 - (\i - 1) * 30}:1);
    }

    #1

    \foreach \i in {1,...,12} {
      \node[
        circle, 
        fill=black, 
        inner sep=1.8pt, 
        label={[font=\normalsize, label distance=-2pt]{90 - (\i - 1) * 30}:\i}
      ] at (n\i) {};
    }
  \end{tikzpicture}%
}
\ncpartition{

  \filldraw[line width=1.5, black, fill=gray!30] 
    (n1) to[bend left=30] (n9) 
         to[bend right=10] (n10) 
         to[bend right=20] cycle;

  \filldraw[line width=1.5, black, fill=gray!30] 
    (n2) to[bend right=10] (n3) 
         to[bend right=30] (n6) 
         to[bend left=20] cycle;
         
\draw[line width=1.5, black] (n4) -- (n5);
  \draw[line width=1.5, black] (n7) -- (n8);
  \draw[line width=1.5, black] (n11) -- (n12);
}}
\end{center}
\caption{An example of the bijection $D\mapsto\mu$ from rational Dyck paths to homogeneous noncrossing partitions, in the case $(a,b)=(5,8)$.} 
\label{fig:rational_dyck_path}
\end{figure}
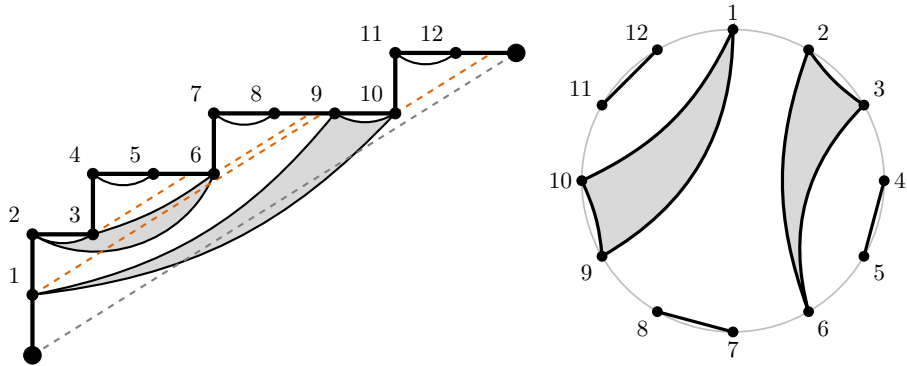

\subsection{Homomesy}

Next, we study how rowmotion of rational Tamari lattices interacts with the down-degree statistic. If $X$ is a finite set, $C$ is a cyclic group acting on $X$, and $f\colon X \to \mathbb{R}$ is some statistic on $X$, then we say that $f$ is \dfn{homomesic} with respect to the action of $C$ on $X$ if the average of $f$ on every $C$-orbit is the same~\cite{propp2015homomesy}. 

\begin{theorem}[{cf.~\cite[Theorem 5.14]{defant2024tamari}}] \label{thm:homomesy}
For any $n,m\geq 1$, and for any alt $m$-Tamari lattice $L$, the down-degree statistic $\ddeg$ is homomesic with respect to the action of rowmotion on $L$, with average $\frac{m(n-1)}{m+1}$.
\end{theorem}

\begin{proof}
As indicated, this is an extension of \cite[Theorem 5.14]{defant2024tamari}, which addressed the case of the~$m$-Tamari lattice. By \cref{thm:intro_main}, this case implies all the other cases. However, here we supply an alternative proof, where we reduce instead to the distributive case, i.e., the lattice of $(n,nm+1)$ rational Dyck paths ordered by nesting. In this case, \cite[Corollary 3.11]{chan2017expected} implies the desired rowmotion homomesy result, because these rational Dyck paths are exactly those inside the partition shape $\lambda=(m(n-1),m(n-2),\ldots,m)$, which is a \dfn{balanced} shape in the sense of that paper.
\end{proof}

\begin{conjecture}[{Cf.~\cite[Conjecture 5.16]{defant2024tamari}}] \label{conj:homomesy}
For any coprime integers $a<b$ and any $(a,b)$-rational alt Tamari lattice $L$, the down-degree statistic $\ddeg$ is homomesic with respect to the action of rowmotion on $L$, with average $\frac{(a-1)(b-1)}{a+b-1}$.
\end{conjecture}

\begin{remark}
The proof we gave of \cref{thm:homomesy} does not work for \cref{conj:homomesy}, because, outside of the Fuss setting, the partition shapes corresponding to rational Dyck paths are not balanced. Moreover, the technique of~\cite{chan2017expected} (which was further developed in \cite{defant2023homomesy}), in which one writes the down-degree statistic as a (constant plus a) linear combination of \dfn{toggleability statistics}~\cite{striker2015tog}, cannot possibly work in the general rational setting, because at this level of generality the down-degree statistic is \emph{not} a (constant plus a) linear combination of toggleability statistics.
\end{remark}

\begin{remark}
It is well-known that the Hasse diagram of the classical Tamari lattice is the one-skeleton of a simple convex polytope, namely, the associahedron. In particular, this means that the Hasse diagram of the classical Tamari lattice is regular. Hence, for any element $x\in \Tam{n+1}$, we will have that $\ddeg(\row(x))+\ddeg(x)=n$, since the up-degree of an element is the down-degree of its image under rowmotion. This yields a very quick proof of down-degree homomesy for rowmotion of the classical Tamari lattice. But this argument does not extend to the rational Tamari lattices, which do not have regular Hasse diagrams.
\end{remark}

\section{Future directions} \label{sec:future}

\subsection{Cross Tamari lattices} \label{sec:cross}

The definition of the alt~$\nu$-Tamari lattices can be extended naturally to depend on moon polyomino shapes~\cite{jonsson2005generalized,rubey2012maximal} to give the family of \dfn{cross Tamari lattices}. For a formal definition of cross Tamari lattices in terms of trees, we refer to~\cite[Section 2.4]{vonbell2025framing}. Note, as mentioned in \cref{sec:alt}, that the alt~$\nu$-Tamari lattices are exactly the cross Tamari lattices associated to stack polyominoes. All cross Tamari lattices are semidistributive (and extremal), and hence carry an action of rowmotion.

One might hope that the methods developed in this paper extend directly to cross Tamari lattices as well. However, in general, the (naive extension of the) rowmotion-intertwining bijection $\varphi$ from~\cref{sec:phi} does not intertwine rowmotion for cross Tamari lattices associated with arbitrary moon polyominoes. A counterexample, demonstrating the failure of~\cref{lem:varphi_widetilde_varphi,prop:varphi_tau_up_commute,prop:varphi_tau_down_commute} in this setting, is illustrated in~\cref{fig:example_cross_no_commute}. Nevertheless, we conjecture that our main result, \cref{thm:intro_main}, extends directly to cross Tamari lattices.

\begin{conjecture} \label{conj:cross_tamari}
Let $L_1$ and $L_2$ be two cross Tamari lattices associated to moon polyominoes that are related by a sequence of row and column permutations. Then rowmotion behaves the same for~$L_1$ and $L_2$. That is, there is a bijection $\Phi\colon L_1 \to L_2$ such that:
\begin{itemize}
\item $\Phi( \row (x) ) = \row (\Phi (x))$ for all $x \in L_1$;
\item $\sum_{y \in \{\row^j(x)\colon j\geq 0\}} \ddeg(y) = \sum_{y \in \{\row^j(x)\colon j\geq 0\}} \ddeg(\Phi(y))$ for all $x \in L_1$.
\end{itemize} 
\end{conjecture}

\begin{figure}[ht]
\begin{center}
\begin{tikzpicture}%[scale=1.5]

  \node (L3) at (0,-6) {\begin{tikzpicture}[scale=0.4]

    \draw [color=gray, dashed, line width=0.5] (1,0)--(2,0);
    \draw [color=gray, dashed, line width=0.5] (0,1)--(3,1);
    \draw [color=gray, dashed, line width=0.5] (0,2)--(3,2);
    \draw [color=gray, dashed, line width=0.5] (1,3)--(2,3);
    \draw [color=gray, dashed, line width=0.5] (0,1)--(0,2);
    \draw [color=gray, dashed, line width=0.5] (1,0)--(1,3);
    \draw [color=gray, dashed, line width=0.5] (2,0)--(2,3);
    \draw [color=gray, dashed, line width=0.5] (3,1)--(3,2);
    
    \draw (2,2) node [scale=0.35, circle, draw,fill=black,anchor=center]{};
    \draw (2,0) node [scale=0.35, circle, draw,fill=black,anchor=center]{};
    \draw (3,2) node [scale=0.35, circle, draw,fill=black,anchor=center]{};
    \draw (2,3) node [scale=0.35, circle, draw,fill=black,anchor=center]{};
    \draw (3,1) node [scale=0.35, circle, draw,fill=black,anchor=center]{};
    \draw (0,2) node [scale=0.35, circle, draw,fill=black,anchor=center]{};
    \draw (1,3) node [scale=0.35, circle, draw,fill=black,anchor=center]{};

    \draw [color=black, line width=1] (1,3)--(2,3)--(2,0);
    \draw [color=black, line width=1] (0,2)--(3,2)--(3,1);

    %\node[anchor=south west,scale=1.2] at (4.5,3) {\color{black}{$T'$}};
    \end{tikzpicture}};

    \node (L2) at (0,-3) {\begin{tikzpicture}[scale=0.4]
\draw [color=gray, dashed, line width=0.5] (1,0)--(2,0);
    \draw [color=gray, dashed, line width=0.5] (0,1)--(3,1);
    \draw [color=gray, dashed, line width=0.5] (0,2)--(3,2);
    \draw [color=gray, dashed, line width=0.5] (1,3)--(2,3);
    \draw [color=gray, dashed, line width=0.5] (0,1)--(0,2);
    \draw [color=gray, dashed, line width=0.5] (1,0)--(1,3);
    \draw [color=gray, dashed, line width=0.5] (2,0)--(2,3);
    \draw [color=gray, dashed, line width=0.5] (3,1)--(3,2);

    \draw (2,2) node [scale=0.35, circle, draw,fill=black,anchor=center]{};
    \draw (2,0) node [scale=0.35, circle, draw,fill=black,anchor=center]{};
    \draw (1,2) node [scale=0.35, circle, draw,fill=black,anchor=center]{};
    \draw (2,1) node [scale=0.35, circle, draw,fill=black,anchor=center]{};
    \draw (3,1) node [scale=0.35, circle, draw,fill=black,anchor=center]{};
    \draw (0,2) node [scale=0.35, circle, draw,fill=black,anchor=center]{};
    \draw (1,3) node [scale=0.35, circle, draw,fill=black,anchor=center]{};

    \draw [color=black, line width=1] (1,3)--(1,2)--(2,2)--(2,0);
    \draw [color=black, line width=1] (0,2)--(1,2);
    \draw [color=black, line width=1] (2,1)--(3,1);

    \end{tikzpicture}};

\node (L1) at (0,0) {\begin{tikzpicture}[scale=0.4]
\draw [color=gray, dashed, line width=0.5] (1,0)--(2,0);
    \draw [color=gray, dashed, line width=0.5] (0,1)--(3,1);
    \draw [color=gray, dashed, line width=0.5] (0,2)--(3,2);
    \draw [color=gray, dashed, line width=0.5] (1,3)--(2,3);
    \draw [color=gray, dashed, line width=0.5] (0,1)--(0,2);
    \draw [color=gray, dashed, line width=0.5] (1,0)--(1,3);
    \draw [color=gray, dashed, line width=0.5] (2,0)--(2,3);
    \draw [color=gray, dashed, line width=0.5] (3,1)--(3,2);

    \draw (2,2) node [scale=0.35, circle, draw,fill=black,anchor=center]{};
    \draw (2,0) node [scale=0.35, circle, draw,fill=black,anchor=center]{};
    \draw (1,2) node [scale=0.35, circle, draw,fill=black,anchor=center]{};
    \draw (2,1) node [scale=0.35, circle, draw,fill=black,anchor=center]{};
    \draw (3,1) node [scale=0.35, circle, draw,fill=black,anchor=center]{};
    \draw (0,2) node [scale=0.35, circle, draw,fill=black,anchor=center]{};
    \draw (1,3) node [scale=0.35, circle, draw,fill=black,anchor=center]{};

    \draw [color=black, line width=1] (1,3)--(1,2)--(2,2)--(2,0);
    \draw [color=black, line width=1] (0,2)--(1,2);
    \draw [color=black, line width=1] (2,1)--(3,1);
    \end{tikzpicture}};

    \node (R1) at (4,0) {\begin{tikzpicture}[scale=0.4]
\draw [color=gray, dashed, line width=0.5] (2,0)--(3,0);
    \draw [color=gray, dashed, line width=0.5] (0,1)--(3,1);
    \draw [color=gray, dashed, line width=0.5] (0,2)--(3,2);
    \draw [color=gray, dashed, line width=0.5] (2,3)--(3,3);
    \draw [color=gray, dashed, line width=0.5] (0,1)--(0,2);
    \draw [color=gray, dashed, line width=0.5] (1,1)--(1,2);
    \draw [color=gray, dashed, line width=0.5] (2,0)--(2,3);
    \draw [color=gray, dashed, line width=0.5] (3,0)--(3,3);

    \draw (2,2) node [scale=0.35, circle, draw,fill=black,anchor=center]{};
    \draw (3,0) node [scale=0.35, circle, draw,fill=black,anchor=center]{};
    \draw (1,2) node [scale=0.35, circle, draw,fill=black,anchor=center]{};
    \draw (2,1) node [scale=0.35, circle, draw,fill=black,anchor=center]{};
    \draw (3,1) node [scale=0.35, circle, draw,fill=black,anchor=center]{};
    \draw (0,2) node [scale=0.35, circle, draw,fill=black,anchor=center]{};
    \draw (2,3) node [scale=0.35, circle, draw,fill=black,anchor=center]{};
    
    \draw [color=black, line width=1] (0,2)--(2,2)--(2,1)--(3,1)--(3,0);
    \draw [color=black, line width=1] (2,2)--(2,3);

    \end{tikzpicture}};

        \node (R2) at (4,-3) {\begin{tikzpicture}[scale=0.4]
\draw [color=gray, dashed, line width=0.5] (2,0)--(3,0);
    \draw [color=gray, dashed, line width=0.5] (0,1)--(3,1);
    \draw [color=gray, dashed, line width=0.5] (0,2)--(3,2);
    \draw [color=gray, dashed, line width=0.5] (2,3)--(3,3);
    \draw [color=gray, dashed, line width=0.5] (0,1)--(0,2);
    \draw [color=gray, dashed, line width=0.5] (1,1)--(1,2);
    \draw [color=gray, dashed, line width=0.5] (2,0)--(2,3);
    \draw [color=gray, dashed, line width=0.5] (3,0)--(3,3);

    \draw (2,2) node [scale=0.35, circle, draw,fill=black,anchor=center]{};
    \draw (3,0) node [scale=0.35, circle, draw,fill=black,anchor=center]{};
    \draw (1,2) node [scale=0.35, circle, draw,fill=black,anchor=center]{};
    \draw (3,2) node [scale=0.35, circle, draw,fill=black,anchor=center]{};
    \draw (3,1) node [scale=0.35, circle, draw,fill=black,anchor=center]{};
    \draw (0,2) node [scale=0.35, circle, draw,fill=black,anchor=center]{};
    \draw (2,3) node [scale=0.35, circle, draw,fill=black,anchor=center]{};

    \draw [color=black, line width=1] (0,2)--(3,2)--(3,0);
    \draw [color=black, line width=1] (2,2)--(2,3);
    \end{tikzpicture}};

        \node (R3) at (4,-6) {\begin{tikzpicture}[scale=0.4]
\draw [color=gray, dashed, line width=0.5] (2,0)--(3,0);
    \draw [color=gray, dashed, line width=0.5] (0,1)--(3,1);
    \draw [color=gray, dashed, line width=0.5] (0,2)--(3,2);
    \draw [color=gray, dashed, line width=0.5] (2,3)--(3,3);
    \draw [color=gray, dashed, line width=0.5] (0,1)--(0,2);
    \draw [color=gray, dashed, line width=0.5] (1,1)--(1,2);
    \draw [color=gray, dashed, line width=0.5] (2,0)--(2,3);
    \draw [color=gray, dashed, line width=0.5] (3,0)--(3,3);

    \draw (2,2) node [scale=0.35, circle, draw,fill=black,anchor=center]{};
    \draw (3,0) node [scale=0.35, circle, draw,fill=black,anchor=center]{};
    \draw (1,2) node [scale=0.35, circle, draw,fill=black,anchor=center]{};
    \draw (3,2) node [scale=0.35, circle, draw,fill=black,anchor=center]{};
    \draw (3,1) node [scale=0.35, circle, draw,fill=black,anchor=center]{};
    \draw (0,2) node [scale=0.35, circle, draw,fill=black,anchor=center]{};
    \draw (2,3) node [scale=0.35, circle, draw,fill=black,anchor=center]{};

    \draw [color=black, line width=1] (0,2)--(3,2)--(3,0);
    \draw [color=black, line width=1] (2,2)--(2,3);
    
    \end{tikzpicture}};

\draw[->, line width=1pt] (L1) -- (L2) node[midway, left] {$\tau^{down}$};
\draw[->, line width=1pt] (L2) -- (L3) node[midway, left] {$\tau^{up}$};

\draw[->, line width=1pt] (R1) -- (R2) node[midway, left] {$\widetilde{\tau}^{up}$};
\draw[->, line width=1pt] (R2) -- (R3) node[midway, left] {$\widetilde{\tau}^{down}$};

\draw[->, line width=1pt] (L1) -- (R1) node[midway, above] {$\text{flush}$};
\draw[->, line width=1pt, gray,dashed] (L3) -- (R3) node[midway, above] {$\cancel{\text{flush}}$};

\end{tikzpicture}
\end{center}
\caption{This diagram does not commute for cross Tamari lattices.}
\label{fig:example_cross_no_commute}
\end{figure}
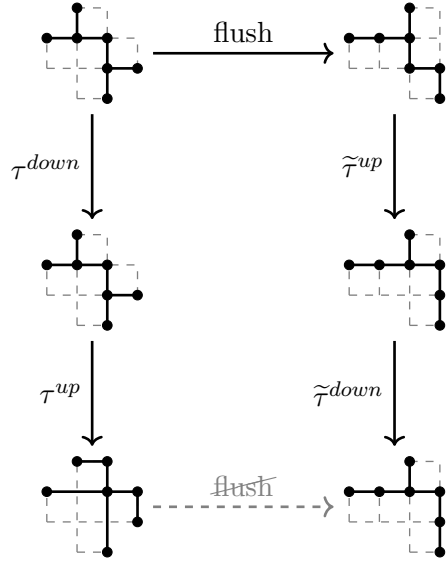

\begin{figure}[ht]
\begin{center}
\input{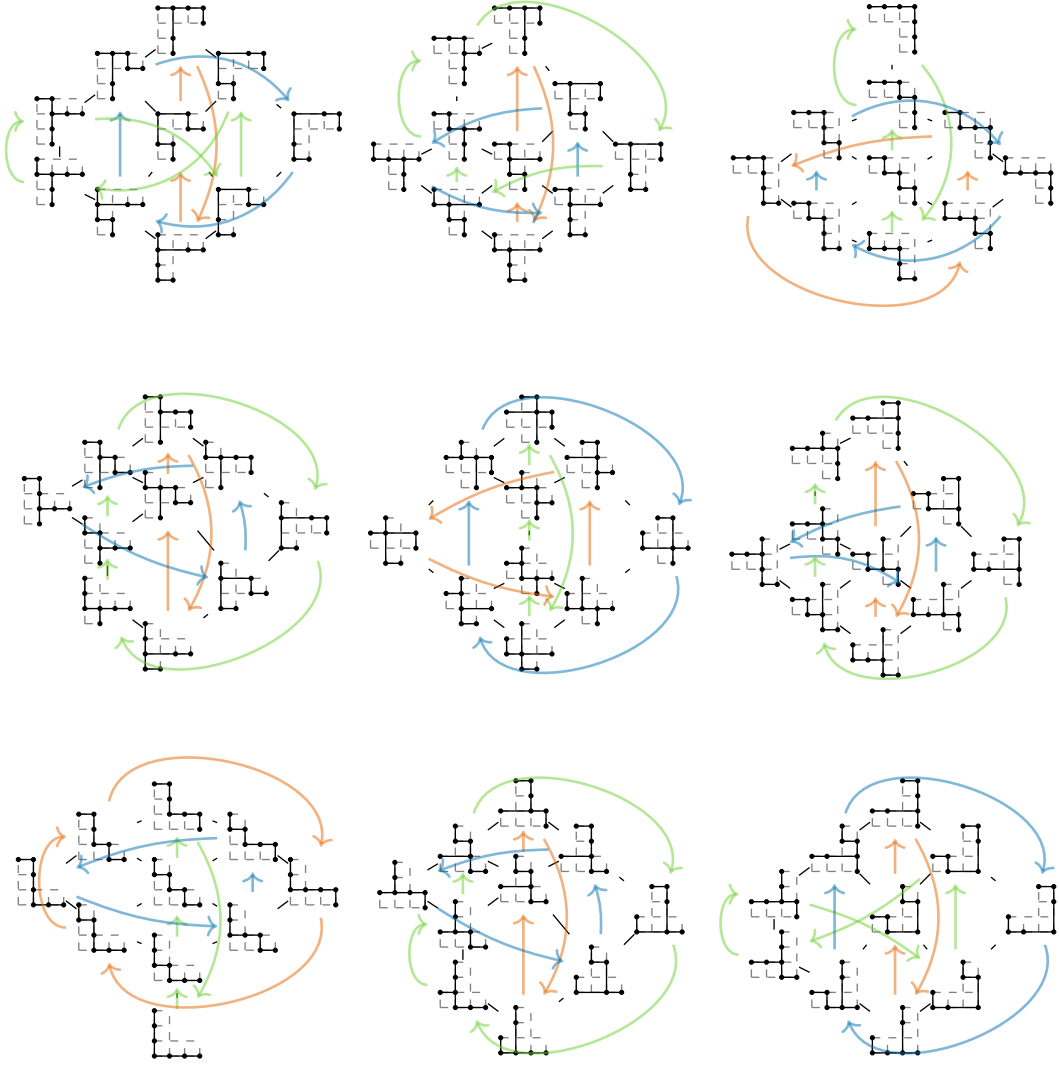}
\end{center}
\caption{Rowmotion orbits for all cross Tamari lattices corresponding to moon polyominoes that are permutations of the Young diagram $(3,1,1)$.}
\label{fig:example_cross}
\end{figure}

\begin{example}
For moon polyominoes that are row-and-column permutations of the Young diagram shape $(3,1,1)$, the associated cross Tamari lattices and their rowmotion orbits are depicted in~\cref{fig:example_cross}. As illustrated, these lattices all possess an identical distribution of orbit sizes and the same sums of down-degrees for their corresponding orbits.
\end{example}

\begin{remark}
Unlike the class of alt~$\nu$-Tamari lattices, the class of cross Tamari lattices is closed under duality: the dual of the cross Tamari lattice corresponding to a given moon polyomino is the one corresponding to its $180^\circ$ rotation. It is clear that a (say, semidistributive) lattice $L$ and its dual $L^*$ have the same behavior of rowmotion. For this reason, the cross Tamari setting would be a more conceptually satisfying context in which to establish invariance of rowmotion results.
\end{remark}

\subsection{Limits to the invariance of rowmotion}

While we conjecture that our main theorem can be extended to cross Tamari lattices, we also know that it cannot be further extended to other natural generalizations of the cross Tamari lattices, including the chute move lattices~\cite{rubey2012maximal,axelrodfreed2025chute,billey2025lattice} and the framing lattices~\cite{vonbell2025framing}.

The \dfn{chute move lattices}~\cite{rubey2012maximal,axelrodfreed2025chute,billey2025lattice} depend on a choice of moon polyomino shape and positive integer parameter $r$. The elements of a chute move lattice are maximal 0-1-fillings of the moon polyomino such that at most $r$ of the $1$ entries are mutually incompatible, i.e., form a chain of boxes strictly increasing in the northeast direction whose smallest containing rectangle lies entirely in the polyomino. The covering relation in a chute move lattice is analogous to the covering relation in a cross Tamari lattice, and in fact the cross Tamari lattices are exactly the chute move lattices with parameter $r=1$. All chute move lattices are semidistributive lattices, hence carry an action of rowmotion. Extending the invariance of rowmotion for cross Tamari lattices conjectured in~\cref{conj:cross_tamari}, one could hope that, for each fixed $r$, chute move lattices associated to moon polyominoes related by permutations of rows and columns have the same behavior of rowmotion. (They always have the same number of elements~\cite{rubey2011increasing,rubey2012maximal}, so this is a reasonable hope.) But this is not the case for $r>1$.

Indeed, letting $F$ be the Young diagram for the staircase partition $(n,n-1,\ldots,1)$, the elements of the chute move lattice associated to $F$ can be interpreted as \dfn{$r$-triangulations} of the $(n+1)$-gon, i.e., maximal collections of diagonals in this polygon such that at most~$r$ diagonals cross mutually. On the other hand, reflecting $F$ across a vertical line, we obtain a shape $F'$ whose corresponding chute move lattice has elements that can be interpreted as \dfn{$r$-fans of nested Dyck paths} of semilength~$(n+1)-2r$. See, e.g., \cite[Fig.~2]{rubey2012maximal} for this comparison. Already for~$r=2$ and~$n=6$, the orbit structures of rowmotion on these two $14$-element chute move lattices differ. The orbit sizes for $2$-triangulations of the $7$-gon are $3,5,6$, whereas the orbit sizes for pairs of nested Dyck paths of semilength $3$ are $2,4,8$; see~\cref{fig:chute}.

\begin{figure}
\begin{center}
\scalebox{0.5}{%
\begin{tikzpicture}[
    baseline=(n0.center),
    latticeNode/.style={draw, rounded corners, fill=gray!10, inner sep=5pt, font=\small\ttfamily}
]

    \node[latticeNode] (n0) at (0,0) {$(1,4), (1,5), (2,5), (2,6)$};
    \node[latticeNode] (n1) at (-2,2) {$(1,4), (1,5) (2,6), (3,6)$};
    \node[latticeNode] (n2) at (3,3) {$(1,4), (1,5), (2,5), (4,7)$};
    \node[latticeNode] (n3) at (-6.5,4) {$(1,5), (2,5), (2,6), (3,6)$};
    \node[latticeNode] (n4) at (-2,4) {$(1,4), (1,5), (3,6), (3,7)$};
    \node[latticeNode] (n5) at (5,5) {$(1,4), (2,5), (2,6), (4,7)$};
    \node[latticeNode] (n6) at (-6.5,6) {$(1,5), (2,5), (3,6), (3,7)$};
    \node[latticeNode] (n7) at (0,6) {$(1,4), (1,5), (3,7), (4,7)$};
    \node[latticeNode] (n8) at (5,7) {$(1,4), (2,6), (3,6), (4,7)$};
    \node[latticeNode] (n9) at (-6.5,8) {$(2,5), (2,6), (3,6), (3,7)$};
    \node[latticeNode] (n10) at (-2,8) {$(1,5), (2,5), (3,7), (4,7)$};
    \node[latticeNode] (n11) at (3,9) {$(1,4), (3,6), (3,7), (4,7)$};
    \node[latticeNode] (n12) at (-2,10) {$(2,5), (2,6), (3,7), (4,7)$};
    \node[latticeNode] (n13) at (0,12) {$(2,6), (3,6), (3,7), (4,7)$};

    % Black straight edges
    \draw [color=black, line width=1] (n0) -- (n1);
    \draw [color=black, line width=1] (n0) -- (n2);
    \draw [color=black, line width=1] (n2) -- (n7);
    \draw [color=black, line width=1] (n2) -- (n5);
    \draw [color=black, line width=1] (n1) -- (n3);
    \draw [color=black, line width=1] (n1) -- (n4);
    \draw [color=black, line width=1] (n3) -- (n6);
    \draw [color=black, line width=1] (n4) -- (n6);
    \draw [color=black, line width=1] (n4) -- (n7);
    \draw [color=black, line width=1] (n5) -- (n8);
    \draw [color=black, line width=1] (n8) -- (n11);
    \draw [color=black, line width=1] (n11) -- (n13);
    \draw [color=black, line width=1] (n12) -- (n13);
    \draw [color=black, line width=1] (n7) -- (n10);
    \draw [color=black, line width=1] (n10) -- (n12);
    \draw [color=black, line width=1] (n9) -- (n12);
    \draw [color=black, line width=1] (n6) -- (n9);
    \draw [color=black, line width=1] (n6) -- (n10);
    \draw [color=black, line width=1] (n7) -- (n11);

    % Curved colored arrows
    \draw[->, line width=3pt, blue, opacity=0.5] (n0) to[bend left=20]  (n7);
    \draw[->, line width=3pt, blue, opacity=0.5] (n7) to[bend right=20]  (n13);
    \draw[->, line width=3pt, blue, opacity=0.5] (n13) to[bend left=20]  (n0);
    
    \draw[->, line width=3pt, green, opacity=0.5] (n1) to[bend left=20]  (n6);
    \draw[->, line width=3pt, green, opacity=0.5] (n6) to[bend left=20]  (n12);
    \draw[->, line width=3pt, green, opacity=0.5] (n12) to[bend right=30]  (n5);
    \draw[->, line width=3pt, green, opacity=0.5] (n5) to[bend right=40]  (n8);
    \draw[->, line width=3pt, green, opacity=0.5] (n8) to[bend right=40]  (n1);

    \draw[->, line width=3pt, red, opacity=0.5] (n2) to[bend right=20]  (n11);
    \draw[->, line width=3pt, red, opacity=0.5] (n11) to[bend right=20]  (n3);
    \draw[->, line width=3pt, red, opacity=0.5] (n3) to[bend right=20]  (n4);
    \draw[->, line width=3pt, red, opacity=0.5] (n4) to[bend right=5]  (n10);
    \draw[->, line width=3pt, red, opacity=0.5] (n10) to[bend right=20]  (n9);
    \draw[->, line width=3pt, red, opacity=0.5] (n9) to[bend right=20]  (n2);

\end{tikzpicture}%
\hspace{3cm}%
\begin{tikzpicture}[
    baseline=(AA.center),
    scale=1.2,
    yscale=1.1,
    node distance=1.5cm and 2cm,
    latticeNode/.style={draw, rounded corners, fill=gray!10, inner sep=5pt, font=\small\ttfamily},
    edgeStyle/.style={thick}
]
    % Rank 0 (Bottom)
    \node[latticeNode] (AA) at (0, 0) {(UDUDUD, UDUDUD)};
    
    % Rank 1
    \node[latticeNode] (AB) at (-2.5, 1.5) {(UDUDUD, UDUUDD)};
    \node[latticeNode] (AC) at (2.5, 1.5) {(UDUDUD, UUDDUD)};
    
    % Rank 2
    \node[latticeNode] (BB) at (-5, 3) {(UDUUDD, UDUUDD)};
    \node[latticeNode] (AD) at (0, 3) {(UDUDUD, UUDUDD)};
    \node[latticeNode] (CC) at (5, 3) {(UUDDUD, UUDDUD)};
    
    % Rank 3
    \node[latticeNode] (BD) at (-3, 4.5) {(UDUUDD, UUDUDD)};
    \node[latticeNode] (AE) at (0, 4.5) {(UDUDUD, UUUDDD)};
    \node[latticeNode] (CD) at (3, 4.5) {(UUDDUD, UUDUDD)};
    
    % Rank 4
    \node[latticeNode] (BE) at (-3, 6) {(UDUUDD, UUUDDD)};
    \node[latticeNode] (DD) at (0, 6) {(UUDUDD, UUDUDD)};
    \node[latticeNode] (CE) at (3, 6) {(UUDDUD, UUUDDD)};
    
    % Rank 5
    \node[latticeNode] (DE) at (0, 7.5) {(UUDUDD, UUUDDD)};
    
    % Rank 6 (Top)
    \node[latticeNode] (EE) at (0, 9) {(UUUDDD, UUUDDD)};
    
    % Edges (Bottom to Top)
    \begin{scope}[every path/.style={edgeStyle}]
        % Rank 0 to Rank 1
        \draw (AA) -- (AB);
        \draw (AA) -- (AC);
        
        % Rank 1 to Rank 2
        \draw (AB) -- (BB);
        \draw (AB) -- (AD);
        \draw (AC) -- (CC);
        \draw (AC) -- (AD);
        
        % Rank 2 to Rank 3
        \draw (BB) -- (BD);
        \draw (CC) -- (CD);
        \draw (AD) -- (BD);
        \draw (AD) -- (CD);
        \draw (AD) -- (AE);
        
        % Rank 3 to Rank 4
        \draw (BD) -- (DD);
        \draw (BD) -- (BE);
        \draw (CD) -- (DD);
        \draw (CD) -- (CE);
        \draw (AE) -- (BE);
        \draw (AE) -- (CE);
        
        % Rank 4 to Rank 5
        \draw (DD) -- (DE);
        \draw (BE) -- (DE);
        \draw (CE) -- (DE);
        
        % Rank 5 to Rank 6
        \draw (DE) -- (EE);
    \end{scope}

    \draw[->, line width=3pt, blue, opacity=0.5] (AE) to[bend left=20]  (DD);
    \draw[->, line width=3pt, blue, opacity=0.5] (DD) to[bend left=20]  (AE);
    
    \draw[->, line width=3pt, red, opacity=0.5] (AA) to[bend left=20]  (AD);
    \draw[->, line width=3pt, red, opacity=0.5] (AD) to[bend right=20]  (DE);
    \draw[->, line width=3pt, red, opacity=0.5] (DE) to[bend left=20]  (EE);
    \draw[->, line width=3pt, red, opacity=0.5] (EE) to[bend left=30]  (AA);

    \draw[->, line width=3pt, green, opacity=0.5] (AB) to[bend left=30]  (BD);
    \draw[->, line width=3pt, green, opacity=0.5] (BD) to[bend left=30]  (CE);
    \draw[->, line width=3pt, green, opacity=0.5] (CE) to[bend left=30]  (BB);
    \draw[->, line width=3pt, green, opacity=0.5] (BB) to[bend left=30]  (AC);
    \draw[->, line width=3pt, green, opacity=0.5] (AC) to[bend right=30]  (CD);
    \draw[->, line width=3pt, green, opacity=0.5] (CD) to[bend left=30]  (BE);
    \draw[->, line width=3pt, green, opacity=0.5] (BE) to[bend left=90]  (CC);
    \draw[->, line width=3pt, green, opacity=0.5] (CC) to[bend left=90]  (AB);

\end{tikzpicture}%
}
\end{center}
\caption{Two chute move lattices with different rowmotion behavior. Left: $2$-triangulations of the~$7$-gon (indexed by their 2-relevant diagonals). Right: pairs of nested Dyck paths of semilength $3$ (indexed by their sequence of Up and Down steps).}
\label{fig:chute}
\end{figure}
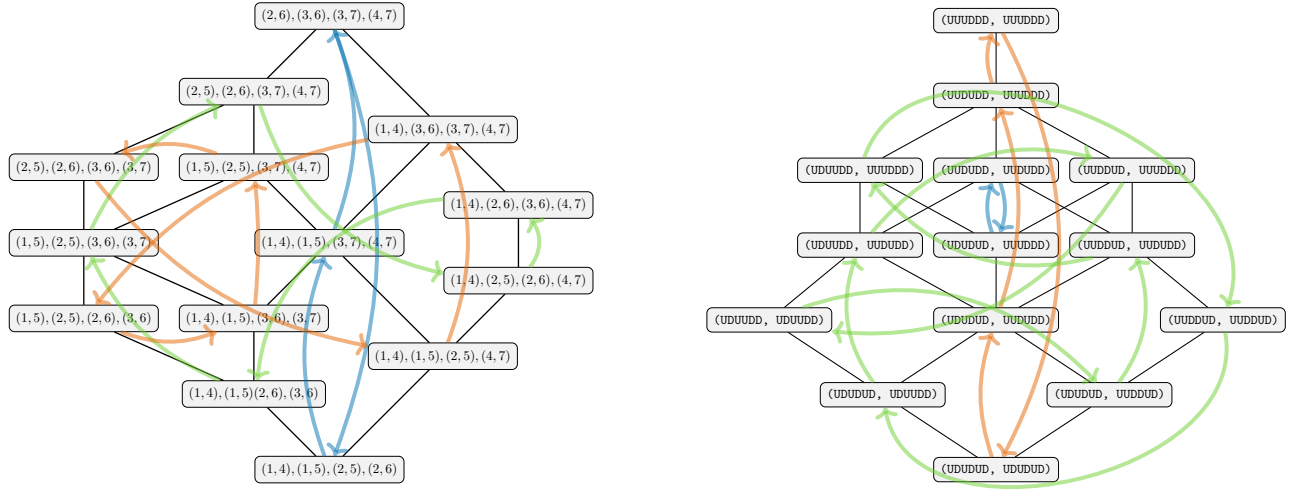

Meanwhile, \dfn{framing lattices}~\cite{vonbell2025framing} are certain lattices that depend on the choice of a directed acyclic graph and a \dfn{framing}, which is a local ordering of the edges incident to each vertex. The Hasse diagram of a framing lattice is dual to the corresponding \dfn{framing triangulation} of the \dfn{flow polytope} associated to the graph. The definition of framing lattice is somewhat involved: the elements are \dfn{maximal coherent collections of routes} in the framed graph, and the cover relations can be thought of as a kind of rotation on these. See~\cite{vonbell2025framing} for precise definitions and details.

In~\cite[\S2.4]{vonbell2025framing}, von Bell and Ceballos show that cross Tamari lattices are framing lattices associated to certain graphs and framings. Moreover, they show that the cross Tamari lattices associated to moon polyominoes that are permutations of one another are framing lattices whose associated graphs are the same, but with different framings. All framing lattices are semidistributive, hence carry an action of rowmotion. Extending the invariance of rowmotion for cross Tamari lattices conjectured in~\cref{conj:cross_tamari}, one could hope that, in general, two framing lattices associated to the same graph, but with different framings, have the same behavior of rowmotion. (They always have the same number of elements~\cite{vonbell2025framing}, so this is a reasonable hope.) But this is not the case.

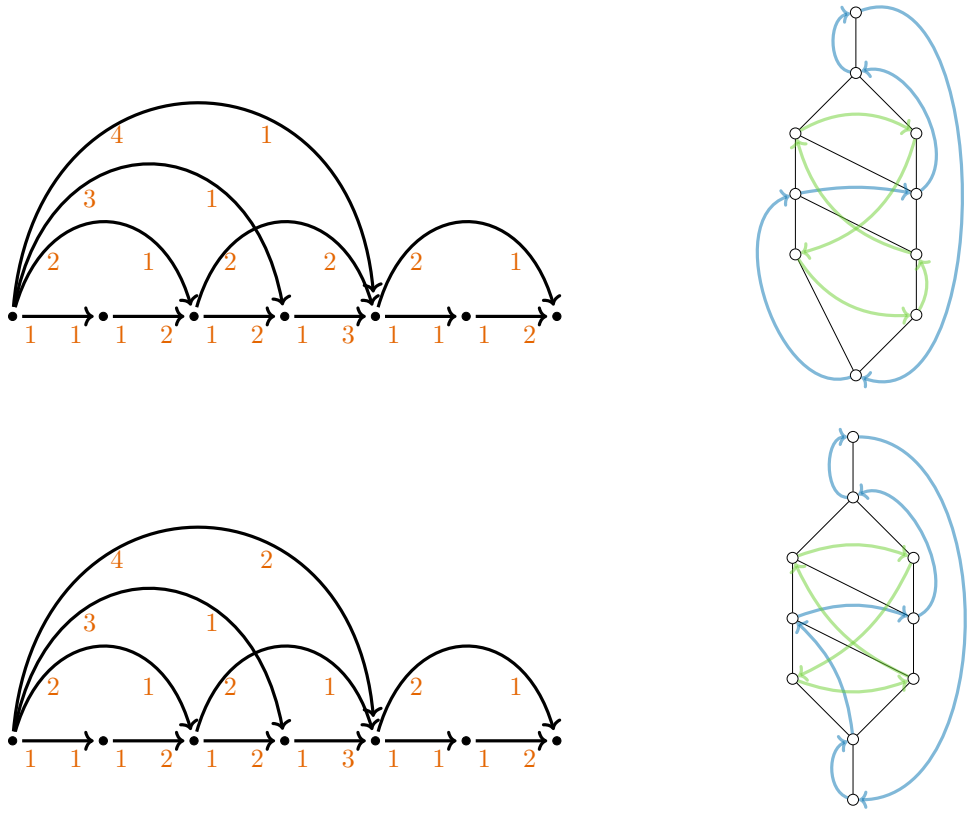
\begin{figure}
\begin{center}
\begin{tabular}{c @{\hspace{2cm}} c} % Adjust @{\hspace{4cm}} to change overall space between left and right columns
    % --- FIRST ROW ---
    \scalebox{0.6}{%
    \begin{tikzpicture}[yscale=2, baseline=(current bounding box.center)]
        \foreach \x/\y in {0/0, 2/0, 4/0, 6/0, 8/0, 10/0, 12/0} {
            \node[scale=.5, circle, draw, fill=black, anchor=center] at (\x,\y) {};
        }

        \draw[->, line width=2pt, black, shorten <=6pt, shorten >=6pt] (0,0) -- (2,0);
        \draw[->, line width=2pt, black, shorten <=6pt, shorten >=6pt] (2,0) -- (4,0);
        \draw[->, line width=2pt, black, shorten <=6pt, shorten >=6pt] (4,0) -- (6,0);
        \draw[->, line width=2pt, black, shorten <=6pt, shorten >=6pt] (6,0) -- (8,0);
        \draw[->, line width=2pt, black, shorten <=6pt, shorten >=6pt] (8,0) -- (10,0);
        \draw[->, line width=2pt, black, shorten <=6pt, shorten >=6pt] (10,0) -- (12,0);

        \draw[->, line width=2pt, black, shorten <=6pt, shorten >=7pt] (0,0) to[bend left=60] (4,0);
        \draw[->, line width=2pt, black, shorten <=6pt, shorten >=7pt] (4,0) to[bend left=60] (8,0);
        \draw[->, line width=2pt, black, shorten <=6pt, shorten >=7pt] (8,0) to[bend left=60] (12,0);

        \draw[->, line width=2pt, black, shorten <=6pt, shorten >=7pt] (0,0) to[bend left=70] (6,0);
        \draw[->, line width=2pt, black, shorten <=6pt, shorten >=15pt] (0,0) to[bend left=80] (8,0);

        \node[color=red, scale=1.5] at (.4,-.2) {$1$};
        \node[color=red, scale=1.5] at (1.4,-.2) {$1$};
        \node[color=red, scale=1.5] at (2.4,-.2) {$1$};
        \node[color=red, scale=1.5] at (3.4,-.2) {$2$};
        \node[color=red, scale=1.5] at (4.4,-.2) {$1$};
        \node[color=red, scale=1.5] at (5.4,-.2) {$2$};
        \node[color=red, scale=1.5] at (6.4,-.2) {$1$};
        \node[color=red, scale=1.5] at (7.4,-.2) {$3$};
        \node[color=red, scale=1.5] at (8.4,-.2) {$1$};
        \node[color=red, scale=1.5] at (9.4,-.2) {$1$};
        \node[color=red, scale=1.5] at (10.4,-.2) {$1$};
        \node[color=red, scale=1.5] at (11.4,-.2) {$2$};
        \node[color=red, scale=1.5] at (.9,.6) {$2$};
        \node[color=red, scale=1.5] at (3,.6) {$1$};
        \node[color=red, scale=1.5] at (4.8,.6) {$2$};
        \node[color=red, scale=1.5] at (11.1,.6) {$1$};
        \node[color=red, scale=1.5] at (8.9,.6) {$2$};
        \node[color=red, scale=1.5] at (1.7,1.3) {$3$};
        \node[color=red, scale=1.5] at (4.4,1.3) {$1$};
        \node[color=red, scale=1.5] at (2.3,2) {$4$};

        \node[color=red, scale=1.5] at (5.6,2) {$1$};
        \node[color=red, scale=1.5] at (7,.6) {$2$};
    \end{tikzpicture}%
    } 
    & 
    \scalebox{0.4}{%
    \begin{tikzpicture}[baseline=(current bounding box.center)]
        \node[draw,circle,fill=white] (n0) at (0,0) {};
        \node[draw,circle,fill=white] (n6) at (2,2) {};
        \node[draw,circle,fill=white] (n5) at (2,4) {};
        \node[draw,circle,fill=white] (n4) at (2,6) {};
        \node[draw,circle,fill=white] (n2) at (2,8) {};
        \node[draw,circle,fill=white] (n3) at (-2,4) {};
        \node[draw,circle,fill=white] (n7) at (-2,6) {};
        \node[draw,circle,fill=white] (n8) at (-2,8) {};
        \node[draw,circle,fill=white] (n9) at (0,10) {};
        \node[draw,circle,fill=white] (n1) at (0,12) {};

        \draw (n0) -- (n6);
        \draw (n6) -- (n5);
        \draw (n5) -- (n4);
        \draw (n4) -- (n2);
        \draw (n2) -- (n9);
        \draw (n9) -- (n1);
        \draw (n0) -- (n3);
        \draw (n3) -- (n7);
        \draw (n7) -- (n8);
        \draw (n8) -- (n9);
        \draw (n5) -- (n7);
        \draw (n4) -- (n8);

        \draw[->, line width=3pt, blue, opacity=0.5] (n0) to[bend left=90]  (n7);
        \draw[->, line width=3pt, blue, opacity=0.5] (n7) to[bend left=10]  (n4);
        \draw[->, line width=3pt, blue, opacity=0.5] (n4) to[bend right=80] (n9);
        \draw[->, line width=3pt, blue, opacity=0.5] (n9) to[bend left=80]  (n1);
        \draw[->, line width=3pt, blue, opacity=0.5] (n1) to[bend left=110] (n0);

        \draw[->, line width=3pt, green, opacity=0.5] (n6) to[bend right=30] (n5);
        \draw[->, line width=3pt, green, opacity=0.5] (n5) to[bend left=30]  (n8);
        \draw[->, line width=3pt, green, opacity=0.5] (n8) to[bend left=30]  (n2);
        \draw[->, line width=3pt, green, opacity=0.5] (n2) to[bend left=30]  (n3);
        \draw[->, line width=3pt, green, opacity=0.5] (n3) to[bend right=30] (n6);
    \end{tikzpicture}%
    } \\[1cm] 
    \scalebox{0.6}{%
    \begin{tikzpicture}[yscale=2, baseline=(current bounding box.center)]
        \foreach \x/\y in {0/0, 2/0, 4/0, 6/0, 8/0, 10/0, 12/0} {
            \node[scale=.5, circle, draw, fill=black, anchor=center] at (\x,\y) {};
        }

        \draw[->, line width=2pt, black, shorten <=6pt, shorten >=6pt] (0,0) -- (2,0);
        \draw[->, line width=2pt, black, shorten <=6pt, shorten >=6pt] (2,0) -- (4,0);
        \draw[->, line width=2pt, black, shorten <=6pt, shorten >=6pt] (4,0) -- (6,0);
        \draw[->, line width=2pt, black, shorten <=6pt, shorten >=6pt] (6,0) -- (8,0);
        \draw[->, line width=2pt, black, shorten <=6pt, shorten >=6pt] (8,0) -- (10,0);
        \draw[->, line width=2pt, black, shorten <=6pt, shorten >=6pt] (10,0) -- (12,0);

        \draw[->, line width=2pt, black, shorten <=6pt, shorten >=7pt] (0,0) to[bend left=60] (4,0);
        \draw[->, line width=2pt, black, shorten <=6pt, shorten >=7pt] (4,0) to[bend left=60] (8,0);
        \draw[->, line width=2pt, black, shorten <=6pt, shorten >=7pt] (8,0) to[bend left=60] (12,0);

        \draw[->, line width=2pt, black, shorten <=6pt, shorten >=7pt] (0,0) to[bend left=70] (6,0);
        \draw[->, line width=2pt, black, shorten <=6pt, shorten >=15pt] (0,0) to[bend left=80] (8,0);

        \node[color=red, scale=1.5] at (.4,-.2) {$1$};
        \node[color=red, scale=1.5] at (1.4,-.2) {$1$};
        \node[color=red, scale=1.5] at (2.4,-.2) {$1$};
        \node[color=red, scale=1.5] at (3.4,-.2) {$2$};
        \node[color=red, scale=1.5] at (4.4,-.2) {$1$};
        \node[color=red, scale=1.5] at (5.4,-.2) {$2$};
        \node[color=red, scale=1.5] at (6.4,-.2) {$1$};
        \node[color=red, scale=1.5] at (7.4,-.2) {$3$};
        \node[color=red, scale=1.5] at (8.4,-.2) {$1$};
        \node[color=red, scale=1.5] at (9.4,-.2) {$1$};
        \node[color=red, scale=1.5] at (10.4,-.2) {$1$};
        \node[color=red, scale=1.5] at (11.4,-.2) {$2$};

        \node[color=red, scale=1.5] at (.9,.6) {$2$};
        \node[color=red, scale=1.5] at (3,.6) {$1$};
        \node[color=red, scale=1.5] at (4.8,.6) {$2$};
        \node[color=red, scale=1.5] at (7,.6) {$1$};
        \node[color=red, scale=1.5] at (11.1,.6) {$1$};
        \node[color=red, scale=1.5] at (8.9,.6) {$2$};

        \node[color=red, scale=1.5] at (1.7,1.3) {$3$};
        \node[color=red, scale=1.5] at (4.4,1.3) {$1$};

        \node[color=red, scale=1.5] at (2.3,2) {$4$};
        \node[color=red, scale=1.5] at (5.6,2) {$2$};
    \end{tikzpicture}%
    } 
    & 
    \hspace*{.8cm}
    \scalebox{0.4}{
    \begin{tikzpicture}[baseline=(current bounding box.center)]
        \node[draw,circle,fill=white] (n0) at (0,0) {};
        \node[draw,circle,fill=white] (n6) at (0,2) {};
        \node[draw,circle,fill=white] (n5) at (2,4) {};
        \node[draw,circle,fill=white] (n4) at (2,6) {};
        \node[draw,circle,fill=white] (n2) at (2,8) {};
        \node[draw,circle,fill=white] (n3) at (-2,4) {};
        \node[draw,circle,fill=white] (n7) at (-2,6) {};
        \node[draw,circle,fill=white] (n8) at (-2,8) {};
        \node[draw,circle,fill=white] (n9) at (0,10) {};
        \node[draw,circle,fill=white] (n1) at (0,12) {};
        
        \draw (n0) -- (n6);
        \draw (n6) -- (n5);
        \draw (n5) -- (n4);
        \draw (n4) -- (n2);
        \draw (n2) -- (n9);
        \draw (n9) -- (n1);
        \draw (n6) -- (n3);
        \draw (n3) -- (n7);
        \draw (n7) -- (n8);
        \draw (n8) -- (n9);
        \draw (n5) -- (n7);
        \draw (n4) -- (n8);
        
        \draw[->, line width=3pt, blue, opacity=0.5] (n0) to[bend left=70]  (n6);
        \draw[->, line width=3pt, blue, opacity=0.5] (n6) to[bend right=20] (n7);
        \draw[->, line width=3pt, blue, opacity=0.5] (n7) to[bend left=20]  (n4);
        \draw[->, line width=3pt, blue, opacity=0.5] (n4) to[bend right=90] (n9);
        \draw[->, line width=3pt, blue, opacity=0.5] (n9) to[bend left=90]  (n1);
        \draw[->, line width=3pt, blue, opacity=0.5] (n1) to[bend left=90]  (n0);

        \draw[->, line width=3pt, green, opacity=0.5] (n5) to[bend left=20] (n8);
        \draw[->, line width=3pt, green, opacity=0.5] (n8) to[bend left=20] (n2);
        \draw[->, line width=3pt, green, opacity=0.5] (n2) to[bend left=20] (n3);
        \draw[->, line width=3pt, green, opacity=0.5] (n3) to[bend right=20] (n5);
    \end{tikzpicture}%
    }
\end{tabular}
\end{center}
\caption{A directed graph that produces two framing lattices with different rowmotion behavior.}
\label{fig:framing_lattices}
\end{figure}

Indeed, consider the directed graph depicted in \cref{fig:framing_lattices}. Two different framings of this graph produce the two different framing lattices depicted there. Like all framing lattices, these lattices are semidistributive (and in fact, they are both extremal -- and the one on the bottom is distributive). However, they have different behavior of rowmotion: the lattice on the top has rowmotion orbits of size~$5$ and~$5$; the one on the bottom has orbits of size~$6$ and~$4$.

\begin{remark}
For a semidistributive lattice $L$, we recall that its (undirected) \dfn{Galois graph} can be defined as the complement graph of the one skeleton of its canonical join complex;  see~\cite{barnard2019canonical, thomas2019rowmotion, thomas2019independence}. In the case where $L=J(P)$ is distributive, then its Galois graph is the \dfn{comparability graph} of $P$. As mentioned in \cref{sec:intro}, in~\cite[Proposition~4.10]{hopkins2022minuscule} it was shown that if $P$ and $Q$ have isomorphic comparability graphs, then $J(P)$ and $J(Q)$ have the same rowmotion behavior. As an extension of this, one might optimistically hope that if two semidistributive lattices $L$ and $L'$ have isomorphic Galois graphs, then they have the same rowmotion behavior.\footnote{But note that already $\Stan{4}$ and $\Tam{4}$ have \emph{non}-isomorphic Galois graphs; so this optimistic hope is anyway orthogonal to our main interests in this paper.} (Again, semidistributive lattices with isomorphic Galois graphs have the same number of elements, so this is a reasonable hope.) However, this is not true: for example, the two semidistributive lattices in \cref{fig:framing_lattices} have isomorphic Galois graphs, but different rowmotion behavior.
\end{remark}

\subsection{Combinatorial models for rowmotion}

As established in~\cite{armstrong2013uniform}, rowmotion on the Stanley lattice is in equivariant bijection with the Kreweras complement of noncrossing partitions. Therefore, we could ask for a combinatorial model for rowmotion on alt~$\nu$-Tamari lattices akin to Kreweras complement of something like noncrossing partitions. Since the lattice of noncrossing partitions is the \dfn{core label order} of the Tamari lattice~\cite{reading2011noncrossing,muhle2019core}, an approach for a combinatorial model for rowmotion on alt~$\nu$-Tamari lattices is to consider the core label order of the $\nu$-Tamari lattice.

In \cite{ceballos2026snoncrossing}, Ceballos and M\"uller introduce a $\nu$-analogue of noncrossing partitions, called $s$-noncrossing partitions, where $s=(s_1,\dots,s_n)$ is a sequence of positive integers. The poset of $s$-noncrossing partitions is isomorphic to the core label order of the $\nu$-Tamari lattice for $\nu=\nu(s)=NE^{s_1} \dots NE^{s_n}$. Therefore, we consider how rowmotion acting on the $\nu$-Tamari lattice behaves in terms of $s$-noncrossing partitions. 

\begin{definition}[\cite{ceballos2026snoncrossing}]
For $\nu=\nu(s)=NE^{s_1} \dots NE^{s_n}$, label the east steps of $F_\nu$ from $1$ to~$|s| \coloneqq (s_1+\dots + s_n)$ and call any index following a north step a first column index. In \cref{fig:descent_tree_coloring} the first column indices are marked by a circle. An~\dfn{$s$-noncrossing partition} is a noncrossing partition~$\alpha = \{\alpha_1, \dots, \alpha_r\}$ of $\{1,\dots,|s|\}$ such that for every block $\alpha_i = \{a_1 < a_2 < \cdots < a_{|\alpha_i|}\}$, the elements $a_2, \dots, a_{|\alpha_i|}$ are all first column indices, while no restriction is placed on the minimal element $a_1$.
\end{definition}

In~\cite{ceballos2026snoncrossing}, the authors define a bijection between $\nu$-trees and $s$-noncrossing partitions as follows. Let $T$ be a $\nu$-tree. Removing all nondescent nodes from $T$, different than the root, splits the graph of $T$ into disconnected components $\widetilde{T_1},...,\widetilde{T_l}$. Let $T_i$ be the closure of~$\widetilde{T_i}$ (include the nodes of the boundary of $\widetilde{T_i}$). The $s$-noncrossing partition associated to~$T$ has one block for each subtree~$T_i$, consisting of the rightmost column indices of its horizontal edges. This is illustrated in~\cref{fig:descent_tree_coloring}.

\begin{figure}[ht]
\begin{center}
    \begin{tikzpicture}[scale=0.5]
    \draw [color=gray, line width=1] (0,-1)--(0,0)--(3,0)--(3,1)--(7,1)--(7,2)--(10,2)--(10,3)--(13,3)--(13,4)--(15,4)--(15,5)--(18,5)--(18,6)--(19,6)--(19,7)--(20,7);
    \draw [color=gray, dotted, line width=1] (0,0)--(0,6)--(18,6);

    \draw [color=gray, dotted, line width=1] (3,1)--(3,7);
    \draw [color=gray, dotted, line width=1] (7,2)--(7,7);
    \draw [color=gray, dotted, line width=1] (10,3)--(10,7);
    \draw [color=gray, dotted, line width=1] (13,4)--(13,7);
    \draw [color=gray, dotted, line width=1] (15,5)--(15,7);

    \node[anchor=south west,scale=0.8, circle, draw] (n1) at (0,7.3) {\color{black}{$1$}};
    \node[anchor=south west,scale=0.8] (n2) at (1,7.3) {\color{black}{$2$}};
    \node[anchor=south west,scale=0.8] (n3) at (2,7.3) {\color{black}{$3$}};
    \node[anchor=south west,scale=0.8, circle, draw] (n4) at (3,7.3) {\color{black}{$4$}};
    \node[anchor=south west,scale=0.8] (n5) at (4,7.3) {\color{black}{$5$}};
    \node[anchor=south west,scale=0.8] (n6) at (5,7.3) {\color{black}{$6$}};
    \node[anchor=south west,scale=0.8] (n7) at (6,7.3) {\color{black}{$7$}};
    \node[anchor=south west,scale=0.8, circle, draw] (n8) at (7,7.3) {\color{black}{$8$}};
    \node[anchor=south west,scale=0.8] (n9) at (8,7.3) {\color{black}{$9$}};
    \node[anchor=south west,scale=0.8] (n10) at (9,7.3) {\color{black}{$10$}};
    \node[anchor=south west,scale=0.7, circle, draw] (n11) at (10.2,7.3) {\color{black}{$11$}};
    \node[anchor=south west,scale=0.8] (n12) at (11.3,7.3) {\color{black}{$12$}};
    \node[anchor=south west,scale=0.8] (n13) at (12,7.3) {\color{black}{$13$}};
    \node[anchor=south west,scale=0.7, circle, draw] (n14) at (13.2,7.3) {\color{black}{$14$}};
    \node[anchor=south west,scale=0.8] (n15) at (14,7.3) {\color{black}{$15$}};
    \node[anchor=south west,scale=0.7, circle, draw] (n16) at (15.2,7.3) {\color{black}{$16$}};
    \node[anchor=south west,scale=0.8] (n17) at (16,7.3) {\color{black}{$17$}};
    \node[anchor=south west,scale=0.8] (n18) at (17,7.3) {\color{black}{$18$}};
    \node[anchor=south west,scale=0.7, circle, draw] (n19) at (18.1,7.3) {\color{black}{$19$}};
    \node[anchor=south west,scale=0.7, circle, draw] (n20) at (19.3,7.3) {\color{black}{$20$}};

    \node[circle, fill=black, inner sep=1.5pt] (n28) at (20,7) {};
    \node[circle, fill=black, inner sep=1.5pt] (n0) at (0,7) {};
    \node[circle, fill=black, inner sep=1.5pt] (n1) at (0,-1) {};
    \node[circle, fill=black, inner sep=1.5pt] (n2) at (0,6) {};
    \node[circle, fill=black, inner sep=1.5pt] (n3) at (16,6) {};
    \node[circle, fill=black, inner sep=1.5pt] (n4) at (16,5) {};
    \node[circle, fill=black, inner sep=1.5pt] (n5) at (17,5) {};
    \node[circle, fill=black, inner sep=1.5pt] (n6) at (18,5) {};
    \node[circle, fill=black, inner sep=1.5pt] (n7) at (19,6) {};
    \node[circle, fill=black, inner sep=1.5pt] (n8) at (15,4) {};
    \node[circle, fill=black, inner sep=1.5pt] (n9) at (14,4) {};
    \node[circle, fill=black, inner sep=1.5pt] (n10) at (13,3) {};
    \node[circle, fill=black, inner sep=1.5pt] (n11) at (12,3) {};
    \node[circle, fill=black, inner sep=1.5pt] (n12) at (11,3) {};
    \node[circle, fill=black, inner sep=1.5pt] (n13) at (10,2) {};
    \node[circle, fill=black, inner sep=1.5pt] (n14) at (9,3) {};
    \node[circle, fill=black, inner sep=1.5pt] (n15) at (9,2) {};
    \node[circle, fill=black, inner sep=1.5pt] (n16) at (8,3) {};
    \node[circle, fill=black, inner sep=1.5pt] (n17) at (8,4) {};
    \node[circle, fill=black, inner sep=1.5pt] (n18) at (1,6) {};
    \node[circle, fill=black, inner sep=1.5pt] (n19) at (2,6) {};
    \node[circle, fill=black, inner sep=1.5pt] (n20) at (4,6) {};
    \node[circle, fill=black, inner sep=1.5pt] (n21) at (4,4) {};
    \node[circle, fill=black, inner sep=1.5pt] (n22) at (5,4) {};
    \node[circle, fill=black, inner sep=1.5pt] (n23) at (6,1) {};
    \node[circle, fill=black, inner sep=1.5pt] (n24) at (7,1) {};
    \node[circle, fill=black, inner sep=1.5pt] (n25) at (6,4) {};
    \node[circle, fill=black, inner sep=1.5pt] (n26) at (2,0) {};
    \node[circle, fill=black, inner sep=1.5pt] (n27) at (3,0) {};

    \draw [color=gray, line width=2] (n28)--(n0)--(n1);
    \draw [color=brown, line width=2] (n1)--(n2)--(n18);
    \draw [color=blue, line width=2] (n18)--(n7);
    \draw [color=blue, line width=2] (n3)--(n4);
    \draw [color=blue, line width=2] (n19)--(n26);
    \draw [color=blue, line width=2] (n20)--(n21);
    \draw [color=black, line width=2] (n21)--(n22);
    \draw [color=red, line width=2] (n22)--(n9);
    \draw [color=red, line width=2] (n23)--(n25);
    \draw [color=red, line width=2] (n17)--(n16);

    \draw [color=green, line width=2] (n16)--(n12);
    \draw [color=green, line width=2] (n14)--(n15);

    \draw [color=black, line width=2] (2,0)--(3,0);
    \draw [color=black, line width=2] (6,1)--(7,1);
    \draw [color=black, line width=2] (9,2)--(10,2);
    \draw [color=black, line width=2] (11,3)--(12,3);
    \draw [color=black, line width=2] (12,3)--(13,3);
    \draw [color=black, line width=2] (14,4)--(15,4);
    \draw [color=black, line width=2] (16,5)--(17,5);
    \draw [color=black, line width=2] (17,5)--(18,5);

    %\node[anchor=south west,scale=0.8] at (14,4) {\color{gray0}{$15$}};
    %\node[anchor=south west,scale=0.8] at (5,4) {\color{red}{$6$}};
    %\node[anchor=south west,scale=0.8] at (6,4) {\color{red}{$6$}};
    %\node[anchor=south west,scale=0.8] at (8,4) {\color{red}{$6$}};
    %\node[anchor=south west,scale=0.8] at (4,4) {\color{col5}{$5$}};
    %\node[anchor=south west,scale=0.8] at (12,3) {\color{col9}{$13$}};
    %\node[anchor=south west,scale=0.8] at (11,3) {\color{col8}{$12$}};
    %\node[anchor=south west,scale=0.8] at (9,3) {\color{green}{$9$}};
    %\node[anchor=south west,scale=0.8] at (8,3) {\color{green}{$9$}};
    %\node[anchor=south west,scale=0.8] at (8,2) {\color{green}{$9$}};
    %\node[anchor=south west,scale=0.8] at (9,2) {\color{col7}{$10$}};
    %\node[anchor=south west,scale=0.8] at (6,1) {\color{col6}{$7$}};
    %\node[anchor=south west,scale=0.8] at (2,0) {\color{col5}{$3$}};

        \node[circle, fill=black, inner sep=1.5pt] (n1) at (0,-1) {};
    \node[circle, fill=black, inner sep=1.5pt] (n2) at (0,6) {};
    \node[circle, fill=black, inner sep=1.5pt] (n3) at (16,6) {};
    \node[circle, fill=black, inner sep=1.5pt] (n4) at (16,5) {};
    \node[circle, fill=black, inner sep=1.5pt] (n5) at (17,5) {};
    \node[circle, fill=black, inner sep=1.5pt] (n6) at (18,5) {};
    \node[circle, fill=black, inner sep=1.5pt] (n7) at (19,6) {};
    \node[circle, fill=black, inner sep=1.5pt] (n8) at (15,4) {};
    \node[circle, fill=black, inner sep=1.5pt] (n9) at (14,4) {};
    \node[circle, fill=black, inner sep=1.5pt] (n10) at (13,3) {};
    \node[circle, fill=black, inner sep=1.5pt] (n11) at (12,3) {};
    \node[circle, fill=black, inner sep=1.5pt] (n12) at (11,3) {};
    \node[circle, fill=black, inner sep=1.5pt] (n13) at (10,2) {};
    \node[circle, fill=black, inner sep=1.5pt] (n14) at (9,3) {};
    \node[circle, fill=black, inner sep=1.5pt] (n15) at (9,2) {};
    \node[circle, fill=black, inner sep=1.5pt] (n16) at (8,3) {};
    \node[circle, fill=black, inner sep=1.5pt] (n17) at (8,4) {};
    \node[circle, fill=black, inner sep=1.5pt] (n18) at (1,6) {};
    \node[circle, fill=black, inner sep=1.5pt] (n19) at (2,6) {};
    \node[circle, fill=black, inner sep=1.5pt] (n20) at (4,6) {};
    \node[circle, fill=black, inner sep=1.5pt] (n21) at (4,4) {};
    \node[circle, fill=black, inner sep=1.5pt] (n22) at (5,4) {};
    \node[circle, fill=black, inner sep=1.5pt] (n23) at (6,1) {};
    \node[circle, fill=black, inner sep=1.5pt] (n24) at (7,1) {};
    \node[circle, fill=black, inner sep=1.5pt] (n25) at (6,4) {};
    \node[circle, fill=black, inner sep=1.5pt] (n26) at (2,0) {};
    \node[circle, fill=black, inner sep=1.5pt] (n27) at (3,0) {};

    \draw [color=blue, line width=2] (1.4,8.2)--(1.4,10)--(18.6,10) -- (18.6,8.2);
    \draw [color=blue, line width=2] (3.4,8.2)--(3.4,10);
    \draw [color=blue, line width=2] (15.6,8.2)--(15.6,10);

    \draw [color=red, line width=2] (5.4,8.2)--(5.4,9.5)--(13.5,9.5) -- (13.5,8.2);
    \draw [color=red, line width=2] (7.4,8.2)--(7.4,9.5);

    \draw [color=green, line width=2] (8.4,8.2)--(8.4,9)--(10.5,9) -- (10.5,8.2);

    \end{tikzpicture}
\end{center}
\caption{A $\nu$-tree and its $s$-noncrossing partition via the descent tree coloring.}
\label{fig:descent_tree_coloring}
\end{figure}
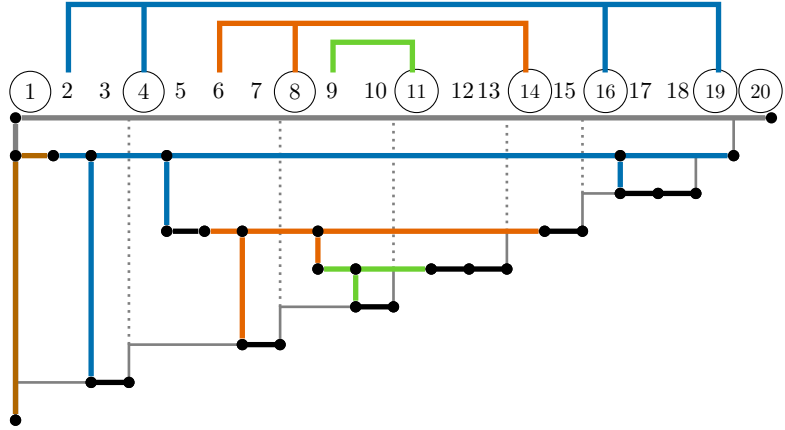

\begin{example}
Several examples for the behavior of $s$-noncrossing partitions under the action of rowmotion on the~$\nu(s)$-Tamari lattice are depicted in~\cref{fig:s_ncp_row1,fig:s_ncp_row2,fig:s_ncp_ex}. Here, rowmotion shares a similar structure with rotation, aside from an additional rule when starting a new block.
\end{example}

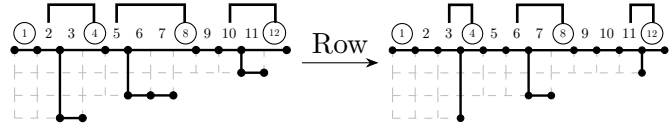
\begin{figure}[ht]
\begin{center}
\begin{tikzpicture}
[node distance=5cm,
    >={Stealth},
]

\node (f1) [inner sep=4pt] {
    \altnuTree{3,4,4,1}{4,4,1}{1,2,1,8}{.3}{}{}{}{
    \node[anchor=south west,scale=.4, circle, draw] (n1) at (0.1,3.4) {\color{black}{$1$}}; 
    \node [anchor=south west,scale=.5] at (1.1,3.3) {$2$};
    \node [anchor=south west,scale=.5] at (2.1,3.3) {$3$};
    \node[anchor=south west,scale=.4, circle, draw] (n4) at (3.2,3.4) {\color{black}{$4$}};  
    \node [anchor=south west,scale=.5] at (4.1,3.3) {$5$};
    \node [anchor=south west,scale=.5] at (5.1,3.3) {$6$};
    \node [anchor=south west,scale=.5] at (6.1,3.3) {$7$};
    \node[anchor=south west,scale=.4, circle, draw] (n4) at (7.2,3.4) {\color{black}{$8$}};

    \node [anchor=south west,scale=.5] at (8.1,3.3) {$9$};
    \node [anchor=south west,scale=.5] at (9 -.1,3.3) {$10$};
    \node [anchor=south west,scale=.5] at (10 -.1,3.3) {$11$};

    \node[anchor=south west,scale=.35, circle, draw] (n4) at (11+.1,3.4) {\color{black}{$12$}};
    
        \draw[line width=1pt] (1.5,4.2)--(1.5,5)--(3.5,5)--(3.5,4.2); 
        \draw[line width=1pt] (4.5,4.2)--(4.5,5)--(7.5,5)--(7.5,4.2); 
        \draw[line width=1pt] (9.5,4.2)--(9.5,5)--(11.5,5)--(11.5,4.2); 
    
        %\draw[line width=1pt, red] (2,0)--(2,3)--(3.5,5)--(3.5,4.2); 

    }};

\node (f2) [right of=f1] {
    \altnuTree{3,4,4,1}{4,4,1}{0,1,0,11}{.3}{}{}{}{
    \node[anchor=south west,scale=.4, circle, draw] (n1) at (0.1,3.4) {\color{black}{$1$}}; 
    \node [anchor=south west,scale=.5] at (1.1,3.3) {$2$};
    \node [anchor=south west,scale=.5] at (2.1,3.3) {$3$};
    \node[anchor=south west,scale=.4, circle, draw] (n4) at (3.2,3.4) {\color{black}{$4$}};  
    \node [anchor=south west,scale=.5] at (4.1,3.3) {$5$};
    \node [anchor=south west,scale=.5] at (5.1,3.3) {$6$};
    \node [anchor=south west,scale=.5] at (6.1,3.3) {$7$};
    \node[anchor=south west,scale=.4, circle, draw] (n4) at (7.2,3.4) {\color{black}{$8$}};

    \node [anchor=south west,scale=.5] at (8.1,3.3) {$9$};
    \node [anchor=south west,scale=.5] at (9 -.1,3.3) {$10$};
    \node [anchor=south west,scale=.5] at (10 -.1,3.3) {$11$};

    \node[anchor=south west,scale=.35, circle, draw] (n4) at (11+.1,3.4) {\color{black}{$12$}};
    
            \draw[line width=1pt] (2.5,4.2)--(2.5,5)--(3.5,5)--(3.5,4.2); 
        \draw[line width=1pt] (5.5,4.2)--(5.5,5)--(7.5,5)--(7.5,4.2); 
        \draw[line width=1pt] (10.5,4.2)--(10.5,5)--(11.5,5)--(11.5,4.2);

    }};

\draw[->] (f1) -- (f2) node[midway,above] {$\row$};

\end{tikzpicture}
\end{center}
\caption{The block $\{a,b\}$, where $a$ is not a first column index, becomes $\{a+1,b\}$.}
\label{fig:s_ncp_row1}
\end{figure}

\begin{figure}[ht]
\begin{center}
\begin{tikzpicture}
[node distance=5cm,
    >={Stealth},
]

\node (f1) [inner sep=4pt] {
    \altnuTree{3,4,4,1}{4,4,1}{1,3,7,1}{.3}{}{}{}{
    \node[anchor=south west,scale=.4, circle, draw] (n1) at (0.1,3.4) {\color{black}{$1$}}; 
    \node [anchor=south west,scale=.5] at (1.1,3.3) {$2$};
    \node [anchor=south west,scale=.5] at (2.1,3.3) {$3$};
    \node[anchor=south west,scale=.4, circle, draw] (n4) at (3.2,3.4) {\color{black}{$4$}};  
    \node [anchor=south west,scale=.5] at (4.1,3.3) {$5$};
    \node [anchor=south west,scale=.5] at (5.1,3.3) {$6$};
    \node [anchor=south west,scale=.5] at (6.1,3.3) {$7$};
    \node[anchor=south west,scale=.4, circle, draw] (n4) at (7.2,3.4) {\color{black}{$8$}};

    \node [anchor=south west,scale=.5] at (8.1,3.3) {$9$};
    \node [anchor=south west,scale=.5] at (9 -.1,3.3) {$10$};
    \node [anchor=south west,scale=.5] at (10 -.1,3.3) {$11$};

    \node[anchor=south west,scale=.35, circle, draw] (n4) at (11+.1,3.4) {\color{black}{$12$}};
    
        \draw[line width=1pt] (1.5,4.2)--(1.5,5)--(3.5,5)--(3.5,4.2); 
        \draw[line width=1pt] (3.5,5)--(7.5,5)--(7.5,4.2);

        %\draw[line width=1pt, red] (2,0)--(2,3)--(3.5,5)--(3.5,4.2); 

    }};

\node (f2) [right of=f1] {
    \altnuTree{3,4,4,1}{4,4,1}{0,2,8,2}{.3}{}{}{}{
    \node[anchor=south west,scale=.4, circle, draw] (n1) at (0.1,3.4) {\color{black}{$1$}}; 
    \node [anchor=south west,scale=.5] at (1.1,3.3) {$2$};
    \node [anchor=south west,scale=.5] at (2.1,3.3) {$3$};
    \node[anchor=south west,scale=.4, circle, draw] (n4) at (3.2,3.4) {\color{black}{$4$}};  
    \node [anchor=south west,scale=.5] at (4.1,3.3) {$5$};
    \node [anchor=south west,scale=.5] at (5.1,3.3) {$6$};
    \node [anchor=south west,scale=.5] at (6.1,3.3) {$7$};
    \node[anchor=south west,scale=.4, circle, draw] (n4) at (7.2,3.4) {\color{black}{$8$}};

    \node [anchor=south west,scale=.5] at (8.1,3.3) {$9$};
    \node [anchor=south west,scale=.5] at (9 -.1,3.3) {$10$};
    \node [anchor=south west,scale=.5] at (10 -.1,3.3) {$11$};

    \node[anchor=south west,scale=.35, circle, draw] (n4) at (11+.1,3.4) {\color{black}{$12$}};
    
            \draw[line width=1pt] (2.5,4.2)--(2.5,5)--(3.5,5)--(3.5,4.2); 
        \draw[line width=1pt] (4.5,4.2)--(4.5,5)--(7.5,5)--(7.5,4.2); 
        \draw[line width=1pt] (0.5,4.2)--(0.5,5.5)--(11.5,5.5)--(11.5,4.2);

    }};

\draw[->] (f1) -- (f2) node[midway,above] {$\row$};

\end{tikzpicture}
\end{center}
\caption{The block $\{2, 4, 8\}$ can be treated as two separate blocks, $\{2, 4\}$ and~$\{4, 8\}$. In both blocks, the smaller element is increased by one, resulting in $\{3, 4\}$ and $\{5, 8\}$. Furthermore, a new block $\{1, 12\}$ is created.}
\label{fig:s_ncp_row2}
\end{figure}
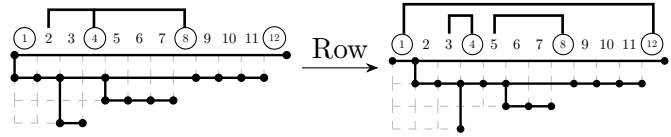

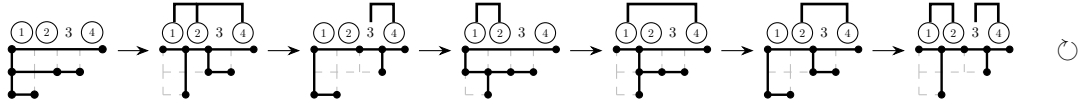
\begin{figure}[ht]
\begin{center}
\begin{tikzpicture}
[node distance=2cm,
    >={Stealth},
]

\node (f1) [inner sep=4pt] {
    \altnuTree{1,2,1}{2,1}{1,2,1}{.3}{}{}{}{
    \node[anchor=south west,scale=.4, circle, draw] (n1) at (0.1,2.4) {\color{black}{$1$}}; 
    \node[anchor=south west,scale=.4, circle, draw] (n2) at (1.2,2.4) {\color{black}{$2$}}; 
    \node [anchor=south west,scale=.5] at (2.1,2.3) {$3$};
    \node[anchor=south west,scale=.4, circle, draw] (n4) at (3.2,2.4) {\color{black}{$4$}};  
    \draw[line width=1pt, color=white] (2.5,3.2)--(2.5,4)--(3.5,4)--(3.5,4); 
    }};

\node (f2) [right of=f1] {
    \altnuTree{1,2,1}{2,1}{0,1,3}{.3}{}{}{}{
    \node[anchor=south west,scale=.4, circle, draw] (n1) at (0.1,2.4) {\color{black}{$1$}}; 
    \node[anchor=south west,scale=.4, circle, draw] (n2) at (1.2,2.4) {\color{black}{$2$}}; 
    \node [anchor=south west,scale=.5] at (2.1,2.3) {$3$};
    \node[anchor=south west,scale=.4, circle, draw] (n4) at (3.2,2.4) {\color{black}{$4$}};
    \draw[line width=1pt] (1.5,3.2)--(1.5,4)--(3.5,4)--(3.5,3.2); 
    \draw[line width=1pt] (1.5,4)--(0.5,4)--(0.5,3.2); 
    
    }};
\node (f3) [right of=f2] {
    \altnuTree{1,2,1}{2,1}{1,0,3}{.3}{}{}{}{
    \node[anchor=south west,scale=.4, circle, draw] (n1) at (0.1,2.4) {\color{black}{$1$}}; 
    \node[anchor=south west,scale=.4, circle, draw] (n2) at (1.2,2.4) {\color{black}{$2$}}; 
    \node [anchor=south west,scale=.5] at (2.1,2.3) {$3$};
    \node[anchor=south west,scale=.4, circle, draw] (n4) at (3.2,2.4) {\color{black}{$4$}};
    \draw[line width=1pt] (2.5,3.2)--(2.5,4)--(3.5,4)--(3.5,3.2); 
    
    }};
\node (f4) [right of=f3] {
    \altnuTree{1,2,1}{2,1}{0,3,1}{.3}{}{}{}{
    \node[anchor=south west,scale=.4, circle, draw] (n1) at (0.1,2.4) {\color{black}{$1$}}; 
    \node[anchor=south west,scale=.4, circle, draw] (n2) at (1.2,2.4) {\color{black}{$2$}}; 
    \node [anchor=south west,scale=.5] at (2.1,2.3) {$3$};
    \node[anchor=south west,scale=.4, circle, draw] (n4) at (3.2,2.4) {\color{black}{$4$}};
    \draw[line width=1pt] (0.5,3.2)--(0.5,4)--(1.5,4)--(1.5,3.2); 
    
    }};
\node (f5) [right of=f4] {
    \altnuTree{1,2,1}{2,1}{0,2,2}{.3}{}{}{}{
    \node[anchor=south west,scale=.4, circle, draw] (n1) at (0.1,2.4) {\color{black}{$1$}}; 
    \node[anchor=south west,scale=.4, circle, draw] (n2) at (1.2,2.4) {\color{black}{$2$}}; 
    \node [anchor=south west,scale=.5] at (2.1,2.3) {$3$};
    \node[anchor=south west,scale=.4, circle, draw] (n4) at (3.2,2.4) {\color{black}{$4$}};
    \draw[line width=1pt] (0.5,3.2)--(0.5,4)--(3.5,4)--(3.5,3.2); 
    
    }};
\node (f6) [right of=f5] {
    \altnuTree{1,2,1}{2,1}{1,1,2}{.3}{}{}{}{
    \node[anchor=south west,scale=.4, circle, draw] (n1) at (0.1,2.4) {\color{black}{$1$}}; 
    \node[anchor=south west,scale=.4, circle, draw] (n2) at (1.2,2.4) {\color{black}{$2$}}; 
    \node [anchor=south west,scale=.5] at (2.1,2.3) {$3$};
    \node[anchor=south west,scale=.4, circle, draw] (n4) at (3.2,2.4) {\color{black}{$4$}};
    \draw[line width=1pt] (1.5,3.2)--(1.5,4)--(3.5,4)--(3.5,3.2); 
    
    }};
\node (f7) [right of=f6] {
    \altnuTree{1,2,1}{2,1}{0,0,4}{.3}{}{}{}{
    \node[anchor=south west,scale=.4, circle, draw] (n1) at (0.1,2.4) {\color{black}{$1$}}; 
    \node[anchor=south west,scale=.4, circle, draw] (n2) at (1.2,2.4) {\color{black}{$2$}}; 
    \node [anchor=south west,scale=.5] at (2.1,2.3) {$3$};
    \node[anchor=south west,scale=.4, circle, draw] (n4) at (3.2,2.4) {\color{black}{$4$}};
    \draw[line width=1pt] (0.5,3.2)--(0.5,4)--(1.5,4)--(1.5,3.2); 

    \draw[line width=1pt] (2.5,3.2)--(2.5,4)--(3.5,4)--(3.5,3.2); 
    
    }};

\draw[->] (f1) -- (f2);
\draw[->] (f2) -- (f3);
\draw[->] (f3) -- (f4);
\draw[->] (f4) -- (f5);
\draw[->] (f5) -- (f6);
\draw[->] (f6) -- (f7);

\node[right=3mm of f7] {$\circlearrowright$};
% \draw[->, bend left=8] (f7.south) to (f1.south);
\end{tikzpicture}
\end{center}
\caption{The $s$-noncrossing partitions for the rowmotion orbit in~\cref{fig:rowmotion_orbits_ENEEN}.}
\label{fig:s_ncp_ex}
\end{figure}

\begin{question}
A complete formulation for a $\nu$-analogue of the Kreweras complement is currently unknown. Can a general rule be explicitly defined?
\end{question}

\subsection{Connections with RSK}

As we have mentioned several times, Armstrong, Stump, and Thomas~\cite{armstrong2013uniform} showed that rowmotion of the Stanley lattice is in equivariant bijection with Kreweras complement of noncrossing partitions, and Striker and Williams~\cite{striker2012rowmotion} expressed this bijection as a certain conjugating map in the toggle group. Adenbaum and Elizalde~\cite{adenbaum2023rowmotion} provided a different interpretation of the bijection from~\cite{armstrong2013uniform} that takes rowmotion to Kreweras complement; they showed that it is essentially the \dfn{Robinson--Schensted--Knuth correspondence} restricted to 321-avoiding permutations. It would be interesting, therefore, to try to connect rowmotion for the other alt~$\nu$-Tamari lattices with RSK or its variants. We believe a promising approach might be to show that the \dfn{scrambled RSK} of Dauvergne~\cite{dauvergne2022hidden} and Garver, Patrias, and Thomas~\cite{garver2023minuscule} intertwines rowmotion on alt~$\nu$-Tamari lattices and Kreweras complement of noncrossing partitions for $\nu$ a staircase shape. For other shapes $\nu$, the connection to RSK is less clear.

\subsection{Piecewise-linear extension}

Around ten years ago, Einstein and Propp~\cite{einstein2014piecewise,einstein2021combinatorial} introduced a \emph{piecewise-linear (PL)} extension of rowmotion for distributive lattices, which has subsequently received a lot of research attention. Invariance of rowmotion of distributive lattices in this {piecewise-linear} setting  has also, to some degree, been considered before. For example, in~\cite[Conjecture~4.38]{hopkins2022minuscule} it is conjectured that if $P$ and $Q$ are posets with isomorphic comparability graphs, then $J(P)$ and $J(Q)$ should have the same PL rowmotion behavior. And in~\cite{johnson2023trapezoid}, Johnson and Liu show that if $P$ is a rectangle and $Q$ is the corresponding trapezoid, then indeed $J(P)$ and $J(Q)$ have the same PL rowmotion behavior. 

More recently, building on~\cite{thomas2019rowmotion,thomas2019independence}, Williams~\cite{williams2020personal} has defined a piecewise-linear extension of ``beyond distributive lattice'' rowmotion. We believe that alt~$\nu$-Tamari lattices exhibit invariance of {PL rowmotion} as well. We now briefly review the definition of PL rowmotion in the case of a semidistributive and extremal lattice~$L$.

Given any finite graph $G$ with vertex set $V$, we let $\mathcal{P}_G$ denote the (convex) polytope in $\mathbb{R}^V$ defined by $x_v \geq 0$ for all $v\in V$ and $\sum_{v \in C} x_v \leq 1$ for all (maximal) cliques $C$. (This polytope goes under various names in the literature; in~\cite{johnson2025piecewise} it is called the \dfn{clique constraint stable set polytope}. Note that the lattice points in $\mathcal{P}_G$ are exactly the indicator functions of \dfn{independent sets}, also known as \dfn{stable sets}, of the graph $G$.) Given a vertex $v \in V$, the \dfn{piecewise-linear toggle} at $v$ is defined to be the (continuous, piecewise-linear) involution $\tau_v\colon\mathcal{P}_G\to\mathcal{P}_G$ that does not change the coordinates~$x_u$ for $u \neq v$, and replaces the coordinate $x_v$ by $1-\max_C\sum_{u\in C}x_u$, where the maximum is over all (maximal) cliques $C$ containing $v$. Now let $L$ be a semidistributive and extremal lattice. Recall that its (undirected) \dfn{Galois graph} $G_L$ can be defined as the complement graph of the one-skeleton of its canonical join complex; see~\cite{barnard2019canonical, thomas2019rowmotion, thomas2019independence}. In particular, recall that the vertex set of $G_L$ is $\mathcal{J}_L$, the set of join-irreducible elements of $L$. For such an $L$, \dfn{piecewise-linear rowmotion} is the operator~$\row \colon \mathcal{P}_{G_L} \to \mathcal{P}_{G_L}$ defined as the composition of all the (PL) toggles $\tau_j$ for $j \in \mathcal{J}_L$ in the order corresponding to any maximal-length chain of $L$. 

In the case where $L=J(P)$ is the distributive lattice of order ideals of a poset $P$, its Galois graph~$G_L$ is the \dfn{comparability graph} of $P$, and the corresponding polytope $\mathcal{P}_{G_L}$ is the \dfn{chain polytope}~\cite{stanley1986two} of $P$. Hence, the PL rowmotion we have defined here is the ``antichain version'' of rowmotion, as studied by Joseph~\cite{joseph2019antichain}. Also note that when $L$ is distributive, this polytope $\mathcal{P}_{G_L}$ is actually a \emph{lattice} polytope; but for other semidistributive, extremal $L$, it will just be a rational polytope in general.

Computational evidence strongly suggests that the ``invariance of rowmotion'' behavior for alt~$\nu$-Tamari lattices should extend to the piecewise-linear setting.

\begin{conjecture} \label{conj:pl}
Fix a lattice path $\nu$, and let $L_1$ and $L_2$ be two alt~$\nu$-Tamari lattices. Then piecewise-linear rowmotion behaves the same for $L_1$ and $L_2$. That is, there is a continuous, piecewise-linear bijection $\Phi\colon \mathcal{P}_{G_{L_1}} \to \mathcal{P}_{G_{L_2}}$, which restricts to a bijection $\Phi\colon \mathcal{P}_{G_{L_1}} \cap (\frac{1}{m}\mathbb{Z})^{\mathcal{J}_{L_1}} \to \mathcal{P}_{G_{L_2}} \cap (\frac{1}{m}\mathbb{Z})^{\mathcal{J}_{L_2}}$ for each integer $m \geq 1$, such that:
\begin{itemize}
\item $\Phi(\row(x)) = \row(\Phi(x))$ for all $x\in \mathcal{P}_{G_{L_1}}$;
\item $\lim_{n\to\infty} \frac{1}{n} \sum_{i=0}^{n-1} \operatorname{sum}(\row^i(x))=\lim_{n\to\infty} \frac{1}{n} \sum_{i=0}^{n-1} \operatorname{sum}(\row^i(\Phi(x)))$ for all $x\in \mathcal{P}_{G_{L_1}}$,
\end{itemize}
where $\operatorname{sum}(x)$ is the sum-of-coordinates (the piecewise-linear analogue of the down-degree statistic).
\end{conjecture}

For an arbitrary distributive lattice $L$, piecewise-linear rowmotion will not have finite order. But in~\cite{grinberg2015birational}, Grinberg and Roby showed that PL rowmotion \emph{does} have finite order for some very special distributive lattices $L$, including, in particular, the Stanley lattice. Hence, a corollary of \cref{conj:pl} would be that for the classical Tamari lattice, PL rowmotion has finite order, a fact that was experimentally observed and conjectured previously by Williams~\cite{williams2020personal}.

For a semidistributive, extremal lattice $L$, and for each fixed integer $m \geq 1$, by restricting to the rational points in $\mathcal{P}_{G_L}$ with denominator dividing $m$, one obtains an action of PL rowmotion on a \emph{finite} set (with the case $m=1$ recovering combinatorial rowmotion on the lattice $L$ itself). For some very special distributive lattices $L$, there are known or conjectured cyclic sieving results for this PL action; see~\cite{hopkins2020cyclic,hopkins2024order}. In particular, for the Stanley lattice it is conjectured that the so-called~\dfn{$q$-multi Catalan number} should be a CSP polynomial for this action, and this is known for some small values of $m$; for $m=1$, it is the result of Armstrong--Stump--Thomas~\cite{armstrong2013uniform}, and for $m=2$ it was recently proved by Hopkins--Kim--Pfannerer~\cite{hopkins2026cyclic}. Likewise, there are known homomesy results for PL rowmotion of certain special distributive lattices, including the Stanley lattice; see, e.g.,~\cite{defant2023homomesy} and its references. Hence, \cref{conj:pl} would transfer these CSP and homomesy results for PL rowmotion from the Stanley lattice to the Tamari lattice.

\Cref{conj:pl} can be naturally extended in various ways as well. For example, we believe it should extend directly to the cross Tamari lattices associated to two moon polyominoes related by permutations of rows and columns. Also, there are even further ``lifts'' of rowmotion beyond piecewise-linear rowmotion, such as birational rowmotion. We believe that~\cref{conj:pl} should extend to birational rowmotion as well. But we have not extensively tested this.

\begin{remark}
\Cref{conj:pl} seems closely related to work of Johnson and Liu~\cite{johnson2025piecewise}, studying 0-1-fillings of moon polyominoes~\cite{rubey2011increasing} and piecewise-linear extensions. Specifically, they construct continuous, piecewise-linear bijections between clique constraint stable set polytopes associated to moon polyominoes related by permutations of rows and columns. However, their polytopes are not exactly the same as the ones we attach to alt~$\nu$-Tamari lattices: the distinction being the difference between 0-1-fillings that avoid having a $1$ \emph{strictly} northeast of another $1$, versus avoid having a $1$ \emph{weakly} northeast of another $1$. So the relationship of \cref{conj:pl} to their work remains unclear. It would certainly be worth exploring whether the techniques from~\cite{johnson2025piecewise} (and~\cite{johnson2024birational}) could be used to resolve \cref{conj:pl}. We note that Adenbaum and Williams are actively pursuing \cref{conj:pl} and related piecewise-linear problems.
\end{remark}

\subsection{Other types} 

It would also be interesting to extend this work to other Lie types. For example, Armstrong, Stump, and Thomas~\cite{armstrong2013uniform} showed that rowmotion of the (distributive lattice of) order ideals of the root poset of \emph{any} finite crystallographic root system $\Phi$ is in equivariant bijection with Kreweras complement of the corresponding noncrossing partition lattice, which is a certain interval in the absolute order of the corresponding Weyl group. Meanwhile, it is also known that (semidistributive) rowmotion of the Type B Tamari lattice is in equivariant bijection with Kreweras complement of Type B noncrossing partitions; see, e.g.,~\cite[\S1.2]{thomas2019rowmotion}. But, as far as we know, there is no definition yet in the literature of ``alt'' Tamari lattices, interpolating between the Tamari lattice and the distributive lattice associated to the root poset, outside of Type A.

\bibliographystyle{plain}
\bibliography{main}

@article {adenbaum2023rowmotion,
    AUTHOR = {Adenbaum, Ben and Elizalde, Sergi},
     TITLE = {Rowmotion on 321-avoiding permutations},
   JOURNAL = {Electron. J. Combin.},
  FJOURNAL = {Electronic Journal of Combinatorics},
    VOLUME = {30},
      YEAR = {2023},
    NUMBER = {3},
     PAGES = {Paper No. 3.5, 26},
      ISSN = {1077-8926},
   MRCLASS = {05E18 (05A05 05A19 06A07)},
  MRNUMBER = {4614539},
MRREVIEWER = {Jia\ Huang},
       DOI = {10.37236/11792},
       URL = {https://doi.org/10.37236/11792},
}

@article {armstrong2013rational,
    AUTHOR = {Armstrong, Drew and Rhoades, Brendon and Williams, Nathan},
     TITLE = {Rational associahedra and noncrossing partitions},
   JOURNAL = {Electron. J. Combin.},
  FJOURNAL = {Electronic Journal of Combinatorics},
    VOLUME = {20},
      YEAR = {2013},
    NUMBER = {3},
     PAGES = {Paper 54, 27},
      ISSN = {1077-8926},
   MRCLASS = {05E45 (05A19)},
  MRNUMBER = {3118962},
MRREVIEWER = {Ragnar\ Freij},
       DOI = {10.37236/3432},
       URL = {https://doi.org/10.37236/3432},
}

@article {armstrong2013uniform,
    AUTHOR = {Armstrong, Drew and Stump, Christian and Thomas, Hugh},
     TITLE = {A uniform bijection between nonnesting and noncrossing
              partitions},
   JOURNAL = {Trans. Amer. Math. Soc.},
  FJOURNAL = {Transactions of the American Mathematical Society},
    VOLUME = {365},
      YEAR = {2013},
    NUMBER = {8},
     PAGES = {4121--4151},
      ISSN = {0002-9947,1088-6850},
   MRCLASS = {05A05 (05A19 20F55)},
  MRNUMBER = {3055691},
MRREVIEWER = {Alessandro\ Conflitti},
       DOI = {10.1090/S0002-9947-2013-05729-7},
       URL = {https://doi.org/10.1090/S0002-9947-2013-05729-7},
}

@article {armstrong2016rational,
    AUTHOR = {Armstrong, Drew and Loehr, Nicholas A. and Warrington, Gregory S.},
     TITLE = {Rational parking functions and {C}atalan numbers},
   JOURNAL = {Ann. Comb.},
  FJOURNAL = {Annals of Combinatorics},
    VOLUME = {20},
      YEAR = {2016},
    NUMBER = {1},
     PAGES = {21--58},
      ISSN = {0218-0006,0219-3094},
   MRCLASS = {05E10 (05A30 05E05 05E18)},
  MRNUMBER = {3461934},
MRREVIEWER = {Thomas\ Ernst},
       DOI = {10.1007/s00026-015-0293-6},
       URL = {https://doi.org/10.1007/s00026-015-0293-6},
}

@misc{axelrodfreed2025chute,
    AUTHOR={Ilani Axelrod-Freed and Colin Defant and Hanna Mularczyk and Son Nguyen and Katherine Tung},
    TITLE={Chute Move Posets are Lattices},
    YEAR={2025},
    HOWPUBLISHED={\arxiv{2507.13214}}
}

@article {barnard2019canonical,
    AUTHOR = {Barnard, Emily},
     TITLE = {The canonical join complex},
   JOURNAL = {Electron. J. Combin.},
  FJOURNAL = {Electronic Journal of Combinatorics},
    VOLUME = {26},
      YEAR = {2019},
    NUMBER = {1},
     PAGES = {Paper No. 1.24, 25},
      ISSN = {1077-8926},
   MRCLASS = {05E45 (05E10 06A07)},
  MRNUMBER = {3919619},
MRREVIEWER = {Konrad\ P.\ Pi\'oro},
       DOI = {10.37236/7866},
       URL = {https://doi.org/10.37236/7866},
}

@article {barnard2021dynamical,
    AUTHOR = {Barnard, Emily and Todorov, Gordana and Zhu, Shijie},
     TITLE = {Dynamical combinatorics and torsion classes},
   JOURNAL = {J. Pure Appl. Algebra},
  FJOURNAL = {Journal of Pure and Applied Algebra},
    VOLUME = {225},
      YEAR = {2021},
    NUMBER = {9},
     PAGES = {Paper No. 106642, 25},
      ISSN = {0022-4049,1873-1376},
   MRCLASS = {16G10 (05E10 06A07 16S90)},
  MRNUMBER = {4195887},
MRREVIEWER = {Xin\ Ma},
       DOI = {10.1016/j.jpaa.2020.106642},
       URL = {https://doi.org/10.1016/j.jpaa.2020.106642},
}

@misc{vonbell2025framing,
    AUTHOR={Bell, Matias von and Ceballos, Cesar},
    TITLE={Framing Lattices and Flow Polytopes},
    YEAR={2025},
    HOWPUBLISHED={\arxiv{2512.20575}}
}

@article{bergeron2012higher,
    AUTHOR = {Bergeron, Fran\c{c}ois and Pr\'{e}ville-Ratelle, Louis-Fran\c{c}ois},
     TITLE = {Higher trivariate diagonal harmonics via generalized {T}amari
              posets},
   JOURNAL = {J. Comb.},
  FJOURNAL = {Journal of Combinatorics},
    VOLUME = {3},
      YEAR = {2012},
    NUMBER = {3},
     PAGES = {317--341},
      ISSN = {2156-3527,2150-959X},
   MRCLASS = {05E10 (05A19)},
  MRNUMBER = {3029440},
MRREVIEWER = {Anthony\ A.\ Mendes},
       DOI = {10.4310/JOC.2012.v3.n3.a4},
       URL = {https://doi.org/10.4310/JOC.2012.v3.n3.a4},
}

@misc{billey2025lattice,
    AUTHOR={Sara C. Billey and Connor McCausland and Clare Minnerath},
    TITLE={A Proof of {R}ubey's Lattice Conjecture},
    YEAR={2025},
    HOWPUBLISHED={\arxiv{2507.18852}},
    NOTE={Forthcoming, \emph{Comb. Theory}}
}

@article {bizley1954longtitle,
    AUTHOR = {Bizley, M. T. L.},
     TITLE = {Derivation of a new formula for the number of minimal lattice
              paths from {$(0,0)$} to {$(km,kn)$} having just {$t$} contacts
              with the line {$my=nx$} and having no points above this line;
              and a proof of {G}rossman's formula for the number of paths
              which may touch but do not rise above this line},
   JOURNAL = {J. Inst. Actuar.},
  FJOURNAL = {Journal of the Institute of Actuaries},
    VOLUME = {80},
      YEAR = {1954},
     PAGES = {55--62},
      ISSN = {0020-2681,2058-1009},
   MRCLASS = {09.0X},
  MRNUMBER = {61567},
MRREVIEWER = {John\ Riordan},
}

@article {bodnar2016rational,
    AUTHOR = {Bodnar, Michelle and Rhoades, Brendon},
     TITLE = {Cyclic sieving and rational {C}atalan theory},
   JOURNAL = {Electron. J. Combin.},
  FJOURNAL = {Electronic Journal of Combinatorics},
    VOLUME = {23},
      YEAR = {2016},
    NUMBER = {2},
     PAGES = {Paper 2.4, 38},
      ISSN = {1077-8926},
   MRCLASS = {05A18 (05A19 05E18)},
  MRNUMBER = {3512626},
       DOI = {10.37236/5681},
       URL = {https://doi.org/10.37236/5681},
}

@book {brouwer1974period,
    AUTHOR = {Brouwer, A. E. and Schrijver, A.},
     TITLE = {On the period of an operator, defined on antichains},
      NOTE = {Mathematisch Centrum Afdeling Zuivere Wiskunde ZW 24/74},
 PUBLISHER = {Mathematisch Centrum, Amsterdam},
      YEAR = {1974},
     PAGES = {i+13}
}

@article {cameron1995orbits,
    AUTHOR = {Cameron, P. J. and Fon-Der-Flaass, D. G.},
     TITLE = {Orbits of antichains revisited},
   JOURNAL = {European J. Combin.},
  FJOURNAL = {European Journal of Combinatorics},
    VOLUME = {16},
      YEAR = {1995},
    NUMBER = {6},
     PAGES = {545--554}
}

@misc{ceballos2024canonical,  
  author = {Cesar Ceballos},
  title = {A canonical realization of the alt $\nu$-associahedron},
  year = {2024},
  note = {\arxiv{2401.17204v1}}
}

@article {ceballos2024altnu,
    AUTHOR = {Ceballos, Cesar and Chenevi\`ere, Cl\'ement},
     TITLE = {On linear intervals in the alt {$\nu$}-{T}amari lattices},
   JOURNAL = {Comb. Theory},
  FJOURNAL = {Combinatorial Theory},
    VOLUME = {4},
      YEAR = {2024},
    NUMBER = {2},
     PAGES = {Paper No. 18, 31},
      ISSN = {2766-1334},
   MRCLASS = {06A07 (05A19 06B05)},
  MRNUMBER = {4807157},
MRREVIEWER = {Joel\ Berman},
}

@misc{ceballos2026snoncrossing,
  author = {Ceballos, Cesar and Matthias M\"{u}ller},
  title  = {$s$-noncrossing partitions},
  note   = {Work in progress},
  year   = {2026}
}

@article {ceballos2020nutamari,
    AUTHOR = {Ceballos, Cesar and Padrol, Arnau and Sarmiento, Camilo},
     TITLE = {The {$\nu$}-{T}amari lattice via {$\nu$}-trees,
              {$\nu$}-bracket vectors, and subword complexes},
   JOURNAL = {Electron. J. Combin.},
  FJOURNAL = {Electronic Journal of Combinatorics},
    VOLUME = {27},
      YEAR = {2020},
    NUMBER = {1},
     PAGES = {Paper No. 1.14, 31},
      ISSN = {1077-8926},
       DOI = {10.37236/8000},
       URL = {https://doi.org/10.37236/8000},
}

@article {chan2017expected,
    AUTHOR = {Chan, Melody and Haddadan, Shahrzad and Hopkins, Sam and Moci,
              Luca},
     TITLE = {The expected jaggedness of order ideals},
   JOURNAL = {Forum Math. Sigma},
  FJOURNAL = {Forum of Mathematics. Sigma},
    VOLUME = {5},
      YEAR = {2017},
     PAGES = {Paper No. e9, 27},
      ISSN = {2050-5094},
   MRCLASS = {05E10 (05A17 06A07 14Q05 60B15)},
  MRNUMBER = {3623577},
MRREVIEWER = {Russ\ Woodroofe},
       DOI = {10.1017/fms.2017.5},
       URL = {https://doi.org/10.1017/fms.2017.5},
}

@article {dao2022rowmotion,
    AUTHOR = {Dao, Quang Vu and Wellman, Julian and Yost-Wolff, Calvin and
              Zhang, Sylvester W.},
     TITLE = {Rowmotion orbits of trapezoid posets},
   JOURNAL = {Electron. J. Combin.},
  FJOURNAL = {Electronic Journal of Combinatorics},
    VOLUME = {29},
      YEAR = {2022},
    NUMBER = {2},
     PAGES = {Paper No. 2.29, 19},
      ISSN = {1077-8926},
   MRCLASS = {05E40 (06A07)},
  MRNUMBER = {4423786},
MRREVIEWER = {Toufik\ Mansour},
       DOI = {10.37236/9769},
       URL = {https://doi.org/10.37236/9769},
}

@article {dauvergne2022hidden,
    AUTHOR = {Dauvergne, Duncan},
     TITLE = {Hidden invariance of last passage percolation and directed
              polymers},
   JOURNAL = {Ann. Probab.},
  FJOURNAL = {The Annals of Probability},
    VOLUME = {50},
      YEAR = {2022},
    NUMBER = {1},
     PAGES = {18--60},
      ISSN = {0091-1798,2168-894X},
   MRCLASS = {60K35 (05A15)},
  MRNUMBER = {4385122},
       DOI = {10.1214/21-aop1527},
       URL = {https://doi.org/10.1214/21-aop1527},
}

@article {defant2023homomesy,
    AUTHOR = {Defant, Colin and Hopkins, Sam and Poznanovi\'c, Svetlana and
              Propp, James},
     TITLE = {Homomesy via toggleability statistics},
   JOURNAL = {Comb. Theory},
  FJOURNAL = {Combinatorial Theory},
    VOLUME = {3},
      YEAR = {2023},
    NUMBER = {2},
     PAGES = {Paper No. 14, 61},
      ISSN = {2766-1334},
   MRCLASS = {06A07 (05A30 05E18 52B05)},
  MRNUMBER = {4646095},
MRREVIEWER = {Claire\ Frechette},
       DOI = {10.5070/c63261992},
       URL = {https://doi.org/10.5070/c63261992},
}

@misc{defant2025echelonmotion,
    AUTHOR={Colin Defant and Yuhan Jiang and Rene Marczinzik and Adrien Segovia and David E Speyer and Hugh Thomas and Nathan Williams},
    TITLE={Rowmotion and Echelonmotion},
    YEAR={2025},
    HOWPUBLISHED={\arxiv{2507.18230}}
}

@article {defant2024tamari,
    AUTHOR = {Defant, Colin and Lin, James},
     TITLE = {Rowmotion on {$m$}-{T}amari and bi{C}ambrian lattices},
   JOURNAL = {Comb. Theory},
  FJOURNAL = {Combinatorial Theory},
    VOLUME = {4},
      YEAR = {2024},
    NUMBER = {1},
     PAGES = {Paper No. 15, 46},
      ISSN = {2766-1334},
   MRCLASS = {06B10 (05E18 06D75)},
  MRNUMBER = {4770594},
MRREVIEWER = {Konrad\ P.\ Pi\'oro},
}

@article {defant2023semidistrim,
    AUTHOR = {Defant, Colin and Williams, Nathan},
     TITLE = {Semidistrim lattices},
   JOURNAL = {Forum Math. Sigma},
  FJOURNAL = {Forum of Mathematics. Sigma},
    VOLUME = {11},
      YEAR = {2023},
     PAGES = {Paper No. e50, 35},
      ISSN = {2050-5094},
   MRCLASS = {06B05 (06B15 06D75)},
  MRNUMBER = {4603109},
MRREVIEWER = {Jinbo\ Yang},
       DOI = {10.1017/fms.2023.46},
       URL = {https://doi.org/10.1017/fms.2023.46},
}

@incollection {einstein2014piecewise,
    AUTHOR = {Einstein, David and Propp, James},
     TITLE = {Piecewise-linear and birational toggling},
 BOOKTITLE = {26th {I}nternational {C}onference on {F}ormal {P}ower {S}eries
              and {A}lgebraic {C}ombinatorics ({FPSAC} 2014)},
    SERIES = {Discrete Math. Theor. Comput. Sci. Proc.},
    VOLUME = {AT},
     PAGES = {513--524},
 PUBLISHER = {Assoc. Discrete Math. Theor. Comput. Sci., Nancy},
      YEAR = {2014},
   MRCLASS = {06A07 (05A19)},
  MRNUMBER = {3466399},
}

@article {einstein2021combinatorial,
    AUTHOR = {Einstein, David and Propp, James},
     TITLE = {Combinatorial, piecewise-linear, and birational homomesy for
              products of two chains},
   JOURNAL = {Algebr. Comb.},
  FJOURNAL = {Algebraic Combinatorics},
    VOLUME = {4},
      YEAR = {2021},
    NUMBER = {2},
     PAGES = {201--224},
      ISSN = {2589-5486},
   MRCLASS = {05E18 (06A07)},
  MRNUMBER = {4244370},
       DOI = {10.5802/alco.139},
       URL = {https://doi.org/10.5802/alco.139},
}

@misc{eu2026rowmotion,
    AUTHOR={Sen-Peng Eu and Vei-Cheng Hioe and Yi-Lin Lee},
    TITLE={Rowmotion on hook and two-row alt $\nu$-{T}amari lattices},
    YEAR={2026},
    HOWPUBLISHED={\arxiv{2605.29431}}
}

@article {garver2023minuscule,
    AUTHOR = {Garver, Alexander and Patrias, Rebecca and Thomas, Hugh},
     TITLE = {Minuscule reverse plane partitions via quiver representations},
   JOURNAL = {Selecta Math. (N.S.)},
  FJOURNAL = {Selecta Mathematica. New Series},
    VOLUME = {29},
      YEAR = {2023},
    NUMBER = {3},
     PAGES = {Paper No. 37, 48},
      ISSN = {1022-1824,1420-9020},
   MRCLASS = {16G20 (05E10)},
  MRNUMBER = {4581740},
MRREVIEWER = {Esther\ Banaian},
       DOI = {10.1007/s00029-023-00831-4},
       URL = {https://doi.org/10.1007/s00029-023-00831-4},
}

@article {grinberg2015birational,
    AUTHOR = {Grinberg, Darij and Roby, Tom},
     TITLE = {Iterative properties of birational rowmotion {II}: rectangles
              and triangles},
   JOURNAL = {Electron. J. Combin.},
  FJOURNAL = {Electronic Journal of Combinatorics},
    VOLUME = {22},
      YEAR = {2015},
    NUMBER = {3},
     PAGES = {Paper 3.40, 49},
      ISSN = {1077-8926},
   MRCLASS = {06A07 (05E99)},
  MRNUMBER = {3414186},
MRREVIEWER = {T.\ Kyle\ Petersen},
       DOI = {10.37236/4335},
       URL = {https://doi.org/10.37236/4335},
}

@article {hopkins2020cyclic,
    AUTHOR = {Hopkins, Sam},
     TITLE = {Cyclic sieving for plane partitions and symmetry},
   JOURNAL = {SIGMA Symmetry Integrability Geom. Methods Appl.},
  FJOURNAL = {SIGMA. Symmetry, Integrability and Geometry. Methods and
              Applications},
    VOLUME = {16},
      YEAR = {2020},
     PAGES = {Paper No. 130, 40},
      ISSN = {1815-0659},
   MRCLASS = {05E18 (05E10 17B10 17B37)},
  MRNUMBER = {4184618},
MRREVIEWER = {Vincenzo\ Chilla},
       DOI = {10.3842/SIGMA.2020.130},
       URL = {https://doi.org/10.3842/SIGMA.2020.130},
}

@article {hopkins2022minuscule,
    AUTHOR = {Hopkins, Sam},
     TITLE = {Minuscule doppelg\"angers, the coincidental down-degree
              expectations property, and rowmotion},
   JOURNAL = {Exp. Math.},
  FJOURNAL = {Experimental Mathematics},
    VOLUME = {31},
      YEAR = {2022},
    NUMBER = {3},
     PAGES = {946--974},
      ISSN = {1058-6458,1944-950X},
   MRCLASS = {06A07 (05E18 06A11)},
  MRNUMBER = {4477416},
MRREVIEWER = {Junyao\ Pan},
       DOI = {10.1080/10586458.2020.1731881},
       URL = {https://doi.org/10.1080/10586458.2020.1731881},
}

@incollection {hopkins2024order,
    AUTHOR = {Hopkins, Sam},
     TITLE = {Order polynomial product formulas and poset dynamics},
 BOOKTITLE = {Open problems in algebraic combinatorics},
    SERIES = {Proc. Sympos. Pure Math.},
    VOLUME = {110},
     PAGES = {135--157},
 PUBLISHER = {Amer. Math. Soc., Providence, RI},
      YEAR = {2024},
      ISBN = {[9781470473334]; [9781470477974]},
   MRCLASS = {06A07 (05A15 05A19 05E10 05E18)},
  MRNUMBER = {4780728},
MRREVIEWER = {Esther\ Banaian},
}

@misc{hopkins2026cyclic,
    AUTHOR={Sam Hopkins and Jesse Kim and Stephan Pfannerer},
    TITLE={Cyclic sieving for staircase plane partitions via crystals and electrical networks},
    YEAR={2026},
    HOWPUBLISHED={\arxiv{2607.14028}}
}

@article {iyama2022distributive,
    AUTHOR = {Iyama, Osamu and Marczinzik, Ren\'e},
     TITLE = {Distributive lattices and {A}uslander regular algebras},
   JOURNAL = {Adv. Math.},
  FJOURNAL = {Advances in Mathematics},
    VOLUME = {398},
      YEAR = {2022},
     PAGES = {Paper No. 108233, 27},
      ISSN = {0001-8708,1090-2082},
   MRCLASS = {16E10 (06D05 16G10)},
  MRNUMBER = {4379204},
MRREVIEWER = {Chao\ Zhang},
       DOI = {10.1016/j.aim.2022.108233},
       URL = {https://doi.org/10.1016/j.aim.2022.108233},
}

@article {johnson2024birational,
    AUTHOR = {Johnson, Joseph and Liu, Ricky Ini},
     TITLE = {Birational rowmotion and the octahedron recurrence},
   JOURNAL = {Algebr. Comb.},
  FJOURNAL = {Algebraic Combinatorics},
    VOLUME = {7},
      YEAR = {2024},
    NUMBER = {5},
     PAGES = {1453--1477},
      ISSN = {2589-5486},
   MRCLASS = {05E10 (06A07)},
  MRNUMBER = {4818780},
MRREVIEWER = {Nohra\ Hage},
       DOI = {10.5802/alco.385},
       URL = {https://doi.org/10.5802/alco.385},
}

@article {johnson2025piecewise,
    AUTHOR = {Johnson, Joseph and Liu, Ricky Ini},
     TITLE = {Piecewise-linear promotion and {RSK} in rectangles and moon
              polyominoes},
   JOURNAL = {Comb. Theory},
  FJOURNAL = {Combinatorial Theory},
    VOLUME = {5},
      YEAR = {2025},
    NUMBER = {3},
     PAGES = {Paper No. 9, 36},
      ISSN = {2766-1334},
   MRCLASS = {05E18 (05A05 05A19 52B05)},
  MRNUMBER = {4962123},
}

@misc{johnson2023trapezoid,
    AUTHOR={Joseph Johnson and Ricky Ini Liu},
    TITLE={Plane partitions and rowmotion on rectangular and trapezoidal posets},
    YEAR={2023},
    HOWPUBLISHED={\arxiv{2311.07133}}
}

@article{jonsson2005generalized,
    AUTHOR = {Jonsson, Jakob},
     TITLE = {Generalized triangulations and diagonal-free subsets of stack
              polyominoes},
   JOURNAL = {J. Combin. Theory Ser. A},
  FJOURNAL = {Journal of Combinatorial Theory. Series A},
    VOLUME = {112},
      YEAR = {2005},
    NUMBER = {1},
     PAGES = {117--142},
      ISSN = {0097-3165},
     CODEN = {JCBTA7},
   MRCLASS = {05A15 (05B50)},
  MRNUMBER = {2167478 (2006d:05011)},
MRREVIEWER = {Eric S. Egge},
       DOI = {10.1016/j.jcta.2005.01.009},
       FURL = {http://dx.doi.org/10.1016/j.jcta.2005.01.009},
}

@article{joseph2019antichain,
    AUTHOR = {Joseph, Michael},
     TITLE = {Antichain toggling and rowmotion},
   JOURNAL = {Electron. J. Combin.},
  FJOURNAL = {Electronic Journal of Combinatorics},
    VOLUME = {26},
      YEAR = {2019},
    NUMBER = {1},
     PAGES = {Paper No. 1.29, 43},
      ISSN = {1077-8926},
   MRCLASS = {05E18 (06A07 20B25)},
  MRNUMBER = {3919614},
       DOI = {10.37236/7454},
       URL = {https://doi.org/10.37236/7454},
}

@misc{klasz2025auslander,
    AUTHOR={Vikt\'{o}ria Kl\'{a}sz and  Rene Marczinzik and Hugh Thomas},
    TITLE={{A}uslander regular algebras and {C}oxeter matrices},
    YEAR={2025},
    HOWPUBLISHED={\arxiv{2501.09447}},
    NOTE={Forthcoming, \emph{Algebra Number Theory}}
}

@article{markowsky1992primes,
    AUTHOR = {Markowsky, George},
     TITLE = {Primes, irreducibles and extremal lattices},
   JOURNAL = {Order},
  FJOURNAL = {Order. A Journal on the Theory of Ordered Sets and its
              Applications},
    VOLUME = {9},
      YEAR = {1992},
    NUMBER = {3},
     PAGES = {265--290},
      ISSN = {0167-8094,1572-9273},
   MRCLASS = {06B05 (06D05)},
  MRNUMBER = {1211380},
       DOI = {10.1007/BF00383950},
       URL = {https://doi.org/10.1007/BF00383950},
}

@article {muhle2019core,
    AUTHOR = {M\"uhle, Henri},
     TITLE = {The core label order of a congruence-uniform lattice},
   JOURNAL = {Algebra Universalis},
  FJOURNAL = {Algebra Universalis},
    VOLUME = {80},
      YEAR = {2019},
    NUMBER = {1},
     PAGES = {Paper No. 10, 22},
      ISSN = {0002-5240,1420-8911},
   MRCLASS = {06B05 (06A07)},
  MRNUMBER = {3908324},
MRREVIEWER = {Keith\ A.\ Kearnes},
       DOI = {10.1007/s00012-019-0585-5},
       URL = {https://doi.org/10.1007/s00012-019-0585-5},
}

@misc{mueller2026combinatorial,
    AUTHOR={Matthias M\"{u}ller},
    TITLE={A combinatorial model for the canonical join complex of alt $\nu$-{T}amari lattices},
    YEAR={2026},
    HOWPUBLISHED={\arxiv{2605.13770}}
}

@article {panyushev2009orbits,
    AUTHOR = {Panyushev, Dmitri I.},
     TITLE = {On orbits of antichains of positive roots},
   JOURNAL = {European J. Combin.},
  FJOURNAL = {European Journal of Combinatorics},
    VOLUME = {30},
      YEAR = {2009},
    NUMBER = {2},
     PAGES = {586--594}
}

@article {preville2017nutamari,
    AUTHOR = {Pr\'{e}ville-Ratelle, Louis-Fran\c{c}ois and Viennot, Xavier},
     TITLE = {The enumeration of generalized {T}amari intervals},
   JOURNAL = {Trans. Amer. Math. Soc.},
  FJOURNAL = {Transactions of the American Mathematical Society},
    VOLUME = {369},
      YEAR = {2017},
    NUMBER = {7},
     PAGES = {5219--5239},
      ISSN = {0002-9947,1088-6850},
       DOI = {10.1090/tran/7004},
       URL = {https://doi.org/10.1090/tran/7004},
       NOTE = {Correct title, ``An extension of Tamari lattices''}
}

@article {propp2015homomesy,
    AUTHOR = {Propp, James and Roby, Tom},
     TITLE = {Homomesy in products of two chains},
   JOURNAL = {Electron. J. Combin.},
  FJOURNAL = {Electronic Journal of Combinatorics},
    VOLUME = {22},
      YEAR = {2015},
    NUMBER = {3},
     PAGES = {Paper 3.4, 29},
      ISSN = {1077-8926},
   MRCLASS = {05E18 (06A07 06A11)},
  MRNUMBER = {3367853},
MRREVIEWER = {Konrad\ P.\ Pi\'oro},
       DOI = {10.37236/3579},
       URL = {https://doi.org/10.37236/3579},
}

@article {reading2011noncrossing,
    AUTHOR = {Reading, Nathan},
     TITLE = {Noncrossing partitions and the shard intersection order},
   JOURNAL = {J. Algebraic Combin.},
  FJOURNAL = {Journal of Algebraic Combinatorics. An International Journal},
    VOLUME = {33},
      YEAR = {2011},
    NUMBER = {4},
     PAGES = {483--530},
      ISSN = {0925-9899,1572-9192},
   MRCLASS = {06A07 (05A18 05E15 20F55)},
  MRNUMBER = {2781960},
       DOI = {10.1007/s10801-010-0255-3},
       URL = {https://doi.org/10.1007/s10801-010-0255-3},
}

@article {reiner2004cyclic,
    AUTHOR = {Reiner, Vic and Stanton, Dennis and White, Dennis},
     TITLE = {The cyclic sieving phenomenon},
   JOURNAL = {J. Combin. Theory Ser. A},
  FJOURNAL = {Journal of Combinatorial Theory. Series A},
    VOLUME = {108},
      YEAR = {2004},
    NUMBER = {1},
     PAGES = {17--50},
      ISSN = {0097-3165,1096-0899},
   MRCLASS = {05A15 (05A30 05E99 20B05 20F55)},
  MRNUMBER = {2087303},
MRREVIEWER = {Timothy\ Y.\ Chow},
       DOI = {10.1016/j.jcta.2004.04.009},
       URL = {https://doi.org/10.1016/j.jcta.2004.04.009},
}

@incollection {roby2016dac,
    AUTHOR = {Roby, Tom},
     TITLE = {Dynamical algebraic combinatorics and the homomesy phenomenon},
 BOOKTITLE = {Recent trends in combinatorics},
    SERIES = {IMA Vol. Math. Appl.},
    VOLUME = {159},
     PAGES = {619--652},
 PUBLISHER = {Springer, [Cham]},
      YEAR = {2016},
      ISBN = {978-3-319-24296-5; 978-3-319-24298-9},
   MRCLASS = {05-02 (05E18 06A07)},
  MRNUMBER = {3526426},
       DOI = {10.1007/978-3-319-24298-9\_25},
       URL = {https://doi.org/10.1007/978-3-319-24298-9_25},
}

@article {rubey2011increasing,
    AUTHOR = {Rubey, Martin},
     TITLE = {Increasing and decreasing sequences in fillings of moon
              polyominoes},
   JOURNAL = {Adv. in Appl. Math.},
  FJOURNAL = {Advances in Applied Mathematics},
    VOLUME = {47},
      YEAR = {2011},
    NUMBER = {1},
     PAGES = {57--87},
      ISSN = {0196-8858,1090-2074},
   MRCLASS = {05E10 (05A19 05B50)},
  MRNUMBER = {2799611},
       DOI = {10.1016/j.aam.2009.11.013},
       URL = {https://doi.org/10.1016/j.aam.2009.11.013},
}

@article {rubey2012maximal,
    AUTHOR = {Rubey, Martin},
     TITLE = {Maximal {$0$}-{$1$}-fillings of moon polyominoes with
              restricted chain lengths and rc-graphs},
   JOURNAL = {Adv. in Appl. Math.},
  FJOURNAL = {Advances in Applied Mathematics},
    VOLUME = {48},
      YEAR = {2012},
    NUMBER = {2},
     PAGES = {290--305},
      ISSN = {0196-8858,1090-2074},
   MRCLASS = {05B50 (05A19 05E05 06A07)},
  MRNUMBER = {2873877},
MRREVIEWER = {Svetlana\ Poznanovi\'c},
       DOI = {10.1016/j.aam.2011.05.005},
       URL = {https://doi.org/10.1016/j.aam.2011.05.005},
}

@article {rush2013orbits,
    AUTHOR = {Rush, David B. and Shi, XiaoLin},
     TITLE = {On orbits of order ideals of minuscule posets},
   JOURNAL = {J. Algebraic Combin.},
  FJOURNAL = {Journal of Algebraic Combinatorics. An International Journal},
    VOLUME = {37},
      YEAR = {2013},
    NUMBER = {3},
     PAGES = {545--569},
      ISSN = {0925-9899,1572-9192},
   MRCLASS = {06A06 (05E15)},
  MRNUMBER = {3035516},
       DOI = {10.1007/s10801-012-0380-2},
       URL = {https://doi.org/10.1007/s10801-012-0380-2},
}

@incollection {sagan2011cyclic,
    AUTHOR = {Sagan, Bruce E.},
     TITLE = {The cyclic sieving phenomenon: a survey},
 BOOKTITLE = {Surveys in combinatorics 2011},
    SERIES = {London Math. Soc. Lecture Note Ser.},
    VOLUME = {392},
     PAGES = {183--233},
 PUBLISHER = {Cambridge Univ. Press, Cambridge},
      YEAR = {2011},
      ISBN = {978-1-107-60109-3},
   MRCLASS = {05-02 (05A05 05E10 05E15)},
  MRNUMBER = {2866734},
MRREVIEWER = {T.\ Kyle\ Petersen},
}

@manual{Sage,
  Key          = {SageMath},
  Author       = {{Sage Developers}},
  Title        = {{S}ageMath, the {S}age {M}athematics {S}oftware {S}ystem ({V}ersion 10.6)},
  note         = {\url{https://www.sagemath.org}},
  Year         = {2025},
}

@book {stanley2015catalan,
    AUTHOR = {Stanley, Richard P.},
     TITLE = {Catalan numbers},
 PUBLISHER = {Cambridge University Press, New York},
      YEAR = {2015},
     PAGES = {viii+215},
      ISBN = {978-1-107-42774-7; 978-1-107-07509-2},
   MRCLASS = {05-01 (01A05 11B75 11B83)},
  MRNUMBER = {3467982},
MRREVIEWER = {David\ Callan},
       DOI = {10.1017/CBO9781139871495},
       URL = {https://doi.org/10.1017/CBO9781139871495},
}

@article{stanley1986two,
    AUTHOR = {Stanley, Richard P.},
     TITLE = {Two poset polytopes},
   JOURNAL = {Discrete Comput. Geom.},
  FJOURNAL = {Discrete \& Computational Geometry. An International Journal
              of Mathematics and Computer Science},
    VOLUME = {1},
      YEAR = {1986},
    NUMBER = {1},
     PAGES = {9--23},
      ISSN = {0179-5376,1432-0444},
   MRCLASS = {52A25 (52A40)},
  MRNUMBER = {824105},
MRREVIEWER = {P.\ McMullen},
       DOI = {10.1007/BF02187680},
       URL = {https://doi.org/10.1007/BF02187680},
}

@book {stanley2012ec1,
    AUTHOR = {Stanley, Richard P.},
     TITLE = {Enumerative combinatorics. {V}olume 1},
    SERIES = {Cambridge Studies in Advanced Mathematics},
    VOLUME = {49},
   EDITION = {Second},
 PUBLISHER = {Cambridge University Press, Cambridge},
      YEAR = {2012},
     PAGES = {xiv+626},
      ISBN = {978-1-107-60262-5},
   MRCLASS = {05-02 (05A15 06-02)},
  MRNUMBER = {2868112},
}

@article {striker2015tog,
    AUTHOR = {Striker, Jessica},
     TITLE = {The toggle group, homomesy, and the {R}azumov-{S}troganov
              correspondence},
   JOURNAL = {Electron. J. Combin.},
  FJOURNAL = {Electronic Journal of Combinatorics},
    VOLUME = {22},
      YEAR = {2015},
    NUMBER = {2},
     PAGES = {Paper 2.57, 17},
      ISSN = {1077-8926},
   MRCLASS = {05E18 (06A06 82B05)},
  MRNUMBER = {3367300},
       DOI = {10.37236/5158},
       URL = {https://doi.org/10.37236/5158},
}

@article {striker2017dac,
    AUTHOR = {Striker, Jessica},
     TITLE = {Dynamical algebraic combinatorics: promotion, rowmotion, and
              resonance},
   JOURNAL = {Notices Amer. Math. Soc.},
  FJOURNAL = {Notices of the American Mathematical Society},
    VOLUME = {64},
      YEAR = {2017},
    NUMBER = {6},
     PAGES = {543--549},
      ISSN = {0002-9920,1088-9477},
   MRCLASS = {05E99 (05E10)},
  MRNUMBER = {3585535},
       DOI = {10.1090/noti1539},
       URL = {https://doi.org/10.1090/noti1539},
}

@article {striker2012rowmotion,
    AUTHOR = {Striker, Jessica and Williams, Nathan},
     TITLE = {Promotion and rowmotion},
   JOURNAL = {European J. Combin.},
  FJOURNAL = {European Journal of Combinatorics},
    VOLUME = {33},
      YEAR = {2012},
    NUMBER = {8},
     PAGES = {1919--1942},
      ISSN = {0195-6698,1095-9971},
   MRCLASS = {06A07 (05A19)},
  MRNUMBER = {2950491},
MRREVIEWER = {Luca\ Ferrari},
       DOI = {10.1016/j.ejc.2012.05.003},
       URL = {https://doi.org/10.1016/j.ejc.2012.05.003},
}

@article {stump2025cataland,
    AUTHOR = {Stump, Christian and Thomas, Hugh and Williams, Nathan},
     TITLE = {Cataland: why the {F}u\ss?},
   JOURNAL = {Mem. Amer. Math. Soc.},
  FJOURNAL = {Memoirs of the American Mathematical Society},
    VOLUME = {305},
      YEAR = {2025},
    NUMBER = {1535},
     PAGES = {vii+143},
      ISSN = {0065-9266,1947-6221},
      ISBN = {978-1-4704-6314-4; 978-1-4704-8037-0},
   MRCLASS = {20F55 (05E10 16G10 20F36)},
  MRNUMBER = {4853267},
MRREVIEWER = {Himmet\ Can},
       DOI = {10.1090/memo/1535},
       URL = {https://doi.org/10.1090/memo/1535},
}

@article {thomas2019independence,
    AUTHOR = {Thomas, Hugh and Williams, Nathan},
     TITLE = {Independence posets},
   JOURNAL = {J. Comb.},
  FJOURNAL = {Journal of Combinatorics},
    VOLUME = {10},
      YEAR = {2019},
    NUMBER = {3},
     PAGES = {545--578},
      ISSN = {2156-3527,2150-959X},
   MRCLASS = {05C69 (06D75)},
  MRNUMBER = {3960513},
MRREVIEWER = {Konrad\ P.\ Pi\'oro},
       DOI = {10.4310/JOC.2019.v10.n3.a5},
       URL = {https://doi.org/10.4310/JOC.2019.v10.n3.a5},
}

@article {thomas2019rowmotion,
    AUTHOR = {Thomas, Hugh and Williams, Nathan},
     TITLE = {Rowmotion in slow motion},
   JOURNAL = {Proc. Lond. Math. Soc. (3)},
  FJOURNAL = {Proceedings of the London Mathematical Society. Third Series},
    VOLUME = {119},
      YEAR = {2019},
    NUMBER = {5},
     PAGES = {1149--1178},
      ISSN = {0024-6115,1460-244X},
   MRCLASS = {06D75 (05E15)},
  MRNUMBER = {3968720},
MRREVIEWER = {Konrad\ P.\ Pi\'oro},
       DOI = {10.1112/plms.12251},
       URL = {https://doi.org/10.1112/plms.12251},
}

@misc{williams2020personal,
    AUTHOR={Nathan Williams},
    TITLE={Piecewise-linear extension of independence poset rowmotion},
    YEAR={2020},
    HOWPUBLISHED={Personal communication}
}

\end{document}